\documentclass[11pt,a4paper,twoside]{article}
\usepackage{bm, amsmath, amssymb, amsthm} 

\def\dt{\partial_t}
\def\T{T^\sharp}

\def\tT{{\tilde T}^\sharp}

\def\tA{{\tilde A}}

\def\th{{\tilde h}}

\newcommand\mD{{\cal D}}
\newcommand\mN{{\cal N}}

\def\<{\langle}
\def\>{\rangle}
\def\RR{\mathbb{R}}

\def\eps{\varepsilon}

\newcommand\mA{A^{\rm mix}}
\newcommand\Ts{T^{\sharp}}
\newcommand\mTs{T^{\sharp {\rm mix}}}
\newcommand\mh{h^{\rm mix}}
\newcommand\mT{T^{\rm mix}}

\newcommand\tr{\operatorname{Tr}}
\newcommand\Div{\operatorname{div}}
\newcommand\id{\operatorname{id}}

\def\vol{\operatorname{vol}}

\def\eq{\hspace*{-1.5mm}&=&\hspace*{-1.2mm}}

\newtheorem{corollary}{Corollary}
\newtheorem{definition}{Definition}
\newtheorem{example}{Example}
\newtheorem{remark}{Remark}
\newtheorem{lemma}{Lemma}
\newtheorem{proposition}{Proposition}

\author
      {
      	Vladimir Rovenski\footnote{Department of Mathematics, University of Haifa, Mount Carmel, Haifa, 3498838, Israel.
      		\newline e-mail: {\tt vrovenski@univ.haifa.ac.il}}\ \
      	\ and \
      	Tomasz Zawadzki
      	\footnote{Uniwersytet \L\'{o}dzki, ul. Banacha 22, 90-238 \L\'{o}d\'{z}, Poland.
      		\newline e-mail: {\tt tomasz.zawadzki@wmii.uni.lodz.pl}}
      }

\title{Variations of the mutual curvature of two orthogonal non-complementary distributions}

\begin{document}

\date{}

\maketitle

\begin{abstract}
On a smooth manifold with distributions ${\cal D}_1$ and ${\cal D}_2$ having trivial intersection, we consider the integral of their mutual curvature, as a functional of Riemannian metrics that make the distributions orthogonal.
The~mutual curvature is defined as the sum of sectional curvatures of planes spanned by all pairs of vectors from an orthonormal basis,
such that one vector of the pair belongs to ${\cal D}_1$ and the second vector belongs to ${\cal D}_2$. As such, it interpolates between the sectional curvature of a plane field (if both distributions are one-dimensional), and the mixed scalar curvature of a Riemannian almost product structure (if both distributions together span the tangent bundle). We derive Euler-Lagrange equations for the functional, formulated in terms of extrinsic geometry of distributions, i.e., their second fundamental forms and integrability tensors. We give examples of critical metrics for distributions defined on domains of Riemannian submersions, twisted products and $f$-$K$-contact manifolds.


\end{abstract}


\section*{Introduction}

Distributions on a smooth manifold (i.e., subbundles of the tangent bundle) appear in various topics of mathematics and theoretical physics,
e.g., \cite{bf,blair}, and are used to build up notions of integrability, and specifically of a foliation (i.e., a partition of a manifold by submanifolds of a constant dimension).
In \cite{rz-1,rz-2} we were guided by the following approach to variational problems in the geometry of distributions
(developed in \cite{RWa-1} for codimension one foliations):
\textit{given a distribution $\mD$ on a smooth manifold and a property $Q$ of
the second fundamental form or the curvature associated with a distribution,
find critical Riemannian metrics for the integral $Q$ in the class of general
or {adapted} {\rm (}i.e., with fixed orthogonal complement to $\mD)$ variations of~metrics.}
Many canonical metrics are critical points of some variational problems, e.g., Einstein metrics are critical points of the integral of the scalar curvature, and their classification 
is a deep problem~\cite{besse}.

In the presence of an additional structure (e.g., a contact structure~\cite{blair},
an almost product structure~\cite{g1967}, or a foliation~\cite{bf}),
one can consider an analogue of Einstein-Hilbert action consistent with this structure
and, instead of the scalar curvature, 
use another suitable curvature invariant.
This approach was implemented in our previous papers, e.g., \cite{rz-1,rz-2},
where the integrated scalar curvature of a smooth manifold $M$ with a Riemannian metric $g$ was replaced by the integrated \textit{mixed scalar curvature} ${S}_{\,\mD,\mD^\bot}$
of an almost product structure $(\mD,\mD^\bot)$.
The mixed scalar curvature ${S}_{\,\mD,\mD^\bot}$ is an averaged sum of sectional curvatures of the planes that non-trivially intersect
with the distribution $\mD$ and its orthogonal complement $\mD^\bot$, see \cite{wa1}.

A natural generalization of the above approach is to consider curvatures defined by a pair of orthogonal, but not necessarily complementary distributions $\mD_1$ and $\mD_2$. Then, their sum $\mD$ is a distribution on $M$ with orthogonal decomposition:
\begin{equation}\label{Eq-mD-2}
 \mD=\mD_1\oplus\mD_2 ,
\end{equation}
which can be regarded 
as an almost product structure
associated with a foliated manifold (if $\mD$ is everywhere integrable) or with a non-holonomic manifold (if $\mD$ is nowhere integrable), which is the central object of sub-Riemannian geometry, e.g.,~\cite{bf}.
In \cite{r6,VRAGAG}, some integral formulae were proved for a scalar curvature-type invariant of such $(M, \mD)$.

In this article, we consider a manifold of dimension $\dim M >2$, with two distributions $\mD_1$ and $\mD_2$ having trivial intersection, and investigate critical points of their integrated mutual curvature, among Riemannian metrics that make $\mD_1$ and $\mD_2$ orthogonal. For local orthonormal bases $\{ E_a \}$
 and $\{ {\cal E}_i \}$
 of orthogonal distributions $\mD_1$ and $\mD_2$, respectively,
the \textit{mutual curvature} is given~by
\begin{equation} \label{mc}
{S}_{\,\mD_1,\mD_2} = \sum\nolimits_{\,a} \sum\nolimits_{\,i} g(R(E_a,{\cal E}_i){\cal E}_i,E_a),
\end{equation}
where $R$
is the curvature tensor of the Levi-Civita connection $\nabla$ on $(M,g)$.
For $\dim \mD_1 = \dim \mD_2 =1$, the mutual curvature ${S}_{\,\mD_1,\mD_2}$ reduces to the sectional curvature of the plane field~$\mD$.
On the other hand, if 
$\mD = TM$, then ${S}_{\,\mD_1,\mD_2}$ reduces to the mixed scalar curvature of the almost-product structure $(\mD_1,\mD_2)$.
The relation of ${S}_{\,\mD_1,\mD_2}$ to sectional curvature means that many curvatures 
can be expressed by linear combinations of 
formulas \eqref{mc} for suitably chosen distributions.
On the other hand, with an appropriate choice of some auxiliary distributions, ${S}_{\,\mD_1,\mD_2}$ can be expressed in terms of their mixed scalar curvatures, which will allow us to develop variational formulas using results from~\cite{rz-2}.

We~introduce and study the following action:
\begin{equation}\label{Eq-Smix}
 J_{\mD_1,\mD_2}: g\mapsto\int_{M} {S}_{\,\mD_1,\mD_2}\,{\rm d}\vol_g,
\end{equation}
where
metrics $g$ satisfy
$g(X,Y)=0$ for all $X \in \mD_1,\ Y \in \mD_2$.
If $M$ is non-compact, we integrate in \eqref{Eq-Smix}
over an arbitrary relatively compact domain $\Omega\subset M$, which contains supports of variations of~$g$.
We derive the Euler-Lagrange equations for \eqref{Eq-Smix},
for variations of metric keeping orthogonality of the two given
distributions $\mD_1$ and $\mD_2$. These equations
are formulated using notions of the extrinsic geometry, i.e., the second fundamental forms and  integrability tensors of
$\mD_1$, $\mD_2$, $\mD$ and their orthogonal complements, and correspond to variations of metric along those distributions.

We present examples of critical points of \eqref{Eq-Smix}
for the following cases: one-dimensional distributions, twisted products, Riemannian submersions and $f$-$K$-contact manifolds.
The~first case gives a new perspective on sectional curvature from the point of view of extrinsic geometry of distributions.
In the second case, we examine a simple but non-trivial property of distributions to be mixed integrable and mixed totally geodesic, arising on twisted products of manifolds; and use examples obtained in \cite{rz-2} to find critical points of action \eqref{Eq-Smix}.
The remaining cases deal with
non-integrable and not mixed integrable distributions
on domains of Riemannian submersions,
and manifolds with $f$-$K$-structures, e.g., \cite{blairfKcontact,Goertsches-2}, and
$K$-contact structures~\cite{blair}.

The article consists of an Introduction and three sections.
Section~\ref{sec:prel} contains necessary definitions,
the properties of the mutual curvature of distributions
$\mD_1$ and $\mD_2$ and new integral formula.
In~Section~\ref{sec:var}, we derive Euler-Lagrange equations of the action \eqref{Eq-Smix} on the set of
metrics keeping $\mD_1$ and $\mD_2$ orthogonal, the three subsections there  correspond to particular, geometrically meaningful, variations of metric. In~Section~\ref{sec:examples} (consisting of five subsections), we give examples of critical metrics of the action \eqref{Eq-Smix} on manifolds with certain structures, naturally defining pairs of distributions.

\section{Preliminaries}
\label{sec:prel}

Let $\mD_1, \mD_2$ be two orthogonal distributions on a Riemannian manifold $(M,g)$,
and let $\mD_3=\mD^\bot$ be the distribution orthogonal to $\mD=\mD_1\oplus\mD_2$, see \eqref{Eq-mD-2}, i.e., $\mD_1 , \mD_2 , \mD_3$ are three pairwise orthogonal distributions on $(M,g)$, together decomposing the tangent bundle. For $\mD_i$, let $\mD_i^\perp$ denote its orthogonal complement, $i=1,2,3$. Let $P_m:TM\to\mD_m\ (m=1,2,3)$ be the orthoprojector onto $\mD_m$
and $P_m^\perp=\id_{\,TM}-P_m$ be the orthoprojector onto $\mD_m^\perp$, hence, $P_1^\perp = P_2 + P_3$, etc.
We~will use orthonormal vectors from a local frame adapted to decomposition $TM = \mD_1 \oplus \mD_2 \oplus \mD_3$, our convention is that
\[
 E_a , E_b \in \mD_1,\quad
 {\cal E}_i, {\cal E}_j \in \mD_2,\quad
 e_\mu , e_\nu \in \mD_3.
\]
We do not write ranges of indices $a,i,\mu$, assuming that in all sums they go through entire local bases of $\mD_1, \mD_2$ or $\mD_3$, respectively. All quantities defined below using an adapted frame do not depend on the choice of this~frame.
On rare occasions we will use a full local orthonormal frame in $TM$, the vectors of which will be denoted by~$\xi_k$.

The set of vector fields on $M$ will be denoted by ${\mathfrak X}_M$, the sets of vector fields with values in particular distribution, e.g., $\mD_1$, will be denoted analogously, i.e., by ${\mathfrak X}_{\mD_1}$. Distribution $\mD_m$ at a point $x \in M$ will be denoted by $\mD_{m_x}$. In all statements formulated for a pair of vectors $X,Y \in TM$, we assume that those vectors belong to the same tangent space.

A plane in $TM$ spanned by vectors $X\in{\mD}_1$ and $Y\in{\mD}_2$ is called~\textit{mi\-xed}. 
Thus, the mutual curvature \eqref{mc} is the sum of sectional curvatures of all mixed planes defined by an orthonormal frame.
Since $S_{\,\mD_i, \mD_j} = S_{\,\mD_j, \mD_i}$ and $\mD_i^\perp=\bigoplus_{j\ne i}\mD_j$, we have
 $S_{\,\mD_1 , \mD_1^\perp} = S_{\,\mD_1 , \mD_2} + S_{\,\mD_1 , \mD_3}$
 and
 $S_{\,\mD_2 , \mD_2^\perp} = S_{\,\mD_1 , \mD_2} + S_{\,\mD_2 , \mD_3}$.
Hence, using pairs of complementary orthogonal distributions, we represent the mutual curvature as
\begin{eqnarray}\label{SD1D2}
 2\,S_{\,\mD_1 , \mD_2}
\eq S_{\,\mD_1 , \mD_1^\perp} + S_{\,\mD_2 , \mD_2^\perp} - S_{\,\mD_3 , \mD_3^\perp}.
\end{eqnarray}

For all $X,Y \in TM$ and $m=1,2,3$, we set the following:
\begin{eqnarray*}
{\tilde h}_m (X,Y) = \frac{1}{2} P_m ( \nabla_{ P^\perp_m X } P^\perp_m Y {+}\nabla_{ P^\perp_m Y } P^\perp_m X ),\quad
 h_m (X,Y) = \frac{1}{2} P_m^\perp ( \nabla_{ P_m X } P_m Y {+}\nabla_{ P_m Y } P_m X ),\\
{\tilde T}_m (X,Y) = \frac{1}{2} P_m ( \nabla_{ P^\perp_m X } P^\perp_m Y {-} \nabla_{ P^\perp_m Y } P^\perp_m X ),\quad
 T_m (X,Y) = \frac{1}{2} P_m^\perp ( \nabla_{ P_m X } P_m Y {-} \nabla_{ P_m Y } P_m X ) ,
\end{eqnarray*}
for the second fundamental tensors and the integrability tensors of the distributions.

The ``musical" isomorphisms $\sharp$ and $\flat$ will be used for rank one and symmetric rank 2 tensors.
For~example, if $\omega$ is a 1-form and $X,Y\in {\mathfrak X}_M$ then
$\omega(Y)=g( \omega^\sharp,Y )$ and $X^\flat(Y) = g( X,Y )$. In particular, for $Z_1, Z_2 \in T_xM$ we have for all $X,Y \in T_xM$: $ Z_1^\flat \otimes Z_2^\flat (X,Y) = g(Z_1, X) g(Z_2 ,Y)$ and ${\rm Sym}(Z_1^\flat \otimes Z_2^\flat) (X,Y) = \frac{1}{2} ( g(Z_1, X) g(Z_2 ,Y) +  g(Z_1, Y) g(Z_2 ,X) )$.
The inner product of tensors defined by the metric $g$ will be denoted by $\< \cdot, \cdot \>$, we have, e.g., for $(0,2)$-tensors $Q_1,Q_2$ and $(1,2)$-tensors $S_1, S_2$:
\begin{eqnarray*}
\< Q_1 , Q_2 \> = \sum\nolimits_{\,k,m} Q_1 (\xi_k , \xi_m )  Q_2 (\xi_k , \xi_m ), \quad
\< S_1 , S_2 \> = \sum\nolimits_{\,k,m} g( S_1 (\xi_k , \xi_m ) ,  S_2 (\xi_k , \xi_m )) ,
\end{eqnarray*}
where $\xi_k$ is an orthonormal frame of $TM$.
For a $(1,2)$-tensor $Q$ and a vector field $Z$ we define a $(0,2)$-tensor $\< Q, Z \>$ by the formula
\[
\< Q, Z \> (X,Y) = g(Q(X,Y),Z),\quad X,Y \in TM.
\]
The shape operator $A_{i,Z}$ and the operator $T_{i,Z}^\sharp$ of ${\mD}_i$, with $Z\in{\mD}_i^\bot$, are defined~by
\[
g( A_{i,Z} X,Y )= g( h_i(X,Y),Z ),\quad
g( T_{i,Z}^\sharp X,Y )= g( T_i(X,Y),Z ), \quad X,Y \in {\mD}_i .
\]
Let $H_m$ be the mean curvature of $\mD_m\ (m=1,2,3)$, i.e., $H_m = \tr_g h_m$; and let ${\tilde H}_m$ be the mean curvature of $\mD_m^\perp$, i.e., ${\tilde H}_m = \tr_g {\tilde h}_m$. We have $H_m \in \mathfrak{X}_{ \mD_m^\perp}$ and ${\tilde H}_m \in \mathfrak{X}_{ \mD_m}$, in particular,
\begin{eqnarray*}
 H_1 = \sum\nolimits_{\,a} h_1 (E_a , E_a),\quad
{\tilde H}_1 = \sum\nolimits_{\,i} {\tilde h}_1 ({\cal E}_i , {\cal E}_i) + \sum\nolimits_{\,\mu} {\tilde h}_1 (e_\mu , e_\mu) .
\end{eqnarray*}
These tensors for various $\mD_i$ are related, e.g., it follows that ${\tilde H}_1 = P_1 H_2 + P_1 H_3$ 
and
for all $X,Y \in TM$ we have
\begin{equation*}
 h_1(X,Y)
 = {\tilde h}_2 (P_1 X , P_1 Y) + {\tilde h}_3 (P_1 X , P_1 Y).
\end{equation*}
A distribution ${\mD}_m$ is called integrable if $T_m=0$,
and ${\mD}_m$ is called {totally umbilical}, {harmonic}, or {totally geodesic}, if ${h}_m=({H}_m/n_m)\,g,\ {H}_m =0$, or ${h}_m=0$, respectively, where $n_m=\dim\mD_m$.

 The {divergence} of a vector~field $Y$ on $M$ is defined by
 $\Div Y =\sum\nolimits_{\,k}g(\nabla_{\xi_k} Y, \,\xi_k)$,
and the divergence of any (1,s)-tensor field $S$ is defined by
 $(\Div S)(X_1, \ldots , X_s) = \!\sum\nolimits_{\,k} g( (\nabla_{\xi_k} S) (X_1, \ldots, X_s), \xi_k )$, for all $X_1, \ldots, X_s \in TM$,
where $\{\xi_k \}$ is a local orthonormal frame of $TM$.

We note that \eqref{SD1D2} can give rise to an integral formula, as for each $\mD_m$ on a Riemannian manifold $(M,g)$
we have the following formula, see \cite{wa1}:
\begin{equation}\label{E-PW}
 S_{\,\mD_m , \mD_m^\perp} = \Div ( H_m + {\tilde H}_m)
 + \| H_m \|^2 + \| {\tilde H}_m \|^2 + \| T_m \|^2 + \| {\tilde T}_m \|^2 - \| h_m \|^2 - \| {\tilde h}_m \|^2 .
\end{equation}

\begin{proposition}
Let $\mD_1, \mD_2, \mD_3$ be pairwise orthogonal distributions on $(M,g)$ such that $TM = \mD_1 \oplus \mD_2 \oplus \mD_3$.
Set $\| {\tilde h}_1|_{\mD_2 \times \mD_3}\|^2 =\sum\nolimits_{\,i, \mu}\| {\tilde h}_1 ({\cal E}_i , e_\mu ) \|^2$, etc.
Then
\begin{eqnarray}\label{E-new2}
 && S_{\,\mD_1 , \mD_2} = \Div( P_2 H_1 + P_1 H_2 ) + \| P_2 H_1 \|^2 + \| P_1 H_2 \|^2 \nonumber \\
&& +\,\| {\tilde h}_3 |_{\mD_1 \times \mD_2 } \|^2
-\| {\tilde h}_1 |_{\mD_2 \times \mD_3 } \|^2 -\| {\tilde h}_2 |_{\mD_1 \times \mD_3 } \|^2 -\| P_1 h_2 \|^2 - \| P_2 h_1 \|^2 \nonumber \\
&& -\,\| {\tilde T}_3 |_{\mD_1 \times \mD_2 } \|^2
 +\| {\tilde T}_1 |_{\mD_2 \times \mD_3 } \|^2 + \| {\tilde T}_2 |_{\mD_1 \times \mD_3 } \|^2  +\| P_2 T_1 \|^2 + \|P_1 T_2 \|^2 .
\end{eqnarray}
If $M$ is a closed manifold, then the following integral formula holds:
\begin{eqnarray}\label{E-new-IF}
\nonumber
&& \int_M \big( S_{\,\mD_1 , \mD_2} -  \| P_2 H_1 \|^2 - \| P_1 H_2 \|^2
- \| {\tilde h}_3 |_{\mD_1 \times \mD_2 } \|^2 + \| {\tilde T}_3 |_{\mD_1 \times \mD_2 } \|^2 \\
\nonumber
&&
+\,\| {\tilde h}_1 |_{\mD_2 \times \mD_3 } \|^2 + \| {\tilde h}_2 |_{\mD_1 \times \mD_3 } \|^2
- \| {\tilde T}_1 |_{\mD_2 \times \mD_3 } \|^2 - \| {\tilde T}_2 |_{\mD_1 \times \mD_3 } \|^2 \\
&& -\,\| P_2 h_1 \|^2 -\|P_1 h_2 \|^2 - \| P_2 T_1 \|^2 - \|P_1 T_2 \|^2
\big) \, {\rm d} \vol_g =0 .
\end{eqnarray}
\end{proposition}

\begin{proof} Using \eqref{E-PW}, we can rewrite \eqref{SD1D2} as
\begin{eqnarray}\label{E-new1}
 2\,S_{\,\mD_1, \mD_2} \eq \| H_1 \|^2 +\| {\tilde H}_1 \|^2 +\| T_1 \|^2 +\| {\tilde T}_1 \|^2 -\| h_1 \|^2 -\| {\tilde h}_1 \|^2
 +\Div( H_1 + {\tilde H}_1 ) \nonumber \\
&& +\,\| H_2 \|^2 + \| {\tilde H}_2 \|^2 + \| T_2 \|^2 + \| {\tilde T}_2 \|^2 - \| h_2 \|^2 - \| {\tilde h}_2 \|^2
 +\Div  ( H_2 + {\tilde H}_2 ) \nonumber \\
&& -\,\| H_3 \|^2 - \| {\tilde H}_3 \|^2 - \| T_3 \|^2 - \| {\tilde T}_3 \|^2 + \| h_3 \|^2 + \| {\tilde h}_3 \|^2
 -\Div  ( H_3 + {\tilde H}_3 ).
\end{eqnarray}
Using ${\tilde H}_3 = P_3 ( H_1 + H_2 )$, ${\tilde H}_1 = P_1 (H_2 + H_3)$ and
${\tilde H}_2 = P_2 (H_1 + H_3)$, we get for terms with~``$H$":
\begin{eqnarray*}
&& \| H_1 \|^2 + \| {\tilde H}_1 \|^2 + \| H_2 \|^2 + \| {\tilde H}_2 \|^2  - \| H_3 \|^2 - \| {\tilde H}_3 \|^2
= 2\,\| P_2 H_1 \|^2 + 2\, \| P_1 H_2 \|^2 , \\
&& \Div ( H_1 + {\tilde H}_1 )  + \Div ( H_2 + {\tilde H}_2 ) -  \Div ( H_3 + {\tilde H}_3 )
= 2\Div ( P_2 H_1 + P_1 H_2 ).
\end{eqnarray*}
For the terms in \eqref{E-new1} with ``$h$" and ``$T$" we have
\begin{eqnarray*}
\| {\tilde h}_3 \|^2 - \| h_1 \|^2 - \| h_2 \|^2 &=& \!\!\!2\sum\nolimits_{a,i} \| {\tilde h}_3 (E_a, {\cal E}_i ) \|^2
- \sum\nolimits_{a,b} \| P_2 h_1 (E_a, E_b ) \|^2 - \sum\nolimits_{i,j} \| P_1 h_2 ({\cal E}_i, {\cal E}_j ) \|^2 \\
&=& 2\| {\tilde h}_3 |_{\mD_1 \times \mD_2 } \|^2 - \| P_2 h_1 \|^2 - \| P_1 h_2 \|^2, \\
\| {\tilde T}_3 \|^2 - \| T_1 \|^2 - \| T_2 \|^2 &=& \!\!\!2\sum\nolimits_{a,i} \| {\tilde T}_3 (E_a, {\cal E}_i ) \|^2
- \sum\nolimits_{a,b} \| P_2 T_1 (E_a, E_b ) \|^2 - \sum\nolimits_{i,j} \| P_1 T_2 ({\cal E}_i, {\cal E}_j ) \|^2 \\
&=& 2 \| {\tilde T}_3 |_{\mD_1 \times \mD_2 } \|^2 - \| P_2 T_1 \|^2 - \| P_1 T_2 \|^2,\\
\| h_3 \|^2 - \| {\tilde h}_1 \|^2 - \| {\tilde h}_2 \|^2 &=&
- 2\sum\nolimits_{i , \mu} \| {\tilde h}_1 (e_\mu , {\cal E}_i) \|^2 - 2\sum\nolimits_{a , \mu} \| {\tilde h}_2 (e_\mu , E_a ) \|^2 \\
&&- \sum\nolimits_{i , j} \| {\tilde h}_1 ({\cal E}_i , {\cal E}_j) \|^2
- \sum\nolimits_{a , b} \| {\tilde h}_2 (E_a , E_b ) \|^2 \\
&=& - 2 \| {\tilde h}_1 |_{\mD_2 \times \mD_3 } \|^2 - 2 \| {\tilde h}_2 |_{\mD_1 \times \mD_3 } \|^2
- \,\| P_1 h_2 \|^2 - \,\| P_2 h_1 \|^2,\\
\| T_3 \|^2 - \| {\tilde T}_1 \|^2 - \| {\tilde T}_2 \|^2 &=& - 2 \| {\tilde T}_1 |_{\mD_2 \times \mD_3 } \|^2
- 2 \| {\tilde T}_2 |_{\mD_1 \times \mD_3 }\|^2 - \| P_1 T_2 \|^2 - \| P_2 T_1 \|^2 .
\end{eqnarray*}
Hence, \eqref{E-new1} reduces to \eqref{E-new2}.
Applying Stokes' theorem to \eqref{E-new2}, we get the integral formula \eqref{E-new-IF}.
\end{proof}

\begin{definition} \rm
Let $TM = \mD_1 \oplus \mD_2 \oplus \mD_3$. We define for all $X,Y,Z \in \mathfrak{X}_M$:
\begin{eqnarray*}
&& \mh_{12}(X,Y) = {\tilde h}_3 (P_1 X , P_2 Y) + {\tilde h}_3 (P_2 X , P_1 Y),\quad
g( \mA_{12,Z} X, Y )= g(\mh_{12}(X,Y) ,Z) ,\\
&& \mT_{12}(X,Y) = {\tilde T}_3 (P_1 X , P_2 Y) + {\tilde T}_3 (P_2 X , P_1 Y),\quad
g( \mTs_{12,Z} X, Y ) = g(\mT_{12}(X,Y) ,Z)\,.
\end{eqnarray*}
\end{definition}

Tensors $\mh_{ij}, \mT_{ij}, \mA_{ij}, \mTs_{ij}$ for other pairs of  distributions $\mD_i,\mD_j$, where $i,j \in \{1,2,3\}$, are defined analogously.
We note that the following equalities are true:
\[
 2\,\| {\tilde h}_3 |_{\mD_1 \times \mD_2} \|^2 = \| \mh_{12} \|^2,\quad
 2\,\| {\tilde T}_3 |_{\mD_1 \times \mD_2 } \|^2 = \| \mT_{12} \|^2,\quad
 {\rm etc}.
\]

\begin{definition}\rm
A pair $(\mD_1,\mD_2)$ of distributions on $(M,g)$ is called

a) \textit{mixed totally geodesic}, if $\mh_{12}=0$,
i.e., $h_{\,12}(X,Y)=0$ for all $X\in \mathfrak{X}_{\mD_1}$ and $Y\in \mathfrak{X}_{\mD_2}$.

b) \textit{mixed integrable}, if $\mT_{12}=0$,
i.e., $T_{\,12}(X,Y)=0$ for all $X\in \mathfrak{X}_{\mD_1}$ and $Y\in \mathfrak{X}_{\mD_2}$.
\end{definition}

\begin{proposition}
Let a closed Riemannian manifold $(M,g)$ admit two orthogonal unit vector fields $\xi_1$ and $\xi_2$.
Put $\mD_i=span(\xi_i)\ (i=1,2)$ and $\mD_3=\mD_1^\bot\cap\mD_2^\bot$.

1. If pairs $(\mD_1,\mD_3)$ and $(\mD_2,\mD_3)$ of distributions are mixed totally geodesic, and the pair $(\mD_1,\mD_2)$ is mixed integrable, then
either the sectional curvature $K(\xi_1(x),\xi_2(x))>0$ at some point $x\in M$, or
$K(\xi_1,\xi_2)=0$ and $P_2 H_1=P_1 H_2={\tilde h}_3 |_{\mD_1 \times \mD_2 }={\tilde T}_1 |_{\mD_2 \times \mD_3 }={\tilde T}_2 |_{\mD_1 \times \mD_3 }=0$ on $M$.

2. If $\xi_1$ and $\xi_2$ are geodesic vector fields,
the pair $(\mD_1,\mD_2)$ of distributions is mixed totally geodesic, and the pairs  $(\mD_1,\mD_3)$ and  $(\mD_2,\mD_3)$ are mixed integrable, then
either $K(\xi_1(x),\xi_2(x))<0$ at some point $x\in M$, or
$K(\xi_1,\xi_2)=0$ and ${\tilde h}_1 |_{\mD_2 \times \mD_3 }={\tilde h}_2 |_{\mD_1 \times \mD_3 }={\tilde T}_3 |_{\mD_1 \times \mD_2 }=0$ on $M$.
\end{proposition}

\begin{proof}
 Note that $h_i=H_i\,g_{\,|\mD_i}$ and $T_i=0$ for $i=1,2$.
By \eqref{E-new-IF}, we get
\begin{equation*}
\int_M \big( K(\xi_1, \xi_2) -  2\| P_2 H_1 \|^2 - 2\| P_1 H_2 \|^2
- \| {\tilde h}_3 |_{\mD_1 \times \mD_2 } \|^2 - \| {\tilde T}_1 |_{\mD_2 \times \mD_3 } \|^2 - \| {\tilde T}_2 |_{\mD_1 \times \mD_3 } \|^2
\big) \, {\rm d} \vol_g =0 .
\end{equation*}
Thus, $\int_M K(\xi_1 , \xi_2)\, {\rm d} \vol_g \ge0$ and the first claim follows.

For the second claim, from the assumptions we have $h_i=H_i=T_i=0$ for $i=1,2$. Also ${\tilde h}_3 |_{\mD_1 \times \mD_2 }={\tilde T}_3 |_{\mD_1 \times \mD_2 }=0$,
where $\mD_3=\mD_1^\bot\cap\mD_2^\bot$.
By \eqref{E-new-IF}, we get
\begin{eqnarray*}
\int_M \big( K(\xi_1 , \xi_2)
+ \| {\tilde h}_1 |_{\mD_2 \times \mD_3 } \|^2 + \| {\tilde h}_2 |_{\mD_1 \times \mD_3 } \|^2
+ \| {\tilde T}_3 |_{\mD_1 \times \mD_2 } \|^2 \big) \, {\rm d} \vol_g =0 .
\end{eqnarray*}
Thus, $\int_M  K(\xi_1 , \xi_2)\, {\rm d} \vol_g \le0$ and the second claim follows.
\end{proof}

We note that in a similar way one can obtain integral formulas with other curvatures built from $S_{\,\mD_i, \mD_j}$, for various pairs of orthogonal distributions $\mD_i, \mD_j$ on a manifold.

\section{Variations}
\label{sec:var}

In this section
we obtain Euler-Lagrange equations for the action \eqref{Eq-Smix} considered
on the set of all Riemannian metrics $g$ on $M$ for which
$\mD_1$ and $\mD_2$ are orthogonal.
We consider smooth $1$-parameter variations $\{g_t\in{\rm Riem}(M):\,|t|<\eps\}$ of the metric $g_0 = g$.
We assume that the infinitesimal variations, represented by symmetric $(0,2)$-tensors
\[
 {B}_t = \partial g_t/\partial t
\]
are supported in a relatively compact domain $\Omega$ in $M$, i.e., $g_t=g$ and ${B}_t=0$ outside $\Omega$ for all $|t|<\eps$.
Then, the Euler-Lagrange equations that we obtain are valid at every point of $\Omega$, which can be chosen arbitrarily. If $M$ is compact, we assume that $\Omega = M$.
A variation $g_t$ is said to be \emph{volume-preserving} if ${\rm Vol}(\Omega,g_t) = {\rm Vol}(\Omega,g)$ for all $t$, where ${\rm Vol}(\Omega,g) = \int_\Omega {\rm d} \vol_g$ for the volume form ${\rm d} \vol_g$ of $g$.

While keeping orthogonality of $\mD_1$ and $\mD_2$ for all variations $g_t$ is necessary for the development of variational formulas in this paper, we also consider less general, \emph{adapted} variations, defined as~follows.
\begin{definition}
Let $\mD_3$ be the $g$-orthogonal complement of $\mD_1 \oplus \mD_2$, we call a variation 
$\{g_t\in{\rm Riem}(M):\,|t|<\eps\}$ \emph{adapted} if all three distributions $\mD_1, \mD_2 , \mD_3$ are pairwise orthogonal with respect to $g_t$ for all $|t|<\eps$.
\end{definition}
Hence, adapted variations preserve the orthogonal decomposition of $TM$ defined by $\mD_1,\mD_2$ and $\mD_3$.
We~adopt the notations $\partial_t \equiv \partial/\partial t,\ {B}\equiv{\dt g_t}_{\,|\,t=0}$.
Since $B$ is symmetric, we have $\<C,\,B\>=\<{\rm Sym}(C),\,B\>$ for any  $(0,2)$-tensor $C$.

\subsection{Auxiliary tensors and variation formula for $S_{\,\mD_1, \mD_2}$}

Tensors defined in this section are used to obtain the Euler-Lagrange equations, but do not appear in their final form.
Definitions below will be usually given in general, for a distribution $\mD_m$,  
but, especially where orthonormal frames and relations between different distributions are considered, sometimes we give a definition for $\mD_1$ only, as definitions for 
other distributions are analogous. In~general, this section follows the notation from \cite{rz-2} adapted to multiple distributions, so tensors are usually described with an additional index, indicating to which distribution they correspond.

%
Recall that for the shape operator, we have
 $g(A_{m,Z} X , Y) = g( h_m(X,Y) , Z)$ for $m=1,2,3$,
which is non-zero only for $X,Y \in \mD_m$ and $Z \in \mD_m^\perp$. Similarly,
$g({\tilde A}_{m,Z} X , Y) = g( {\tilde h}_m(X,Y) , Z)$, 
which is non-zero only for $X,Y \in \mD_m^\perp$ and $Z \in \mD_m$. We will use notation ${A}_{1, a}={A}_{1, E_a}$, $A_{2, i} = A_{2,{\cal E}_i}$, $A_{3,\mu} = A_{3,e_\mu}$ and analogously for other tensors.

For all $X,Y \in TM$ and a given $Z \in \mathfrak{X}_M$, let
\begin{eqnarray*}
 \theta_m(X,Y) \eq \frac{1}{2}\,( T^\sharp_{m, P_m^\perp X} (P_m Y) + T^\sharp_{m, P_m^\perp Y} (P_m X) ),\\
 {\tilde \theta}_m(X,Y) \eq \frac{1}{2}\,( {\tilde T}^\sharp_{m, P_m X} (P_m^\perp Y) + {\tilde T}^\sharp_{m, P_m Y} (P_m^\perp X) ),\\
 \alpha_m (X,Y) \eq \frac{1}{2}\,( A_{m, P_m^\perp X} (P_m Y) + A_{m, P_m^\perp Y} (P_m X) ),\\
 {\tilde \alpha}_m (X,Y) \eq \frac{1}{2}\,( {\tilde A}_{m, P_m X} (P_m^\perp Y) + {\tilde A}_{m, P_m Y} (P_m^\perp X) ),\\
 {\tilde \delta}_{m,Z} (X,Y) \eq \frac{1}{2}\,( g( \nabla_{P_m X} Z , (P_m^\perp Y)) +  g( \nabla_{P_m Y} Z , (P_m^\perp X)) ).
\end{eqnarray*}

For any $(1,2)$-tensors $Q_1,Q_2$ we define the $(0,2)$-tensor $\Upsilon_{Q_1,Q_2}$~by
\[
\Upsilon_{Q_1,Q_2} (X,Y) = \sum\nolimits_{\,k,m} ( g(Q_1(\xi_k, \xi_m) , X) g(Q_2(\xi_k , \xi_m) , Y) + g(Q_2(\xi_k, \xi_m) , X) g(Q_1(\xi_k , \xi_m) , Y))
\]
for all $X,Y \in TM$, where $\{ \xi_k \}$ is a local orthonormal frame of $TM$.
We have, e.g.,
\begin{eqnarray*}
&& \Upsilon_{ {\tilde \alpha}_1 , {\alpha}_1  }(X,Y) \\ 
&& =\frac{1}{2}\sum\nolimits_{\,a,i} \big( g(P_1^\perp X , {\tilde A}_{1 ,  a} {\cal E}_i )\, g( P_1 Y,  {A}_{1 , i }  E_a )
 +
 g(P_1^\perp Y , {\tilde A}_{1 ,  a} {\cal E}_i )\, g( P_1 X,  {A}_{1 , i }  E_a ) \big) \\
&&
 +\,\frac{1}{2}\sum\nolimits_{\,a,\mu} \big( g(P_1^\perp X ,{\tilde A}_{ 1 ,  a} e_\mu )\, g( P_1 Y,  A_{1 ,  \mu} E_a  )
 +
 g(P_1^\perp Y ,{\tilde A}_{ 1 , a} e_\mu )\, g( P_1 X,  A_{1 , \mu} E_a  ) \big).
\end{eqnarray*}
for all $X,Y \in TM$.

For $(1,1)$-tensors $S_1, S_2$ let $[S_1, S_2] = S_1 S_2 - S_2 S_1$.
We have the following tensors, which may have non-zero values only for $X,Y \in \mD_1$:
\begin{eqnarray*}
&& {\cal K}_1^\flat (X,Y) = \sum\nolimits_{\,i} g(\,[ {T}^\sharp_{1,i}, \,{A}_{1,i}] X , \,Y )
 + \sum\nolimits_{\,\mu} g(\,[ {T}^\sharp_{1,\mu}, \,{A}_{1,\mu}] X , \,Y ),
\\
&& {\cal T}_1^\flat (X,Y) = \sum\nolimits_{i} g( ( T_{1,i}^\sharp T_{1,i}^\sharp X,Y) + \sum\nolimits_{\mu} g( T_{1,\mu}^\sharp T_{1,\mu}^\sharp X,Y), \\
&& \frac{1}{2}\Upsilon_{{\tilde h}_1,{\tilde h}_1} (X,Y) = \sum\nolimits_{\,i,j} g( {\tilde h}_1( {\cal E}_i, {\cal E}_j) ,X) g( {\tilde h}_1( {\cal E}_i, {\cal E}_j) ,Y)  \\
&& +\,2\sum\nolimits_{\,i, \mu} g( {\tilde h}_1({\cal E}_i , e_\mu ) ,X)
g( {\tilde h}_1({\cal E}_i , e_\mu ) ,Y)
 +\sum\nolimits_{\,\mu,\nu} g( {\tilde h}_1( e_\mu , e_\nu ) ,X)
g( {\tilde h}_1( e_\mu , e_\nu ) ,Y), \\
&& \frac{1}{2}\Upsilon_{{\tilde T}_1,{\tilde T}_1} (X,Y) = \sum\nolimits_{\,i,j} g( {\tilde T}_1( {\cal E}_i, {\cal E}_j) ,X) g( {\tilde T}_1( {\cal E}_i, {\cal E}_j) ,Y)  \\
 && +\,2\sum\nolimits_{\,i, \mu} g( {\tilde T}_1({\cal E}_i , e_\mu ) ,X)
g( {\tilde T}_1({\cal E}_i , e_\mu ) ,Y)
 +\sum\nolimits_{\,\mu,\nu} g( {\tilde T}_1( e_\mu , e_\nu ) ,X)
g( {\tilde T}_1( e_\mu , e_\nu ) ,Y),
\end{eqnarray*}
and the following tensors, which may have non-zero values only for $X,Y \in \mD_1^\perp$:
\begin{eqnarray*}
&& \widetilde{\cal K}_1^\flat (X,Y) = \sum\nolimits_{\,a} g(\,[{\tilde T}^\sharp_{1,a},\, {\tilde A}_{1,a}] X , \,Y) ,\\
&& \tilde{\cal T}_1^\flat (X,Y) = \sum\nolimits_{\,a} g( {\tilde T}_{1,a}^\sharp {\tilde T}_{1,a}^\sharp X ,Y), \\
&& \frac{1}{2}\Upsilon_{h_1,h_1} (X,Y) = \sum\nolimits_{\,a,b} g( h_1(E_a, E_b) , X) g( h_1(E_a, E_b),Y) , \\
&& \frac{1}{2}\Upsilon_{T_1,T_1} (X,Y) = \sum\nolimits_{\,a,b} g( T_1(E_a, E_b) , X) g( T_1(E_a, E_b),Y).
\end{eqnarray*}
For the metric $g$, let
 $ g_{m}^\top = g |_{\mD_m \times \mD_m}$
 and
  $g_{m}^\perp = g |_{\mD_m^\perp \times \mD_m^\perp}$ for $m=1,2,3$.

The sum of two orthogonal distributions $\mD_i$ and $\mD_j$ will be denoted by $\mD_{ij}$, and the orthogonal projection onto it will be denoted by $P_{ij}$. 
The tensors defined above considered for $\mD_{ij}$ will be denoted by index $ij$, e.g., for $\mD_{12} = \mD_1 \oplus \mD_2$ we have
\begin{eqnarray*}
&& h_{12}(X,Y) = \frac{1}{2}\,P_3(\nabla_{ P_{12} X }(P_{12} Y) + \nabla_{ P_{12} Y }(P_{12} X) ) , \quad X,Y \in \mD_{12} , \\
 && \widetilde{\cal K}_{12}^\flat = \big( \sum\nolimits_{\,a} [\tT_{12,a},\, \tA_{12,a}]^\flat
 + \sum\nolimits_{\,i} [\tT_{12,i}, \,\tA_{12,i}]^\flat \big)_{\,|\,\mD_3 \times \mD_3} ,\\
 && H_{12} = P_3 H_1 + P_3 H_2,\quad {\tilde H}_{12} = H_3,\quad
 {\cal K}_{12}^\flat = \sum\nolimits_{\,\mu} [\T_{12,\mu}, \,A_{12,\mu}]^\flat,
\end{eqnarray*}
and 
\[
 -2\Upsilon_{{\tilde\alpha}_{12}, \theta_{12}} (X,Y) = -g(X , \tA_{12,a} e_\mu)\,g(Y, \T_{12,\mu} E_a)
 - g(X, \tA_{12,i} e_\mu )\,g(Y, \T_{12,\mu} {\cal E}_i ) ,
\]
which may have non-zero values only for $X \in \mD_3$, $Y \in \mD_3^\perp$.

\begin{proposition} \label{propdtS1S2general}
Let variation $g_t$ of $g$ satisfy $g_t(X,Y) =0$ for all $X \in \mathfrak{X}_{\mD_1}, Y \in \mathfrak{X}_{ \mD_2}$, and all $t$. Then
\begin{equation}\label{dtS1S2general}
 2\,{\rm\frac{d}{dt}}\, J_{\mD_1,\mD_2}(g_t) |_{\,t=0}
 =  I_1 + I_2 +I_3 +I_4 +I_5 +I_6,
\end{equation}
where
\begin{eqnarray} \nonumber
\label{I1}
&& I_1 = \int_{\Omega}\big\<4\Upsilon_{{\tilde\alpha}_1, \theta_1} -\Div{\tilde h_1} |_{\mD_1^\perp \times \mD_1^\perp}
 -\widetilde{\cal K}_1^\flat -H_1^\flat \otimes H_1^\flat + \frac{1}{2} \Upsilon_{h_1,h_1} +2\,\tilde{\cal T}_1^\flat \nonumber \\
&&\quad  +\,\frac{1}{2} \Upsilon_{T_1,T_1} +4\<\theta_1,\,H_{23}\> +2 ( \Div(\alpha_1 -\tilde\theta_1) )_{ \,| {(\mD_1 \times \mD_1^\perp )
\cup (\mD_1^\perp \times \mD_1) } }
 {+}2\Upsilon_{\alpha_1, {\tilde\alpha}_1 + {\tilde \theta}_1} {+}2\<{\tilde\theta}_1 - {\tilde\alpha}_1, H_1 \>\nonumber \\
&&\quad +\,2\,{\rm Sym}(H_1^{\flat} \otimes H_{23}^{\flat})
 -2\,{\tilde\delta}_{1,H_1} + 2\Upsilon_{{\tilde \theta}_1, \theta_1 - \alpha_1} \nonumber \\
&&\quad  +\,\frac{1}{2}\,\big( S_{\,\mD_1 , \mD_1^\perp} + \Div( H_{23} - H_1)\big)\, g_{1}^{\perp},\ B\big\>\,{\rm d}\vol_g , 
 \\
 \label{I2}
&&I_2 = \int_{\Omega} \< -\Div{h_1} |_{\mD_1\times\mD_1} - {\cal K}_1^\flat - H_{23}^\flat \otimes H_{23}^\flat
 +\frac{1}{2}\Upsilon_{ {\tilde h}_1,{\tilde h}_1} + \frac{1}{2} \Upsilon_{ {\tilde T}_1, {\tilde T}_1}
 +2\,{\cal T}_1^\flat \nonumber \\
&&\quad +\,\frac{1}{2}\,\big( S_{\,\mD_1 , \mD_1^\perp} + \Div(H_1 -  H_{23})\big)\, g_{1}^\top,\ B\big\>\,{\rm d}\vol_g , 
\\
\label{I3}
&& I_3 = \int_{\Omega}
\big\<4\Upsilon_{{\tilde\alpha}_2, \theta_2} -\Div{\tilde h_2} |_{\mD_2^\perp \times \mD_2^\perp} -\widetilde{\cal K}_2^\flat
- H_2^\flat \otimes H_2^\flat + \frac{1}{2} \Upsilon_{h_2,h_2} + \frac{1}{2} \Upsilon_{T_2,T_2} +2\,\tilde{\cal T}_2^\flat \nonumber \\
&&\quad +\,4\,\<\theta_2,\, H_{13} \>  +\,2 ( \Div(\alpha_2 -\tilde\theta_2) )_{ \,| { (\mD_2 \times \mD_2^\perp) \cup (\mD_2^\perp \times \mD_2) } } + 2 \Upsilon_{\alpha_2, {\tilde\alpha}_2 + {\tilde \theta}_2} + 2\<{\tilde \theta}_2 - {\tilde\alpha}_2, H_2 \>  \nonumber \\
&&\quad +\,2\,{\rm Sym}(H_2^{\flat} \otimes H_{13}^{\flat})
 -2\,{\tilde \delta}_{2,H_2} +2\Upsilon_{{\tilde \theta}_2, \theta_2 - \alpha_2} \nonumber \\
&&\quad
 +\frac{1}{2}\,\big( S_{\,\mD_2 , \mD_2^\perp} + \Div( H_{13} - H_2)\big)\, g_{2}^{\perp},\ B\big\>\,{\rm d}\vol_g , 
 \\
 \label{I4}
&& I_4 = \int_{\Omega} \<-\Div{h_2}|_{\mD_2\times\mD_2} -{\cal K}_2^\flat -H_{13}^\flat\otimes H_{13}^\flat
+\frac{1}{2}\Upsilon_{{\tilde h}_2,{\tilde h}_2} + \frac{1}{2} \Upsilon_{ {\tilde T}_2, {\tilde T}_2} +2\, {\cal T}_2^\flat  \nonumber \\
&&\quad +\,\frac{1}{2}\,\big( S_{\,\mD_2 , \mD_2^\perp} + \Div(H_2 -  H_{13})\big)\, g_{2}^\top,\ B\big\>\,{\rm d}\vol_g , 
\end{eqnarray}
\begin{eqnarray}
\label{I5}
&& I_5 = \int_{\Omega}\big\< - 4\Upsilon_{{\tilde\alpha}_{12}, \theta_{12}} + \Div{\tilde h_{12}} |_{\mD_{12}^\perp \times\mD_{12}^\perp } + \widetilde{\cal K}_{12}^\flat + H_{12}^\flat \otimes H_{12}^\flat - \frac{1}{2} \Upsilon_{h_{12},h_{12}}
 -\frac{1}{2} \Upsilon_{T_{12},T_{12}} - 2\,\tilde{\cal T}_{12}^\flat \nonumber \\
&&\quad - \,4\,\<\theta_{12},\, {\tilde H}_{12} \>  - \,2 ( \Div(\alpha_{12} -\tilde\theta_{12}) )_{ \,| ({\mD_{12} \times \mD_{12}^\perp) \cup (\mD_{12}^\perp \times \mD_{12}) } } - 2 \Upsilon_{\alpha_{12}, {\tilde\alpha}_{12} + {\tilde \theta}_{12}}
 - 2\<{\tilde \theta}_{12} - {\tilde\alpha}_{12}, H_{12} \>
 \nonumber \\
&&\quad -2\,{\rm Sym}(H_{12}^{\flat} \otimes {\tilde H}_{12}^{\flat})
+ 2 \, {\tilde \delta}_{12, H_{12}} - 2 \, \Upsilon_{{\tilde \theta}_{12}, \theta_{12} {-} \alpha_{12}}  \nonumber \\
 &&\quad
 -\,\frac{1}{2}\,(S_{\,\mD_{12}, \mD_{12}^\perp} +\Div( {\tilde H}_{12} - H_{12}))\,g_{3}^{\top}, B\big\>\,{\rm d}\vol_g , 
 \\
 \label{I6}
&& I_6 = \int_{\Omega} \< \Div{h_{12}} |_{\mD_{12}\times\mD_{12}} +{\cal K}_{12}^\flat +{\tilde H}_{12}^\flat \otimes {\tilde H}_{12}^\flat
- \frac{1}{2} \Upsilon_{ {\tilde h}_{12},{\tilde h}_{12}} -\frac{1}{2}\Upsilon_{ {\tilde T}_{12}, {\tilde T}_{12}} -2\,{\cal T}_{12}^\flat \nonumber \\
&&\quad -\,\frac{1}{2}\,\big( S_{\,\mD_{12} , \mD_{12}^\perp} + \Div(H_{12} -  {\tilde H}_{12} )\big)\, g_{3}^\perp,\ B\big\>\,{\rm d}\vol_g.  
\end{eqnarray}
\end{proposition}

\begin{proof}
Since $S_{\,\mD_1 , \mD_1^\perp}$ is the mixed scalar curvature of $(\mD_1,\mD_1^\bot)$, we have the following
formula for general variations of $S_{\,\mD_1 , \mD_1^\perp}$:
\begin{eqnarray}\label{E-varJh-init2-1}
{\rm\frac{d}{dt}}\int_M S_{\,\mD_1 , \mD_1^\perp(t)} \,{\rm d}\vol_{g_t} |_{\,t=0} = I_1 + I_2,
\end{eqnarray}
which was obtained in \cite{rz-2} 
and adjusted to our notation defined above. 
In particular, we explicitly write restrictions of $\Div{\tilde h_1}  |_{\mD_1^\perp \times\mD_1^\perp }$ and $\Div h_1 |_{\mD_1\times\mD_1}$ to certain distributions -- this notation was omitted in \cite{rz-2}, where only two complementary distributions were considered.

Similarly, for general variations of $S_{\,\mD_2 , \mD_2^\perp}$ we get
\begin{eqnarray}\label{E-varJh-init2-2}
{\rm\frac{d}{dt}}\int_M S_{\,\mD_2 , \mD_2^\perp(t)} \,{\rm d}\vol_{g_t} |_{\,t=0} = I_3 + I_4,
\end{eqnarray}
and for $\mD_{12}=\mD_1 \oplus \mD_2$ we have the following:
\begin{eqnarray}\label{E-varJh-init2-3}
- \, {\rm\frac{d}{dt}}\int_M S_{\,\mD_{12} , \mD_{12}^\perp(t)} \,{\rm d}\vol_{g_t}|_{\,t=0} = I_5 + I_6 .
\end{eqnarray}

For
a family of metrics $g_t$ such that $g_t(X,Y) =0$ for $X \in \mD_1$ and $Y \in \mD_2$, we have
\[
 \mD_1(t)=\mD_1,\quad
 \mD_2(t)=\mD_2,\quad
 \mD_1^\perp(t)=\mD_2\oplus\mD_3(t),\quad
 \mD_2^\perp(t)=\mD_1\oplus\mD_3(t).
\]
Thus, as in \eqref{SD1D2}, we have
 $2S_{\,\mD_1, \mD_2} = S_{\,\mD_1 , \mD_1^\perp(t)} +  S_{\,\mD_2 , \mD_2^\perp(t)} -  S_{\,\mD_3(t) , \mD_3^\perp}$,
which yields
\begin{eqnarray*}
 2\,{\rm\frac{d}{dt}}\,J_{\mD_1,\mD_2}(g_t) |_{\,t=0}
 &=&{\rm\frac{d}{dt}}\int_M ( S_{\,\mD_1 , \mD_1^\perp(t)} + S_{\,\mD_2 , \mD_2^\perp(t)} -  S_{\,\mD_3(t) , \mD_3^\perp} )\,{\rm d}\vol_{g_t} |_{\,t=0}  \\
&=&{\rm\frac{d}{dt}}\int_M ( S_{\,\mD_1 , \mD_1^\perp(t)} + S_{\,\mD_2 , \mD_2^\perp(t)} -  S_{\,\mD_{12} , \mD_{12}^\perp} )\,{\rm d}\vol_{g_t} |_{\,t=0} ,
\end{eqnarray*}
since $\mD_3(t)$ is the $g(t)$-orthogonal complement of $\mD_{12}= \mD_1 \oplus \mD_2$. From \eqref{E-varJh-init2-1}--\eqref{E-varJh-init2-3}, we get \eqref{dtS1S2general}.
\end{proof}

To obtain the Euler-Lagrange equations for the action \eqref{Eq-Smix}, we note that \eqref{dtS1S2general} can be written as
\[
{\rm\frac{d}{dt}}\, J_{\mD_1,\mD_2}(g_t) |_{\,t=0}
 = \int_M \< \delta J_{\mD_1,\mD_2} , B \> \,{\rm d}\vol_g
\]
for the symmetric $(0,2)$-tensor $\delta J_{\mD_1,\mD_2}$ defined by the sum of integrands in \eqref{I1}-\eqref{I6}. It follows that metric $g$ is critical for the action \eqref{Eq-Smix} with respect to variations keeping orthogonality of $\mD_1$ and $\mD_2$ if and only if $\int_M \< \delta J_{\mD_1,\mD_2} , B \> \,{\rm d}\vol_g =0$ for all symmetric $(0,2)$-tensor fields $B$, such that $B(X,Y) =0$ for all $X \in \mathfrak{X}_{\mD_1}$ and $Y \in \mathfrak{X}_{\mD_2}$. This is equivalent to Euler-Lagrange equations
\begin{equation} \label{eqEL}
\delta J_{\mD_1,\mD_2} \, |_{\mD_i \times \mD_j} =0 ,
\end{equation}
where $(i,j) \notin \{ (1,2) , (2,1) \}$.
In next subsections, we derive the Euler-Lagrange equations for the action \eqref{Eq-Smix}, by finding
components of \eqref{eqEL} corresponding to particular distributions $\mD_i$. As every infinitesimal variation $B$ can be decomposed in components $B |_{\mD_i \times \mD_j}$, these Euler-Lagrange equations together can describe any variation of the metric (preserving orthogonality of $\mD_1$ and $\mD_2$), while each of them also has a separate meaning (describing variations keeping some components of $g$ fixed), analogously as $g^\perp$- and $g^\top$-variations considered in \cite{rz-2}.

\subsection{Variations for $B = B |_{\mD_1 \times \mD_1}$}

\begin{proposition}
Let variation $g_t$ of $g$ satisfy $B(X,Y) = B(P_1 X, P_1 Y)$ for all $X,Y \in \mathfrak{X}_M$.
Then ${\rm\frac{d}{dt}}\, J_{\mD_1,\mD_2}(g_t) |_{\,t=0}=0$ if and only if the following Euler-Lagrange equation~holds for all $X,Y\in\mD_1$:
\begin{eqnarray}\label{EL11withmixed}
&& -\,2( \Div (P_2 h_1) ) (X,Y)
 - 2\sum\nolimits_{\,i,a} \big( g( \Ts_{1,i} E_a , Y )\,g( A_{1,i} E_a, X )
 + g( \Ts_{1,i} E_a , X )\,g( A_{1,i} E_a, Y ) \big)  \nonumber\\
&& -\,2g(X, H_2)\,g(Y, H_2) -g(X, H_2)\,g(Y,H_3) -g(X, H_3)\,g(Y, H_2) \nonumber\\
&& +\,2\sum\nolimits_{i,j} \big(  g( h_2 ({\cal E}_i , {\cal E}_j ) , X  )\,g( h_2 ({\cal E}_i , {\cal E}_j ) , Y )
 + g( P_1 T_2 ({\cal E}_i , {\cal E}_j ) , X  )\,g( P_1 T_2 ({\cal E}_i , {\cal E}_j ) , Y ) \big)
\nonumber\\
&& -\,4\sum\nolimits_{i,a} g( \Ts_{1,i} E_a , X )\,g( \Ts_{1,i} E_a , Y )
 + S_{\,\mD_1, \mD_2} g(X,Y)+\,\Div(P_2 H_1 + P_2 H_3 - H_2 )\,g(X,Y) \nonumber\\
&& + \sum\nolimits_{i,\mu}\!\Big( {-}g( \mA_{13,i} e_\mu , Y )\,g( (\mh_{23}
{+} \mT_{23}) ({\cal E}_i ,  e_\mu ) , X )
 {-}g( \mA_{13,i} e_\mu , X )\,g( (\mh_{23}
 {+} \mT_{23}) ({\cal E}_i ,  e_\mu ) , Y ) \nonumber\\
&& -\,g( \mA_{12,\mu} {\cal E}_i , Y )\,g( (\mh_{23}
- \mT_{23}) ({\cal E}_i ,  e_\mu ) , X )
-g( \mA_{12,\mu} {\cal E}_i , X )\,g( (\mh_{23}
- \mT_{23}) ({\cal E}_i ,  e_\mu ) , Y ) \nonumber\\
&& +\,\big( g( \mTs_{12,\mu} {\cal E}_i , Y )\,g( \mA_{12, \mu} {\cal E}_i , X)
   + g( \mTs_{12,\mu} {\cal E}_i , X )\,g( \mA_{12, \mu} {\cal E}_i , Y) \big)
\nonumber\\
&& -\,g( \mA_{13,i} e_\mu , X )\,g( \mTs_{13,i} e_\mu , Y ) - g( \mA_{13,i} e_\mu , Y )\,g( \mTs_{13,i} e_\mu , X )
\nonumber\\
&& +\,2\,g( \mh_{23} ({\cal E}_i , e_\mu ) , X  )\,g( \mh_{23} ({\cal E}_i , e_\mu ) , Y )
+\,2\,g( \mT_{23} ({\cal E}_i , e_\mu ) , X  )\,g( \mT_{23} ({\cal E}_i , e_\mu ) , Y ) \nonumber\\
&& +\,2\,g( \mTs_{12,\mu} {\cal E}_i ,  X )\,g( \mTs_{12,\mu} {\cal E}_i , Y )
  -2\,g( \mTs_{13, i} e_\mu , X )\,g( \mTs_{13,i} e_\mu , Y) \Big)
 =0.
\end{eqnarray}
\end{proposition}

\begin{proof}
If $B(X,Y) = B(P_1 X, P_1 Y)$ for all $X,Y \in \mathfrak{X}_M$, then \eqref{dtS1S2general} has
the following terms:
\begin{eqnarray*} 
&& {\rm\frac{d}{dt}}\int_M \big( S_{\,\mD_1,\mD_1^\perp(t)} +S_{\,\mD_2,\mD_2^\perp(t)} -S_{\,\mD_3(t),\mD_{12}}\big)
 \,{\rm d}\vol_{g_t}|_{\,t=0} = I_2 + \tilde I_3 + I_6,
 \end{eqnarray*}
where
\begin{eqnarray} \label{tildeI3} \nonumber
&& \tilde I_3 = \int_{\Omega}\big\< -\Div{\tilde h_2} |_{\mD_{2}^\perp \times \mD_{2}^\perp} -\widetilde{\cal K}_2^\flat
 - H_2^\flat \otimes H_2^\flat +\frac{1}{2} \Upsilon_{h_2,h_2} + \frac{1}{2} \Upsilon_{T_2,T_2} \nonumber \\
&&\quad
 +\,2\,\tilde{\cal T}_2^\flat
+\frac{1}{2}\,\big( S_{\,\mD_2 , \mD_2^\perp} + \Div( H_{13} - H_2)\big)\, g_{2}^{\perp},\ B\big\>\,{\rm d}\vol_g \,
\end{eqnarray}
is the only 
part of \eqref{I3}, which may be non-vanishing, and $I_2 , I_6$ are given by \eqref{I2} and \eqref{I6}.
In~what follows, let $X,Y \in \mD_1$. 
Then we have
\begin{eqnarray*}
&& \Div h_{12} (X,Y)
=\sum\nolimits_a g(\nabla_{E_a} P_3 h_{1}(X,Y) , E_a ) + \sum\nolimits_i g(\nabla_{ {\cal E}_i } P_3 h_{1}(X,Y) , {\cal E}_i ) \\
&&\quad +\sum\nolimits_\mu g(\nabla_{e_\mu} P_3 h_{1}(X,Y) , e_\mu ) - \sum\nolimits_\mu g( h_{1}( P_1 \nabla_{e_\mu}  X , Y) , e_\mu ) \\
&&\quad -\sum\nolimits_\mu g( h_{1}( X , P_1 \nabla_{e_\mu} Y) , e_\mu ) - \sum\nolimits_\mu g( h_{12}( P_2 \nabla_{e_\mu}  X , Y) , e_\mu )
- \sum\nolimits_\mu g( h_{12}( X , P_2 \nabla_{e_\mu} Y) , e_\mu ),\\
&& \Div{h_1}(X,Y) =\sum\nolimits_a g(\nabla_{E_a} h_{1}(X,Y) , E_a ) + \sum\nolimits_i g(\nabla_{ {\cal E}_i } h_{1}(X,Y) , {\cal E}_i ) \\
&&\quad +\sum\nolimits_\mu g(\nabla_{e_\mu} h_{1}(X,Y) , e_\mu ) - \sum\nolimits_\mu g( h_{1}( P_1 \nabla_{e_\mu}  X , Y) , e_\mu ) \\
&&\quad -\sum\nolimits_\mu g( h_{1}( X , P_1 \nabla_{e_\mu} Y) , e_\mu ) -\sum\nolimits_i g( h_{1}( P_1 \nabla_{ {\cal E}_i }  X , Y) , {\cal E}_i )
 - \sum\nolimits_i g( h_{1}( X , P_1 \nabla_{ {\cal E}_i } Y) , {\cal E}_i ),\\
 &&\Div{\tilde h_2}(X,Y)
%
= \sum\nolimits_a g(\nabla_{E_a} P_2 h_1(X,Y) , E_a ) +\sum\nolimits_i g(\nabla_{ {\cal E}_i } P_2 h_1(X,Y) , {\cal E}_i ) \\
&& + \sum\nolimits_\mu g(\nabla_{e_\mu} P_2 h_1(X,Y) , e_\mu ) -\sum\nolimits_i g( h_1( P_1 \nabla_{ {\cal E}_i }  X , Y) , {\cal E}_i ) \\
&& - \sum\nolimits_i g( h_1( X, P_1 \nabla_{{\cal E}_i } Y), {\cal E}_i ) -\sum\nolimits_i g( {\tilde h_2}( P_3 \nabla_{ {\cal E}_i } X, Y), {\cal E}_i )
- \sum\nolimits_i g( {\tilde h_2}( X , P_3 \nabla_{ {\cal E}_i } Y) , {\cal E}_i ).
\end{eqnarray*}
Therefore,
\begin{eqnarray*}
 && \Div h_{12} (X,Y) - \Div {h_1} (X,Y)  - \Div {\tilde h_2} (X,Y) \\
&& = -2 \Div(P_2 h_1)(X,Y) -\sum\nolimits_\mu g( h_{12}( P_2 \nabla_{e_\mu}  X , Y) , e_\mu )\\
 && -\sum\nolimits_\mu g( h_{12}( X , P_2 \nabla_{e_\mu} Y) , e_\mu )
 +\sum\nolimits_i g( {\tilde h_2}( P_3 \nabla_{ {\cal E}_i }  X , Y) , {\cal E}_i )
 +\sum\nolimits_i g( {\tilde h_2}( X , P_3 \nabla_{ {\cal E}_i } Y) , {\cal E}_i ).
\end{eqnarray*}
Next,
\begin{eqnarray*}
 && -{\cal K}_1^\flat(X,Y) = -\sum\nolimits_i g( [\T_{1,i},\, A_{1,i}] X, \,Y ) - \sum\nolimits_\mu g([\T_{1,\mu}, \,A_{1,\mu}] X, \,Y) \\
 &&\quad = \sum\nolimits_{\,a,i} g( T_1( Y, E_a), {\cal E}_i)(h_1( E_a, X ), {\cal E}_i)
 +\sum\nolimits_{\,a,i} g( T_1(X, E_a), {\cal E}_i )( h_1(E_a, Y), {\cal E}_i) \\
 &&\quad +\sum\nolimits_{\,a,\mu} g( T_{1} (Y, E_a ) , e_\mu )(h_1( E_a , X) , e_\mu )
 +\sum\nolimits_{\,a,\mu} g( T_{1} (X, E_a ) , e_\mu )(h_1( E_a , Y) , e_\mu ),\\
 && {\cal K}_{12}^\flat (X,Y) = \sum\nolimits_\mu g( [ \T_{12,\mu}, \, A_{12,\mu}] X, \,Y ) \\
 &&\quad = - \sum\nolimits_{\,a,\mu} g( T_{1} (Y, E_a ) , e_\mu )(h_1( E_a , X) , e_\mu )
 -\sum\nolimits_{\,a,\mu} g( T_{1} (X, E_a ) , e_\mu )(h_1( E_a , Y) , e_\mu ) \\
&&\quad -\sum\nolimits_{\,i,\mu} g( T_{12} (Y, {\cal E}_i ) , e_\mu )(h_{12}( {\cal E}_i , X) , e_\mu )
-\sum\nolimits_{\,i,\mu} g( T_{12} (X, {\cal E}_i ) , e_\mu )(h_{12}( {\cal E}_i , Y) , e_\mu ),\\
 && -\widetilde{\cal K}_2^\flat (X,Y)
 = \sum\nolimits_{a,i} g( h_1 ( X , E_a ) , {\cal E}_i )\,g( {\cal E}_i , T_1 ( Y , E_a ) )
 +\sum\nolimits_{i,\mu} g( {\tilde h}_2 ( X , e_\mu ) , {\cal E}_i )\,g( {\cal E}_i , {\tilde T}_{2} ( Y , e_\mu ) ) \\
&&\quad +\sum\nolimits_{a,i} g( h_1 ( Y , E_a ) , {\cal E}_i )\,g( {\cal E}_i , T_1 ( X , E_a ) )
 +\sum\nolimits_{i,\mu} g( {\tilde h}_2 ( Y , e_\mu ) , {\cal E}_i )\,g( {\cal E}_i , {\tilde T}_{2} ( X , e_\mu ) ).
\end{eqnarray*}
Hence,
\begin{eqnarray*}
&&\hskip-5mm -{\cal K}_1^\flat (X,Y) - \widetilde{\cal K}_2^\flat (X,Y)  +	{\cal K}_{12}^\flat (X,Y)  \\
&& =2\sum\nolimits_{\,a,i} \big( g( T_1( Y, E_a ) , {\cal E}_i )( h_1( E_a, X ) , {\cal E}_i )
 + g( T_1( X, E_a ) , {\cal E}_i )( h_1( E_a, Y ) , {\cal E}_i ) \big) \\
&& +\sum\nolimits_{i,\mu} \Big( g( {\tilde h}_2 ( X , e_\mu ) , {\cal E}_i )\,g( {\cal E}_i , {\tilde T}_{2} ( Y , e_\mu ) )
 + g( {\tilde h}_2 ( Y , e_\mu ) , {\cal E}_i )\,g( {\cal E}_i , {\tilde T}_{2} ( X , e_\mu ) ) \\
&& -  g( {\tilde T}_3 (Y, {\cal E}_i ) , e_\mu )( {\tilde h}_3( {\cal E}_i , X) , e_\mu )
 - g( {\tilde T}_3 (X, {\cal E}_i ) , e_\mu )({\tilde h}_3( {\cal E}_i , Y) , e_\mu ) \Big) .
\end{eqnarray*}
Similarly,
\begin{eqnarray*}
&& ( - H_{23}^\flat \otimes H_{23}^\flat - H_2^\flat \otimes H_2^\flat  + {\tilde H}_{12}^\flat \otimes {\tilde H}_{12}^\flat )(X,Y)  \\
&&= -g(X, H_2+H_3)\,g(Y, H_2+H_3) - g(H_2,X)\,g(H_2,Y) + g(X, H_3)\,g(Y, H_3) \\
&&= -2g(X, H_2)\,g(Y, H_2) -g(X, H_2)\,g(Y,H_3) -g(X, H_3)\,g(Y, H_2),
\end{eqnarray*}
and
\begin{eqnarray*}
 && \frac{1}{2}\Upsilon_{{\tilde h}_1,{\tilde h}_1} (X,Y) =\sum\nolimits_{i,j} g( {\tilde h}_1({\cal E}_i , {\cal E}_j ) ,X) g( { \tilde h}_1({\cal E}_i, {\cal E}_j) ,Y) \\
 &&\quad +\,2\sum\nolimits_{i,\mu} g( {\tilde h}_1({\cal E}_i, e_\mu ) ,X) g( {\tilde h}_1 ({\cal E}_i , e_\mu ) ,Y)
  +\sum\nolimits_{\mu,\nu} g( {\tilde h}_1 (e_\nu, e_\mu) ,X)g( {\tilde h}_1 (e_\nu , e_\mu ) , Y), \\
 && \Upsilon_{ {\tilde h}_{12},{\tilde h}_{12}} (X,Y) =
 2\sum\nolimits_{\mu, \nu} g( {\tilde h}_{1}(e_\nu , e_\mu ) , X) g( {\tilde h}_{1} (e_\nu , e_\mu ), Y) , \\
&& \Upsilon_{ h_2, h_2 } (X,Y)
 = 2\sum\nolimits_{i,j} g( {\tilde h}_1 ({\cal E}_i, {\cal E}_j ) ,X) g( {\tilde h}_1 ({\cal E}_i , {\cal E}_j ) , Y) .
\end{eqnarray*}
Hence,
\begin{eqnarray*}
( \Upsilon_{{\tilde h}_1,{\tilde h}_1} +\Upsilon_{ h_2, h_2} -\Upsilon_{{\tilde h}_{12},{\tilde h}_{12}} )(X,Y)
 &=& 4\sum\nolimits_{i,j} g( {\tilde h}_1 ({\cal E}_i , {\cal E}_j ) ,X) g( {\tilde h}_1 ({\cal E}_i , {\cal E}_j ) , Y) \\
&& +4\sum\nolimits_{i,\mu} g( {\tilde h}_1 ({\cal E}_i, e_\mu ) ,X) g( {\tilde h}_1 ({\cal E}_i , e_\mu ), Y), 
\end{eqnarray*}
and, analogously,
\begin{eqnarray*}
( \Upsilon_{{\tilde T}_1,{\tilde T}_1} +\Upsilon_{T_2,T_2} -\Upsilon_{ {\tilde T}_{12},{\tilde T}_{12}} ) (X,Y)
 &=& 4\sum\nolimits_{i,j} g( {\tilde T}_1 ({\cal E}_i, {\cal E}_j ) ,X) g( {\tilde T}_1 ({\cal E}_i , {\cal E}_j ) ,Y )\\
&& +4\sum\nolimits_{i,\mu} g( {\tilde T}_1 ({\cal E}_i , e_\mu ) ,X) g( {\tilde T}_1 ({\cal E}_i , e_\mu ) ,Y) .
\end{eqnarray*}
We have
\begin{eqnarray*}
 && {\cal T}_1^\flat(X,Y) = \sum\nolimits_{i} g(T_{1,i}^\sharp T_{1,i}^\sharp X, Y)
 +\sum\nolimits_{\mu} g(T_{1,\mu}^\sharp T_{1,\mu}^\sharp X, Y ) \\
&&\quad =-\sum\nolimits_{i,a} g( T_{1}( X, E_a ) , {\cal E}_i )\,g({\cal E}_i, T_{1} (Y, E_a) )
 -\sum\nolimits_{\mu,a} g( T_1 ( X, E_a ) , e_\mu )\,g( e_\mu , T_1 ( Y , E_a) ),
 \end{eqnarray*}
\begin{eqnarray*}
 && -{\cal T}_{12}^\flat(X,Y) = - \sum\nolimits_{\mu} g(T_{12,\mu}^\sharp T_{12,\mu}^\sharp X, Y )
 = \sum\nolimits_{\mu} g( T_{12,\mu}^\sharp X, T_{12,\mu}^\sharp Y ) \\
&&\quad =\sum\nolimits_{\mu, a} g( T_{1} ( X, E_a) , e_\mu )\,g(e_\mu, T_{1} ( Y , E_a) )
 +\sum\nolimits_{\mu, j} g( T_{12} ( X, {\cal E}_j ) , e_\mu )\,g(e_\mu, T_{12} ( Y, {\cal E}_j ) ),\\
 && \tilde{\cal T}_{2}^\flat (X,Y) = \sum\nolimits_i g( \tT_{2,i} \tT_{2,i} X,Y) \\
&&\quad =-\sum\nolimits_{i,\mu} g( {\tilde T}_2 ( X,  e_\mu ) ,{\cal E}_i  )\,g( {\cal E}_i  , {\tilde T}_{2} ( Y , e_\mu ) )
 -\sum\nolimits_{a,i} g( T_1 ( X,  E_a ) {\cal E}_i )\,g( {\cal E}_i  , T_1 ( Y , E_a ) ).
\end{eqnarray*}
Hence,
\begin{eqnarray*}
 && ( {\cal T}_1^\flat + \tilde{\cal T}_{2}^\flat - {\cal T}_{12}^\flat )(X,Y)
 =-2\sum\nolimits_{i,a} g( T^\sharp_{1,{\cal E}_i}(E_a) ,X) g( T^\sharp_{1,{\cal E}_i}(E_a) ,Y) \\
 && +\sum\nolimits_{\mu, j} g( {\tilde T}_{3,e_\mu}({\cal E}_j) ,X) g( {\tilde T}_{3,e_\mu}({\cal E}_j) ,Y)
 -\sum\nolimits_{i,\mu} g( {\tilde T}_{2,{\cal E}_i}( e_\mu) ,X) g( {\tilde T}_{2,{\cal E}_i}( e_\mu) ,Y).
%
%
%
\end{eqnarray*}
By the above and $\Div\big( H_1 - H_{23} + H_{13} - H_2 - H_{12} +  {\tilde H}_{12} \big) = 2\,\Div\big( P_2 H_1 + P_2 H_3 - H_2 \big) $,
the Euler-Lagrange equation for variations of metric with the property
 $B(X,Y) = B(P_1 X, P_1 Y)$
is the following:
\begin{eqnarray} \label{EL11}
 && -\,2( \Div (P_2 h_1) ) (X,Y) + \sum\nolimits_{i} g( {\tilde h}_2 ( P_3 \nabla_{ {\cal E}_i } X,Y) , {\cal E}_i )
 + \sum\nolimits_{i} g( {\tilde h}_2 ( P_3 \nabla_{ {\cal E}_i } Y,X) , {\cal E}_i ) \nonumber\\
&& -\sum\nolimits_\mu g( h_{12}( P_2 \nabla_{e_\mu}  X , Y) , e_\mu ) - \sum\nolimits_\mu g( h_{12}( X , P_2 \nabla_{e_\mu} Y) , e_\mu ) \nonumber\\
&& +\,2\sum\nolimits_{\,a,i} g( T_1( Y, E_a ) , {\cal E}_i )( h_1( E_a, X ) , {\cal E}_i )
 +2\sum\nolimits_{\,a,i} g( T_1( X, E_a ) , {\cal E}_i )( h_1( E_a, Y ) , {\cal E}_i ) \nonumber\\
&& -\sum\nolimits_{\,i,\mu} g( {\tilde T}_3 (Y, {\cal E}_i ) , e_\mu )( {\tilde h}_3( {\cal E}_i , X) , e_\mu )
 -\sum\nolimits_{\,i,\mu} g( {\tilde T}_3 (X, {\cal E}_i ) , e_\mu )({\tilde h}_3( {\cal E}_i , Y) , e_\mu ) \nonumber\\
&& +\sum\nolimits_{i,\mu} g( {\tilde h}_2 ( X , e_\mu ) , {\cal E}_i )\,g( {\cal E}_i , {\tilde T}_{2} ( Y , e_\mu ) )
 +\sum\nolimits_{i,\mu} g( {\tilde h}_2 ( Y , e_\mu ) , {\cal E}_i )\,g( {\cal E}_i , {\tilde T}_{2} ( X , e_\mu ) ) \nonumber\\
&& -\,2\,g(X, H_2)\,g(Y, H_2) -g(X, H_2)\,g(Y,H_3) -g(X, H_3)\,g(Y, H_2) \nonumber\\
&& +\,2\sum\nolimits_{i,j} g( {\tilde h}_1 ({\cal E}_i , {\cal E}_j ) , X  )\,g( {\tilde h}_1 ({\cal E}_i , {\cal E}_j ) , Y )
 +2\sum\nolimits_{i,\mu} g( {\tilde h}_1 ({\cal E}_i , e_\mu ) , X  )\,g( {\tilde h}_1 ({\cal E}_i , e_\mu ) , Y ) \nonumber\\
&& +\,2\sum\nolimits_{i,j} g( {\tilde T}_1 ({\cal E}_i , {\cal E}_j ) , X  )\,g( {\tilde T}_1 ({\cal E}_i , {\cal E}_j ) , Y )
 +2\sum\nolimits_{i,\mu} g( {\tilde T}_1 ({\cal E}_i , e_\mu ) , X  )\,g( {\tilde T}_1 ({\cal E}_i , e_\mu ) , Y ) \nonumber\\
&& -\,4\sum\nolimits_{i,a} g( T_{1}( X, E_a ) , {\cal E}_i )\,g({\cal E}_i, T_{1} (Y, E_a) )
+2\sum\nolimits_{\mu, j} g( {\tilde T}_3 ( X, {\cal E}_j ) , e_\mu )\,g(e_\mu, {\tilde T}_3 ( Y, {\cal E}_j ) ) \nonumber\\
&& -\,2\sum\nolimits_{i,\mu} g( {\tilde T}_2 ( X,  e_\mu ) ,{\cal E}_i  )\,g( {\cal E}_i  , {\tilde T}_{2} ( Y , e_\mu ) )
+S_{\,\mD_1, \mD_2}\cdot g(X,Y) \nonumber\\
&& +\,\Div(
 P_2 H_1 + P_2 H_3 - H_2 )\,g(X,Y) =0,\quad
 X,Y \in \mD_1.
\end{eqnarray}
Using
$h_{12} ( P_2 \nabla_{ e_\mu } X,Y) = \mh_{12} ( P_2 \nabla_{ e_\mu } X,Y)$ and
\[
P_1 \nabla_{ e_\mu } {\cal E}_i = h_{23} ({\cal E}_i ,  e_\mu ) - T_{23} ({\cal E}_i ,  e_\mu )  = \mh_{23} ({\cal E}_i ,  e_\mu ) - \mT_{23} ({\cal E}_i ,  e_\mu ),
\]
we obtain
\begin{eqnarray*}
 && \sum\nolimits_{\mu} g( h_{12} ( P_2 \nabla_{ e_\mu } X,Y) , e_\mu ) =
\sum\nolimits_{i,\mu} g( \mA_{12,\mu} Y , {\cal E}_i )\, g( {\cal E}_i , P_2 \nabla_{ e_\mu } X ) \\
&&= -\!\sum\nolimits_{i,\mu} g( \mA_{12,\mu} {\cal E}_i , Y )\,g( \nabla_{ e_\mu }{\cal E}_i, X )
 = -\sum\nolimits_{i,\mu} g( \mA_{12,\mu} {\cal E}_i, Y )\,g( \mh_{23} ({\cal E}_i,  e_\mu ) - \mT_{23} ({\cal E}_i,  e_\mu ), X ) .
\end{eqnarray*}
Similarly,
\begin{eqnarray*}
 && \sum\nolimits_{i} g( {\tilde h}_2 ( P_3 \nabla_{ {\cal E}_i } X,Y) , {\cal E}_i )
 = \sum\nolimits_{i} g( h_{13} ( P_3 \nabla_{ {\cal E}_i } X,Y) , {\cal E}_i ) \\
 &&\quad = - \sum\nolimits_{i,\mu} g( \mA_{13,i} e_\mu , Y )\,g( \mh_{23} ({\cal E}_i ,  e_\mu ) + \mT_{23} ({\cal E}_i ,  e_\mu ) , X ) ,\\
 && \sum\nolimits_{\,a,i} g( T_1( Y, E_a ) , {\cal E}_i )\,g( h_1( E_a, X ) , {\cal E}_i )
 = -\sum\nolimits_{\,a,i} g( \Ts_{1,i} E_a , Y )\,g( A_{1,i} E_a, X ) ,\\
 && -\sum\nolimits_{\,i,\mu} g( {\tilde T}_3 (Y, {\cal E}_i ) , e_\mu )\,g( {\tilde h}_3( {\cal E}_i , X) , e_\mu )
 =-\sum\nolimits_{\,i,\mu} g( \mT_{12} (Y, {\cal E}_i ) , e_\mu )\,g( \mh_{12} ( {\cal E}_i , X) , e_\mu ) \\
 &&\quad = -\sum\nolimits_{\,i,\mu} g( \mTs_{12,\mu} Y, {\cal E}_i )\,g( \mA_{12, \mu} {\cal E}_i , X)
 = \sum\nolimits_{\,i,\mu} g( \mTs_{12,\mu} {\cal E}_i , Y )\,g( \mA_{12, \mu} {\cal E}_i , X) ,
\end{eqnarray*}
\begin{eqnarray*}
&&\sum\nolimits_{i,\mu} g( {\tilde h}_2 ( X , e_\mu ) , {\cal E}_i )\,g( {\cal E}_i , {\tilde T}_{2} ( Y , e_\mu ) )
 = \sum\nolimits_{i,\mu} g( \mh_{13} ( X , e_\mu ) , {\cal E}_i )\,g( {\cal E}_i , \mT_{13} ( Y , e_\mu ) ) \\
&&\quad = \sum\nolimits_{i,\mu} g( \mA_{13,i} X , e_\mu )\,g( \mTs_{13,i} Y , e_\mu )
 = - \sum\nolimits_{i,\mu} g( \mA_{13,i} e_\mu , X )\,g( \mTs_{13,i} e_\mu , Y ),\\
&& -\,4\sum\nolimits_{i,a} g( T_{1}( X, E_a ) , {\cal E}_i )\,g({\cal E}_i, T_{1} (Y, E_a) )
= -\,4\sum\nolimits_{i,a} g( \Ts_{1,i} X, E_a )\,g( \Ts_{1,i} Y, E_a ) \\
&&\quad= -\,4\sum\nolimits_{i,a} g( \Ts_{1,i} E_a , X )\,g( \Ts_{1,i} E_a , Y ) , \\
&& 2\sum\nolimits_{\mu, j} g( {\tilde T}_3 ( X, {\cal E}_j ) , e_\mu )\,g(e_\mu, {\tilde T}_3 ( Y, {\cal E}_j ) ) =
2\sum\nolimits_{\mu, j} g( \mT_{12} ( X, {\cal E}_j ) , e_\mu )\,g( e_\mu, \mT_{12} ( Y, {\cal E}_j ) ) \\
&&\quad = 2\sum\nolimits_{\mu, j} g( \mTs_{12,\mu} X, {\cal E}_j )\,g( \mTs_{12,\mu} Y, {\cal E}_j )
= 2\sum\nolimits_{\mu, j} g( \mTs_{12,\mu} {\cal E}_j ,  X )\,g( \mTs_{12,\mu} {\cal E}_j , Y ),\\
 && -\,2\sum\nolimits_{i,\mu} g( {\tilde T}_2 ( X,  e_\mu ) ,{\cal E}_i  )\,g( {\cal E}_i  , {\tilde T}_{2} ( Y , e_\mu ) ) =
-\,2\sum\nolimits_{i,\mu} g( \mT_{13} ( X,  e_\mu ) ,{\cal E}_i  )\,g( {\cal E}_i  , \mT_{13} ( Y , e_\mu ) ) \\
&&\quad -\,2\sum\nolimits_{i,\mu} g( \mTs_{13, i} X,  e_\mu  )\,g( \mTs_{13,i} Y , e_\mu )
 =-\,2\sum\nolimits_{i,\mu} g( \mTs_{13, i} e_\mu , X )\,g( \mTs_{13,i} e_\mu , Y) .
\end{eqnarray*}
Using the above computations, we obtain \eqref{EL11withmixed} as an equivalent form of \eqref{EL11}.
\end{proof}

The Euler-Lagrange equation for variations of metric satisfying
 $B(X,Y) = B(P_2 X, P_2 Y)$
(i.e., when the metric changes only along $\mD_2 \times \mD_2$)
is dual to \eqref{EL11withmixed}, i.e., can be obtained from \eqref{EL11withmixed} by interchanging $\mD_1$ and $\mD_2$ (and the corresponding tensors and vectors from their frames).

\subsection{Variations for $B = B |_{\mD_3 \times \mD_3}$}

\begin{proposition}
Let  variation $g_t$ of $g$ satisfy $B(X,Y) = B(P_3 X, P_3 Y)$ for all $X,Y \in \mathfrak{X}_M$.
Then ${\rm\frac{d}{dt}}\, J_{\mD_1,\mD_2}(g_t) |_{\,t=0}=0$ if and only if the following Euler-Lagrange equation holds for all $X,Y \in \mD_3$:
\begin{eqnarray} \label{EL33withmixed}
&& \sum\nolimits_{\,a,i}\big( -g(\mA_{23, a}{\cal E}_i, Y)\,g( (\mh_{12}
+ \mT_{12})({E_a}, {\cal E}_i ), X)\nonumber\\
&& -\,g( \mA_{23, a} {\cal E}_i , X )\,g( (\mh_{12}
+ \mT_{12}) ( {E_a} , {\cal E}_i ) , Y ) \nonumber\\
&& -\,g( \mA_{13, i} E_a , Y )\,g( (\mh_{12}
- \mT_{12}) ( {E_a} , {\cal E}_i ) , X ) \nonumber\\
&& -\,g( \mA_{13, i} E_a , X )\,g( (\mh_{12}
- \mT_{12}) ( {E_a} , {\cal E}_i ) , Y ) \nonumber\\
&& -\,g( \mA_{23,a} {\cal E}_i , X )\,g( \mTs_{23,a} {\cal E}_i , Y )
 -g( \mA_{23,a} {\cal E}_i , Y )\,g(\mTs_{23,a} {\cal E}_i , X ) \nonumber\\
&& -\,g( \mA_{13, i} E_a , X )\,g( \mTs_{13,i} E_a , Y )
 - g( \mA_{13, i} E_a , Y )\,g( \mTs_{13,i} E_a , X ) \nonumber\\
&& -\,2 g( \mTs_{23,a} {\cal E}_i , X )\,g( \mTs_{23,a} {\cal E}_i , Y )
-2 g( \mTs_{13,i} E_a , X )\,g( \mTs_{13,i} E_a , Y ) \nonumber\\
&& -\,2 g(X, \mh_{12} ( E_a , {\cal E}_i ) )\,g(Y, \mh_{12} ( E_a , {\cal E}_i ) )
 - 2 g(X, \mT_{12} ( E_a , {\cal E}_i ) )\,g(Y, \mT_{12} ( E_a , {\cal E}_i ) )
 \big)\nonumber\\
&& +\,g(P_3 H_2 , X)\,g(P_3 H_1, Y) + g(P_3 H_2 , Y)\,g(P_3 H_1, X)
 +S_{\,\mD_1,\mD_2}\cdot g(X,Y)
 = 0.
\end{eqnarray}
\end{proposition}

\begin{proof}
For $B$ non-zero only on $\mD_3 \times \mD_3$ (i.e., $B = B |_{\mD_3  \times \mD_3}$),
we have in \eqref{dtS1S2general}
only
the following terms:
\begin{eqnarray*}
\nonumber
&& {\rm\frac{d}{dt}}\int_M \big( S_{\,\mD_1 , \mD_1^\perp(t)} +  S_{\,\mD_2 , \mD_2^\perp(t)} - S_{\,\mD_3(t) , \mD_3^\perp}\big)
 \,{\rm d}\vol_{g_t} |_{\,t=0} = \tilde I_1 + \tilde I_3 + \tilde I_5,
 \end{eqnarray*}
where ${\tilde I}_3$, given by \eqref{tildeI3}, and
\begin{eqnarray*}
&& \tilde I_1 =\int_{\Omega}
\big\< -\Div{\tilde h_1}  |_{\mD_{1}^\perp \times \mD_{1}^\perp} - \widetilde{\cal K}_1^\flat
- H_1^\flat \otimes H_1^\flat + \frac{1}{2} \Upsilon_{h_1,h_1}
+ \frac{1}{2} \Upsilon_{T_1,T_1}
+2\,\tilde{\cal T}_1^\flat \nonumber \\
&& +\,\frac{1}{2}\,\big( S_{\,\mD_1 , \mD_1^\perp} + \Div( H_{23} - H_1)\big)\, g_{1}^{\perp},\ B \big\>\,{\rm d}\vol_g ,\,\nonumber \\
&& \tilde I_5 = \int_{\Omega}
\big\< \Div{\tilde h_{12}}  |_{\mD_{12}^\perp \times \mD_{12}^\perp } + \widetilde{\cal K}_{12}^\flat
+ H_{12}^\flat \otimes H_{12}^\flat - \frac{1}{2} \Upsilon_{h_{12},h_{12}}
- \frac{1}{2} \Upsilon_{T_{12},T_{12}}
- 2\,\tilde{\cal T}_{12}^\flat \nonumber \\
&& - \frac{1}{2}\,\big( S_{\,\mD_{12} , \mD_{12}^\perp} + \Div( {\tilde H}_{12} - H_{12})\big)\, g_{3}^{\top},\ B \big\>\,{\rm d}\vol_g,
\end{eqnarray*}
are the only terms of \eqref{I3}, \eqref{I1} and \eqref{I5}, which may be non-vanishing.

In~what follows 
we assume that 
$X,Y \in \mD_3$. 
We have
\begin{eqnarray*}
 && \widetilde{\cal K}_{12}^\flat(X,Y) = \sum\nolimits_{a}[\tT_{12,a},\,\tA_{12,a}]^\flat(X,Y) +\sum\nolimits_{i}[\tT_{12,i},\,\tA_{12,i}]^\flat(X,Y)\\
&&\quad = -\sum\nolimits_{a,\mu} \big(g(  {\tilde h}_{1}( X, e_\mu ) , E_a )\,g( E_a , {\tilde T}_{1} (Y , e_\mu) )
 + g(  {\tilde h}_{1}( Y, e_\mu ) , E_a )\,g( E_a , {\tilde T}_{1} (X , e_\mu) )\big) \\
&&\quad  - \sum\nolimits_{i, \mu} \big( g( {\tilde h}_{2}( X, e_\mu ) , {\cal E}_i )\,g( {\cal E}_i  , {\tilde T}_{2} ( Y , e_\mu) )
 + g( {\tilde h}_{2}( Y, e_\mu ) , {\cal E}_i )\,g( {\cal E}_i  , {\tilde T}_{2} ( X , e_\mu) ) \big).
\end{eqnarray*}
On the other hand, we have
\begin{eqnarray*}
 && -\widetilde{\cal K}_1^\flat (X,Y) = - \sum\nolimits_a g( [ \tT_{1,a}, \, \tA_{1,a}] X, \,Y  ) \\
&&\quad =\sum\nolimits_{a,i} \big( g( {\tilde h}_1 ( X , {\cal E}_i ) , E_a )\,g(E_a , {\tilde T}_{1} ( Y , {\cal E}_i) )
  + g( {\tilde h}_1 ( Y , {\cal E}_i ) , E_a )\,g(E_a , {\tilde T}_{1} ( X , {\cal E}_i) ) \big)\\
&&\quad +\sum\nolimits_{a,\mu} \big( g( {\tilde h}_1 ( X , e_\mu ) , E_a )\,g(E_a , {\tilde T}_{1} ( Y , e_\mu ) )
 + g( {\tilde h}_1 ( Y , e_\mu ) , E_a )\,g(E_a , {\tilde T}_{1} ( X , e_\mu ) ) \big),\\
 && - \widetilde{\cal K}_2^\flat (X,Y) =
 \sum\nolimits_{a,i} \big( g( {\tilde h}_2 ( X , E_a ) , {\cal E}_i )\,g( {\cal E}_i , {\tilde T}_{2} ( Y , E_a ) )
 + g( {\tilde h}_2 ( Y , E_a ) , {\cal E}_i )\,g( {\cal E}_i , {\tilde T}_{2} ( X , E_a ) ) \big)\\
&&\quad
 +\sum\nolimits_{i,\mu} \big( g( {\tilde h}_2 ( X , e_\mu ) , {\cal E}_i )\,g( {\cal E}_i , {\tilde T}_{2} ( Y , e_\mu ) )
 + g( {\tilde h}_2 ( Y , e_\mu ) , {\cal E}_i )\,g( {\cal E}_i , {\tilde T}_{2} ( X , e_\mu ) ) \big).
\end{eqnarray*}
Hence,
\begin{eqnarray*}
 && \widetilde{\cal K}_{12}^\flat(X,Y) -\widetilde{\cal K}_1^\flat (X,Y)- \widetilde{\cal K}_2^\flat (X,Y)
 =\sum\nolimits_{a,i} \big( g( {\tilde h}_1 ( X , {\cal E}_i ) , E_a )\,g(E_a , {\tilde T}_{1} ( Y , {\cal E}_i) ) \\
&& \quad +\,g( {\tilde h}_1 ( Y , {\cal E}_i ) , E_a )\,g(E_a , {\tilde T}_{1} ( X , {\cal E}_i) )
 + g( {\tilde h}_2 ( X , E_a ) , {\cal E}_i )\,g( {\cal E}_i , {\tilde T}_{2} ( Y , E_a ) ) \\
&& \quad +\,g( {\tilde h}_2 ( Y , E_a ) , {\cal E}_i )\,g( {\cal E}_i , {\tilde T}_{2} ( X , E_a ) )\big) .
\end{eqnarray*}
Also we get
\begin{eqnarray*}
&& ( H_{12}^\flat \otimes H_{12}^\flat  -H_1^\flat \otimes H_1^\flat - H_2^\flat \otimes H_2^\flat )(X,Y) \\
&& = ( (P_3 H_1 + P_3 H_2)^\flat \otimes (P_3 H_1 + P_3 H_2)^\flat
 - H_1^\flat \otimes H_1^\flat - H_2^\flat \otimes H_2^\flat )(X,Y)
\\
&& = ( P_3 H_2 \otimes P_3 H_1 + P_3 H_1 \otimes P_3 H_2 )(X,Y),\\
 && \Upsilon_{h_{12},h_{12}} (X,Y) = 2\sum\nolimits_{a,b} g(X, h_1 (E_a, E_b) )\,g(Y, h_1 (E_a, E_b) ) \\
&&\quad +\,2\sum\nolimits_{i,j} g(X, h_2 ( {\cal E}_i , {\cal E}_j ) )\,g(Y, h_2 ( {\cal E}_i , {\cal E}_j ) )
 +4\sum\nolimits_{a,i} g(X, h_{12} ( {\cal E}_i , E_a ) )\,g(Y, h_{12} ( {\cal E}_i , E_a ) ),\\
 && \Upsilon_{h_{1},h_{1}} (X,Y) = 2\sum\nolimits_{a,b} g(X, h_1 (E_a, E_b) )\,g(Y, h_1 (E_a, E_b) ),\\
 && \Upsilon_{h_{2},h_{2}} (X,Y) = 2\sum\nolimits_{i,j} g(X, h_2 ( {\cal E}_i , {\cal E}_j ) )\,g(Y, h_2 ( {\cal E}_i , {\cal E}_j ) ).
\end{eqnarray*}
Hence,
 $( \Upsilon_{h_{1},h_{1}} + \Upsilon_{h_{2},h_{2}} - \Upsilon_{h_{12},h_{12}} )(X,Y) =
 -4\sum\nolimits_{\,a,i} g( h_{12}({\cal E}_i , E_a ) , X)g( h_{12} ( {\cal E}_i , E_a ) , Y)$.
Similarly,
\begin{eqnarray*}
&& ( \Upsilon_{T_{1},T_{1}} + \Upsilon_{T_{2},T_{2}} - \Upsilon_{T_{12},T_{12}} )(X,Y)
 = - 4\sum\nolimits_{a,i} g( T_{12} ( {\cal E}_i , E_a ) ,X) g( T_{12} ( {\cal E}_i , E_a ) , Y),
\end{eqnarray*}
and
\begin{eqnarray*}
&& -(\Div {\tilde h}_1 ) (X,Y) =
- \sum\nolimits_{\,k} g(  \nabla_{\xi_k} ( P_1 {\tilde h}_{12}(X,Y) ) , \xi_k ) \\
&&\quad +\sum\nolimits_{\,a} \big( g( {\tilde h}_{1}( P_1^\perp \nabla_{E_a} X, Y ) , E_a )
 + g( {\tilde h}_{1}( P_1^\perp \nabla_{E_a} Y, X ) , E_a ) \big),\\
 && -(\Div {\tilde h}_2 ) (X,Y) =
 -\sum\nolimits_{\,k} g(  \nabla_{\xi_k} ( P_2 {\tilde h}_{12}(X,Y) ) , \xi_k ) \\
&&\quad +\sum\nolimits_{\,i} \big( g( {\tilde h}_{2}( P_2^\perp \nabla_{ {\cal E}_i } X, Y ) ,  {\cal E}_i )
 + g( {\tilde h}_{2}( P_2^\perp \nabla_{ {\cal E}_i } Y, X ) ,  {\cal E}_i ) \big),\\
 && ( \Div {\tilde h}_{12} )(X,Y)
 =
 \sum\nolimits_{\,k} g(  \nabla_{\xi_k} ( (P_{12} {\tilde h}_{12}(X,Y) ) , \xi_k )
 - \sum\nolimits_{\,i} g( {\tilde h}_{2}( P_3 \nabla_{ {\cal E}_i } X, Y ) ,  {\cal E}_i ) \\
&&\quad
 - \sum\nolimits_{\,i} g( {\tilde h}_{2}( P_3 \nabla_{ {\cal E}_i } Y, X ) ,  {\cal E}_i )
 - \sum\nolimits_{\,a} g( {\tilde h}_{1}( P_3 \nabla_{ E_a } X, Y ) ,  E_a )
 - \sum\nolimits_{\,a} g( {\tilde h}_{1}( P_3 \nabla_{ E_a } Y, X ) ,  E_a ).
\end{eqnarray*}
Hence,
\begin{eqnarray*}
 &&  (\Div {\tilde h}_{12}  - \Div {\tilde h}_1 - \Div {\tilde h}_2 )(X,Y)
 =\sum\nolimits_{\,i} g( {\tilde h}_{2}( P_1 \nabla_{ {\cal E}_i } X, Y ) ,  {\cal E}_i )
  +\sum\nolimits_{\,i} g( {\tilde h}_{2}( P_1 \nabla_{ {\cal E}_i } Y, X ) ,  {\cal E}_i )\\
&&\quad +\sum\nolimits_{\,a} g( {\tilde h}_{1}( P_2 \nabla_{ E_a } X, Y ) ,  E_a )
 +\sum\nolimits_{\,a} g( {\tilde h}_{1}( P_2 \nabla_{ E_a } Y, X ) ,  E_a ).
\end{eqnarray*}
We have
 \\
\begin{eqnarray*}
 && - \tilde{\cal T}_{12}^\flat (X,Y) = -\sum\nolimits_{\,a} g( \tT_{12,a} \tT_{12,a} X,Y) -\sum\nolimits_{\,i} g( \tT_{12,i} \tT_{12,i} X,Y) \\
&&\quad =\sum\nolimits_{a,\mu} g( T_3 (X,  e_\mu ) , E_a )\,g( E_a, T_3( Y , e_\mu ) )
 +\sum\nolimits_{i,\mu} g( T_3 (X,  e_\mu ) , {\cal E}_i )\,g( {\cal E}_i , T_3( Y , e_\mu ) ) ,\\
 && \tilde{\cal T}_{1}^\flat (X,Y) = \sum\nolimits_{\,a} g( \tT_{1,a} \tT_{1,a} X,Y) \\
 &&\quad =-\sum\nolimits_{a,\mu} g( T_3 ( X,  e_\mu ) , E_a )\,g( E_a , T_3 ( Y , e_\mu ) )
- \sum\nolimits_{a,i} g( {\tilde T}_1 ( X,  {\cal E}_i ) , E_a )\,g( E_a , {\tilde T}_{1} ( Y , {\cal E}_i ) ), \\
 && \tilde{\cal T}_{2}^\flat (X,Y) = -\sum\nolimits_{i,\mu} g( T_3 ( X,  e_\mu ) , {\cal E}_i )\,g( {\cal E}_i , T_3 ( Y , e_\mu ) )
 -\sum\nolimits_{a,i} g( {\tilde T}_2 ( X,  E_a ) , {\cal E}_i )\,g( {\cal E}_i , {\tilde T}_{2} ( Y , E_a ) ) .
\end{eqnarray*}
Hence,
\begin{eqnarray*}
&&\quad (\tilde{\cal T}_{1}^\flat + \tilde{\cal T}_{2}^\flat -\tilde{\cal T}_{12}^\flat)(X,Y) =  \\
&& -\sum\nolimits_{a,i} \big( g( {\tilde T}_1 ( X,  {\cal E}_i ) , E_a )\,g( E_a , {\tilde T}_{1} ( Y , {\cal E}_i ) )
 + g( {\tilde T}_2 ( X,  E_a ) , {\cal E}_i )\,g( {\cal E}_i , {\tilde T}_{2} ( Y , E_a ) ) \big).
\end{eqnarray*}
We also have
$\Div \big( H_{23} - H_1 + H_{13} - H_2 - {\tilde H}_{12} + H_{12} \big)
 = 0 $.

By the above, the Euler-Lagrange equation for variations of metric satisfying
 $B(X,Y) = B(P_3 X , P_3 Y)$
is the following, for all $X,Y \in \mD_3$:
\begin{eqnarray}
\label{EL33}
&& \sum\nolimits_a g( {\tilde h}_{1}( P_2 \nabla_{ E_a } X, Y ) ,  E_a )
 +\sum\nolimits_i \big( g( {\tilde h}_{2}( P_1 \nabla_{ {\cal E}_i } X, Y ) ,  {\cal E}_i )
 +g( {\tilde h}_{2}( P_1 \nabla_{ {\cal E}_i } Y, X ) ,  {\cal E}_i ) \nonumber\\
&& +\, g( {\tilde h}_{1}( P_2 \nabla_{ E_a } Y, X ) ,  E_a )\big) +2\,{\rm Sym}(P_3 H_2 \otimes P_3 H_1) (X,Y)
 \nonumber\\
&& +\sum\nolimits_{a,i} \big( g( {\tilde h}_1 ( X , {\cal E}_i ) , E_a )\,g(E_a , {\tilde T}_{1} ( Y , {\cal E}_i) )
 + g( {\tilde h}_1 ( Y , {\cal E}_i ) , E_a )\,g(E_a , {\tilde T}_{1} ( X , {\cal E}_i) ) \nonumber\\
&& +\,g( {\tilde h}_2 ( X , E_a ) , {\cal E}_i )\,g( {\cal E}_i , {\tilde T}_{2} ( Y , E_a ) )
+ g( {\tilde h}_2 ( Y , E_a ) , {\cal E}_i )\,g( {\cal E}_i , {\tilde T}_{2} ( X , E_a ) ) \nonumber\\
&& -\,2 g(X, h_{12} ( {\cal E}_i , E_a ) )\,g(Y, h_{12} ( {\cal E}_i , E_a ) )
 - 2 g(X, T_{12} ( {\cal E}_i , E_a ) )\,g(Y, T_{12} ( {\cal E}_i , E_a ) ) \nonumber\\
&& -\,2 g( {\tilde T}_1 ( X,  {\cal E}_i ) , E_a )\,g( E_a , {\tilde T}_{1} ( Y , {\cal E}_i ) )
 -2 g( {\tilde T}_2 ( X,  E_a ) , {\cal E}_i )\, g( {\cal E}_i , {\tilde T}_{2} ( Y , E_a ) )\big)  \nonumber\\
&& + S_{\,\mD_1,\mD_2}\cdot g(X,Y)
 = 0.
\end{eqnarray}
We have
\begin{eqnarray*}
&& \sum\nolimits_{\,a} g( {\tilde h}_1 ( P_2 \nabla_{E_a} X, Y ) , E_a )
= \sum\nolimits_{\,a} g( \mh_{23} ( P_2 \nabla_{E_a} X, Y ) , E_a ) \\
&&= \sum\nolimits_{\,a,i} g( \mA_{23, a} Y , {\cal E}_i )\,g( {\cal E}_i , P_2 \nabla_{E_a} X )
 = - \sum\nolimits_{\,a,i} g( \mA_{23, a} Y , {\cal E}_i )\,g( \nabla_{E_a} {\cal E}_i , X ) \\
&&= - \sum\nolimits_{\,a,i} g( \mA_{23, a} {\cal E}_i , Y )\,g( \mh_{12} ( {E_a} , {\cal E}_i ) + \mT_{12} ( {E_a} , {\cal E}_i ) , X ) ,\\
&& \sum\nolimits_{\,i} g( {\tilde h}_2 ( P_1 \nabla_{ {\cal E}_i } X, Y ) , {\cal E}_i ) =
- \sum\nolimits_{\,a,i} g( \mA_{13, i} E_a , Y )\,g( \mh_{12} ( E_a , {\cal E}_i ) - \mT_{12} ( E_a , {\cal E}_i ) , X ) , \\
 && \sum\nolimits_{\,a,i} g( {\tilde h}_1 ( X , {\cal E}_i ) , E_a )\,g(E_a , {\tilde T}_{1} ( Y , {\cal E}_i) ) =
\sum\nolimits_{\,a,i} g( \mh_{23} ( X , {\cal E}_i ) , E_a )\,g(E_a , \mT_{23} ( Y , {\cal E}_i) ) \\
&&= \sum\nolimits_{\,a,i} g( \mA_{23,a} X , {\cal E}_i )\,g(\mTs_{23,a} Y , {\cal E}_i )
= - \sum\nolimits_{\,a,i} g( \mA_{23,a} {\cal E}_i , X )\,g(\mTs_{23,a} {\cal E}_i , Y ) ,\\
 && \sum\nolimits_{\,a,i} g( {\tilde h}_2 ( X , E_a ) , {\cal E}_i )\,g( {\cal E}_i , {\tilde T}_{2} ( Y , E_a ) ) =
\sum\nolimits_{\,a,i} g( \mh_{13} ( X , E_a ) , {\cal E}_i )\,g( {\cal E}_i , \mT_{13} ( Y , E_a ) ) \\
&&= \sum\nolimits_{\,a,i} g( \mA_{13, i} X , E_a )\,g( \mTs_{13,i} Y , E_a )
 = -\sum\nolimits_{\,a,i} g( \mA_{13, i} E_a , X )\,g( \mTs_{13,i} E_a , Y ) ,\\
 && \sum\nolimits_{\,a,i} g( {\tilde T}_1 ( X,  {\cal E}_i ) , E_a )\,g( E_a , {\tilde T}_{1} ( Y , {\cal E}_i ) ) =
\sum\nolimits_{\,a,i} g( \mT_{23} ( X,  {\cal E}_i ) , E_a )\,g( E_a , \mT_{23} ( Y , {\cal E}_i ) ) \\
&&= \sum\nolimits_{\,a,i} g( \mTs_{23,a} X,  {\cal E}_i )\,g( \mTs_{23,a} Y , {\cal E}_i )
 = \sum\nolimits_{\,a,i} g( \mTs_{23,a} {\cal E}_i , X )\,g( \mTs_{23,a} {\cal E}_i , Y ) ,\\
 && \sum\nolimits_{\,a,i} g( {\tilde T}_2 ( X,  E_a ) , {\cal E}_i )\,g( {\cal E}_i , {\tilde T}_{2} ( Y , E_a ) ) =
\sum\nolimits_{\,a,i} g( \mTs_{13,i} E_a , X )\,g( \mTs_{13,i} E_a , Y ) .
\end{eqnarray*}
Using the above, we obtain \eqref{EL33withmixed} as an equivalent form of \eqref{EL33}. 
\end{proof}

\subsection{Variations for $B = B |_{ (\mD_1 \times \mD_3) \cup (\mD_3 \times \mD_1) }$}

\begin{proposition}
Let $g_t$ be a variation of $g$ such that
\begin{equation}\label{Eq-prop6}
 B(X,Y) = B(P_1 X, P_3 Y) + B(P_3 X , P_1 Y)\quad (X,Y \in \mathfrak{X}_M).
\end{equation}
Then ${\rm\frac{d}{dt}}\, J_{\mD_1,\mD_2}(g_t) |_{\,t=0}=0$ if and only if the following Euler-Lagrange equation holds for all $X \in \mD_1$, $Y \in \mD_3$:
\begin{eqnarray} \label{EL13withmixed}
&& 2\sum\nolimits_{a,i} g(Y , \mA_{23 , E_a} {\cal E}_i )\,g(X,  \T_{1 ,{\cal E}_i }  E_a )
 +\,2\,g( T^\sharp_{1, Y}  X , P_1 (H_2 + H_3) 
 ) + (\Div A_{1,Y}) X - (\Div \tT_{1,X}) Y \nonumber \\
&& +\sum\nolimits_{a,i} g(Y , \mA_{23 , E_a} {\cal E}_i )\,g(X,  {A}_{1 ,{\cal E}_i }  E_a )
 +\,g( {\tilde T}^\sharp_{1,X} Y - {\tilde A}_{1, X} Y , H_1) + g(H_1 , Y)\,g( P_1 (H_2 + H_3)  
 , X)
 \nonumber \\
&&
 - g( \nabla_{X} H_1 , Y)
 -(\Div {\tilde h}_2) (X,Y)  + \sum\nolimits_{a,i} g(Y , \mTs_{23 , E_a} {\cal E}_i )\,g(X,  \T_{1 ,{\cal E}_i }  E_a )
  \nonumber \\
&& - \sum\nolimits_{a,i} g( A_{1, {\cal E}_i } E_a , X )\,g( \mTs_{13 , {\cal E}_i } E_a , Y )
- \sum\nolimits_{i,\mu} g( \mA_{13 , {\cal E}_i } X , e_\mu )\,g( \T_{3 , {\cal E}_i } e_\mu , Y ) \nonumber\\
&& - \sum\nolimits_{a,i} g( \mA_{13 , {\cal E}_i} E_a , Y )\,g( \T_{1, {\cal E}_i } E_a , X )
- \sum\nolimits_{i,\mu} g( A_{3, {\cal E}_i } e_\mu , Y )\,g( \mTs_{13 , {\cal E}_i} e_\mu , X ) \nonumber \\
%
&&  -\,g(H_2 , X)\,g(H_2 ,Y)
 +\sum\nolimits_{i,j} g(X, h_2 ( {\cal E}_i , {\cal E}_j ) )\,g(Y, h_2 ( {\cal E}_i , {\cal E}_j ) ) \nonumber\\
&& +\sum\nolimits_{i,j} g(X, T_2 ( {\cal E}_i , {\cal E}_j ) )\,g(Y, T_2 ( {\cal E}_i , {\cal E}_j ) )
 -2\sum\nolimits_{i,\mu} g( \mTs_{13 , {\cal E}_i } e_\mu , X )\,g( \T_{3, {\cal E}_i} e_\mu , Y ) \nonumber\\
&&
- 2 \sum\nolimits_{a,i} g( \T_{1 , {\cal E}_i } E_a , X )\,g( \mTs_{13 , {\cal E}_i} E_a , Y )
%
-2 \sum\nolimits_{i,\mu} g( A_{3 , {\cal E}_i } e_\mu , Y )\,g( \mTs_{12 , e_\mu } {\cal E}_i , X ) \nonumber\\
&& -\,2\,g(Y, {\tilde T}_3 (X, H_3) ) - (\Div A_{12,Y}) X + (\Div \tT_{12,X} ) Y
%
%
 - \sum\nolimits_{i,\mu} g( A_{3 , {\cal E}_i } e_\mu , Y )\,g( \mA_{12 , e_\mu } {\cal E}_i , X ) \nonumber\\
&& -\,g( T_3 ( Y,  P_3 (H_1 + H_2) 
 ) , X) + g( h_3 ( Y ,  P_3 (H_1 + H_2) 
 ) , X) \nonumber\\
&& -g( P_3 (H_1 + H_2) 
 , Y)\,g(H_3 , X) + g( \nabla_{X} (P_3 (H_1 + H_2)) 
 , Y ) \nonumber\\
&& 
 - \sum\nolimits_{i,\mu} g(  \T_{3, {\cal E}_i } e_\mu , Y )\,g( \mTs_{12, e_\mu } {\cal E}_i , X ) = 0.
\end{eqnarray}
\end{proposition}

\begin{proof}
For $B(X,Y)=B(P_1 X , P_3 Y) + B(P_3 X , P_1 Y)$, we obtain from \eqref{dtS1S2general}
the following terms:
\begin{eqnarray*}
\nonumber
&& {\rm\frac{d}{dt}}\int_M(S_{\,\mD_1,\mD_1^\perp(t)} +S_{\,\mD_2,\mD_2^\perp(t)} -S_{\,\mD_3(t),\mD_3^\perp})\,{\rm d}\vol_{g_t}|_{\,t=0}
 = \hat I_1 + \hat I_3 + \hat I_5\,,
 \end{eqnarray*}
 where
 \begin{eqnarray*}
&& \hat I_1 = \int_{\Omega}
\big\<4\Upsilon_{{\tilde\alpha}_1, \theta_1}
 +4\,\<\theta_1,\, H_{23} \>  +\,2 ( \Div(\alpha_1 -\tilde\theta_1) )_{ \,| { (\mD_1 \times \mD_1^\perp ) \cup (\mD_1^\perp \times \mD_1) } }
  +2\Upsilon_{\alpha_1, {\tilde\alpha}_1 +{\tilde \theta}_1}
  \nonumber \\
&& +\,2\<{\tilde \theta}_1
- {\tilde\alpha}_1, H_1 \> +2\,{\rm Sym}(H_1^{\flat} \otimes H_{23}^{\flat})
 -2\,{\tilde \delta}_{1,H_1} +2\Upsilon_{{\tilde \theta}_1, \theta_1 - \alpha_1},\ B\big\>\,{\rm d}\vol_g \,,\nonumber \\
&& \hat I_5 = \int_{\Omega}
\big\< - 4\Upsilon_{{\tilde\alpha}_{12}, \theta_{12}}
 -4\,\<\theta_{12},\, {\tilde H}_{12} \>  - \,2 ( \Div(\alpha_{12} -\tilde\theta_{12}) )_{ \,| ({\mD_{12} \times \mD_{12}^\perp ) \cup ( \mD_{12}^\perp \times \mD_{12} ) } }
 - 2 \Upsilon_{\alpha_{12}, {\tilde\alpha}_{12} + {\tilde \theta}_{12}} \nonumber \\
&&  -\,2\<{\tilde \theta}_{12} - {\tilde\alpha}_{12}, H_{12} \> - 2\,{\rm Sym}(H_{12}^{\flat} \otimes {\tilde H}_{12}^{\flat})
 +2\,{\tilde \delta}_{12, H_{12}} -2\Upsilon_{{\tilde \theta}_{12}, \theta_{12} - \alpha_{12}} ,\ B\big\>\,{\rm d}\vol_g, \\
 && \hat I_3 = \int_{\Omega}\big\< -\Div{\tilde h_2} |_{\mD_{2}^\perp \times \mD_{2}^\perp} -\widetilde{\cal K}_2^\flat
 - H_2^\flat \otimes H_2^\flat
 +\frac{1}{2} \Upsilon_{h_2,h_2} + \frac{1}{2} \Upsilon_{T_2,T_2} +2\,\tilde{\cal T}_2^\flat,\ B\big\>\,{\rm d}\vol_g \,
\end{eqnarray*}
are the only 
terms of \eqref{I1}, \eqref{I3} and \eqref{I5}, which may be non-zero.

In what follows, let $X \in \mD_1$ and $Y \in \mD_3$, then we get
\begin{eqnarray*}
&& 2\Upsilon_{{\tilde\alpha}_1 , {\alpha}_1}(X,Y)
 = \sum\nolimits_{a,i}  g( E_a , {\tilde h}_1( {\cal E}_i , Y) )\,g( {\cal E}_i , h_1(E_a , X) )
 \\ &&
 +\sum\nolimits_{a,\mu}  g( E_a , {\tilde h}_1( e_\mu , Y) )\,g(  e_\mu , h_1(E_a , X) ) , \\
 && 2\Upsilon_{{\tilde\alpha}_1, \theta_1} = \sum\nolimits_{a,i} g(Y , {\tilde A}_{1 , E_a} {\cal E}_i )\,g(X,  {T}_{1 ,{\cal E}_i }  E_a )
 +\sum\nolimits_{a,\mu} g(Y ,{\tilde A}_{ 1 , E_a} e_\mu )\,g( X,  T_{1 , e_\mu} E_a  ),\\
 && 2\Upsilon_{ {\tilde \alpha}_1 , {\alpha}_1  }(X,Y) =
\sum\nolimits_{a,i} g(Y , {\tilde A}_{1 , E_a} {\cal E}_i )\,g(X,  {A}_{1 ,{\cal E}_i }  E_a )
 +\sum\nolimits_{a,\mu} g(Y ,{\tilde A}_{ 1 , E_a} e_\mu )\,g( X,  A_{1 , e_\mu} E_a  ),\\
 && 2\Upsilon_{ {\tilde \theta}_1 , {\theta}_1  }(X,Y) =
\sum\nolimits_{a,i} g(Y , {\tT}_{1 , E_a} {\cal E}_i )\,g(X,  \T_{1 ,{\cal E}_i }  E_a )
 +\sum\nolimits_{a,\mu} g(Y ,{\tT}_{ 1 , E_a} e_\mu )\,g( X,  \T_{1 , e_\mu} E_a  ),
\end{eqnarray*}
\begin{eqnarray*}
 && 2\, {\tilde \delta}_{1,H_1} (X,Y) = g( \nabla_{X} H_1 , Y), \\
&& 2\,\<\theta_1,\, H_{23} \> (X,Y) = g( T^\sharp_{1, Y}  X , H_{23} ),\\
 && 2\<{\tilde \theta}_1 - {\tilde\alpha}_1, H_1 \> (X,Y) = g( {\tilde T}^\sharp_{1,X} Y - {\tilde A}_{1, X} Y , H_1),\\
 && 2\,{\rm Sym}(H_1^{\flat} \otimes H_{23}^{\flat})(X,Y) = g(H_1 , Y)\, g( H_{23} , X), \\
 && -\widetilde{\cal K}_2^\flat (X,Y) =  \sum\nolimits_{a,i} g( {\tilde h}_2 ( X , E_a ) , {\cal E}_i )\,g( {\cal E}_i , {\tilde T}_{2} ( Y , E_a ) )
 +\sum\nolimits_{i,\mu} g( {\tilde h}_2 ( X , e_\mu ) , {\cal E}_i )\,g( {\cal E}_i , {\tilde T}_{2} ( Y , e_\mu ) ) \\
&& + \sum\nolimits_{a,i} g( {\tilde h}_2 ( Y , E_a ) , {\cal E}_i )\,g( {\cal E}_i , {\tilde T}_{2} ( X , E_a ) )
 +\sum\nolimits_{i,\mu} g( {\tilde h}_2 ( Y , e_\mu ) , {\cal E}_i )\,g( {\cal E}_i , {\tilde T}_{2} ( X , e_\mu ) ) , \\
 && (H_2^\flat \otimes H_2^\flat) (X,Y) = g(H_2 , X)\,g(H_2 ,Y),\\
 && \frac{1}{2} \Upsilon_{h_{2},h_{2}} (X,Y) = \sum\nolimits_{i,j} g(X, h_2 ( {\cal E}_i , {\cal E}_j ) )\,g(Y, h_2 ( {\cal E}_i , {\cal E}_j ) ) , \\
 && \Upsilon_{T_{2},T_{2}} (X,Y) = 2\sum\nolimits_{\,i,j} g(X, T_2 ( {\cal E}_i , {\cal E}_j ) )\,g(Y, T_2 ( {\cal E}_i , {\cal E}_j ) ),\\
 && \tilde{\cal T}_{2}^\flat (X,Y) = \sum\nolimits_{\,i} g( \tT_{2,i} \tT_{2,i} X,Y) \\
&& = -\sum\nolimits_{i,\mu} g( {\tilde T}_2 (X,  e_\mu ) , {\cal E}_i )\,g({\cal E}_i , T_3( Y , e_\mu) )
 -\sum\nolimits_{a,i} g( T_1 ( X,  E_a ) , {\cal E}_i )\,g( {\cal E}_i , {\tilde T}_{2} ( Y , E_a ) ) , \\
 && 2\< \theta_{12} , {\tilde H}_{12} \> (X,Y) = g(H_3 , \T_{12, Y} X ) = g( Y , T_{12} (X, H_3) ) = g(Y, {\tilde T}_3 (X, H_3) ),\\
 && 2\,{\rm Sym} (H_{12} \otimes {\tilde H}_{12})(X,Y) = g( H_{12} , Y)\,g(H_3 , X),\\
 && 2\, {\tilde \delta}_{12, H_{12}}(X,Y) = g( \nabla_{X} H_{12}  , Y ),\\
 && -2\, \< {\tilde \theta} {-} {\tilde \alpha} , H_{12} \>(X,Y) = - g( H_{12} , \tT_{12 , X} Y {-} \tA_{12, X} Y )
  = - g( T_3 ( Y,  H_{12} ) , X) + g( h_3 ( Y , H_{12} ) , X).
\end{eqnarray*}
Similarly,
\begin{eqnarray*}
 && - 2\Upsilon_{{\tilde\alpha}_{12}, \theta_{12}}(X,Y) = -\sum\nolimits_{a,\mu} g(Y , \tA_{12,a} e_\mu )\,g(X, \T_{12,\mu} E_a)
 - \sum\nolimits_{i,\mu} g(Y , \tA_{12,i} e_\mu )\,g(X, \T_{12,\mu} {\cal E}_i ) \\
&&\quad = -\sum\nolimits_{a,\mu} g(  {\tilde h}_1 (e_\mu , Y) , E_a )\,g( {\tilde T}_3 (E_a , X) , e_\mu  )
 -\sum\nolimits_{i,\mu} g(  {\tilde h}_2 (e_\mu , Y) , {\cal E}_i )\,g( {\tilde T}_3 ({\cal E}_i , X) , e_\mu  ),\\
&& -2\Upsilon_{{\tilde\alpha}_{12}, \alpha_{12}} (X,Y) = - \sum\nolimits_{a,\mu} g(Y , \tA_{12,a} e_\mu )\,g(X, A_{12,\mu} E_a) - \sum\nolimits_{i,\mu} g(Y , \tA_{12,i} e_\mu )\,g(X, A_{12,\mu} {\cal E}_i ) \\
&&\quad = -\sum\nolimits_{a,\mu} g(  {\tilde h}_1 (e_\mu , Y) , E_a )\,g( {\tilde h}_3 (E_a , X) , e_\mu  )
- \sum\nolimits_{i,\mu} g(  {\tilde h}_2 (e_\mu , Y) , {\cal E}_i )\,g( {\tilde h}_3 ({\cal E}_i , X) , e_\mu  ) \\
&&\quad = -\sum\nolimits_{a,\mu} g(  h_3 (e_\mu , Y) , E_a )\,g( h_1 (E_a , X) , e_\mu  )
- \sum\nolimits_{i,\mu} g(  h_3 (e_\mu , Y) , {\cal E}_i )\,g( {\tilde h}_3 ({\cal E}_i , X) , e_\mu  ),\\
 && - 2\Upsilon_{{\tilde \theta}_{12}, \theta_{12}} (X,Y) = - \sum\nolimits_{a,\mu} g(Y , \tT_{12,a} e_\mu )\,g(X, \T_{12,\mu} E_a) - \sum\nolimits_{i,\mu} g(Y , \tT_{12,i} e_\mu )\,g(X, \T_{12,\mu} {\cal E}_i ) \\
&&\quad = -\sum\nolimits_{a,\mu} g( T_3 (e_\mu , Y) , E_a )\,g( T_1 (E_a , X) , e_\mu  )
- \sum\nolimits_{i,\mu} g(  T_3 (e_\mu , Y) , {\cal E}_i )\,g( {\tilde T}_3 ({\cal E}_i , X) , e_\mu  ).
\end{eqnarray*}
We also have
\begin{eqnarray*}
-2 \Div (\alpha_{12} - {\tilde \theta}_{12}) (X,Y) \eq - (\Div A_{12,Y}) X + (\Div \tT_{12,X} ) Y  ,\\
 2\Div( \alpha_1 - {\tilde \theta}_1) (X,Y) \eq (\Div A_{1,Y}) X - (\Div \tT_{1,X}) Y.
\end{eqnarray*}
Thus, the Euler-Lagrange equation for variations of metric satisfying \eqref{Eq-prop6}
is the following, for all $X \in \mD_1$ and $Y \in \mD_3$:
\begin{eqnarray}
\label{EL13}
&& 2\sum\nolimits_{a,i} g(Y , {\tilde A}_{1 , E_a} {\cal E}_i )\,g(X,  \T_{1 ,{\cal E}_i }  E_a )
 + 2\sum\nolimits_{a,\mu} g(Y ,{\tilde A}_{ 1 , E_a} e_\mu )\,g( X,  \T_{1 , e_\mu} E_a  ) \nonumber \\
&& +\,2\,g( T^\sharp_{1, Y}  X , H_{23} ) + (\Div A_{1,Y}) X - (\Div \tT_{1,X}) Y \nonumber \\
&& +\sum\nolimits_{a,i} g(Y , {\tilde A}_{1 , E_a} {\cal E}_i )\,g(X,  {A}_{1 ,{\cal E}_i }  E_a )
 +\sum\nolimits_{a,\mu} g(Y ,{\tilde A}_{ 1 , E_a} e_\mu )\,g( X,  A_{1 , e_\mu} E_a  ) \nonumber \\
&& +\,g( {\tilde T}^\sharp_{1,X} Y - {\tilde A}_{1, X} Y , H_1) + g(H_1 , Y)\,g( H_{23} , X) - g( \nabla_{X} H_1 , Y) \nonumber \\
&& + \sum\nolimits_{a,i} g(Y , {\tT}_{1 , E_a} {\cal E}_i )\,g(X,  \T_{1 ,{\cal E}_i }  E_a ) +  \sum\nolimits_{a,\mu} g(Y ,{\tT}_{ 1 , E_a} e_\mu )\,g( X,  \T_{1 , e_\mu} E_a  )
-(\Div {\tilde h}_2) (X,Y)  \nonumber \\
&& + \sum\nolimits_{a,i} g( {\tilde h}_2 ( X , E_a ) , {\cal E}_i )\,g( {\cal E}_i , {\tilde T}_{2} ( Y , E_a ) )
+  \sum\nolimits_{i,\mu} g( {\tilde h}_2 ( X , e_\mu ) , {\cal E}_i )\,g( {\cal E}_i , {\tilde T}_{2} ( Y , e_\mu ) ) \nonumber\\
&&
+ \sum\nolimits_{a,i} g( {\tilde h}_2 ( Y , E_a ) , {\cal E}_i )\,g( {\cal E}_i , {\tilde T}_{2} ( X , E_a ) )
+  \sum\nolimits_{i,\mu} g( {\tilde h}_2 ( Y , e_\mu ) , {\cal E}_i )\,g( {\cal E}_i , {\tilde T}_{2} ( X , e_\mu ) ) \nonumber\\
&&  -\,g(H_2 , X)\,g(H_2 ,Y)
 +\sum\nolimits_{i,j} g(X, h_2 ( {\cal E}_i , {\cal E}_j ) )\,g(Y, h_2 ( {\cal E}_i , {\cal E}_j ) ) \nonumber\\
&& +\sum\nolimits_{i,j} g(X, T_2 ( {\cal E}_i , {\cal E}_j ) )\,g(Y, T_2 ( {\cal E}_i , {\cal E}_j ) ) \nonumber\\
&& -2\sum\nolimits_{i,\mu} g( {\tilde T}_2 (X,  e_\mu ) , {\cal E}_i )\,g({\cal E}_i , T_3 ( Y , e_\mu) )
- 2 \sum\nolimits_{a,i} g( T_1 ( X,  E_a ) , {\cal E}_i )\,g( {\cal E}_i , {\tilde T}_{2} ( Y , E_a ) ) \nonumber\\
&& -2\sum\nolimits_{a,\mu} g(  {\tilde h}_1 (e_\mu , Y) , E_a )\,g( {\tilde T}_3 (E_a , X) , e_\mu  )
-2 \sum\nolimits_{i,\mu} g(  {\tilde h}_2 (e_\mu , Y) , {\cal E}_i )\,g( {\tilde T}_3 ({\cal E}_i , X) , e_\mu  ) \nonumber\\
&& -\,2\,g(Y, {\tilde T}_3 (X, H_3) ) - (\Div A_{12,Y}) X + (\Div \tT_{12,X} ) Y \nonumber\\
&& - \sum\nolimits_{a,\mu} g(  h_3 (e_\mu , Y) , E_a )\,g( h_1 (E_a , X) , e_\mu  )
- \sum\nolimits_{i,\mu} g( h_3 (e_\mu , Y) , {\cal E}_i )\,g( {\tilde h}_3 ({\cal E}_i , X) , e_\mu  ) \nonumber\\
&& -\,g( T_3 ( Y,  H_{12} ) , X) + g( h_3 ( Y , H_{12} ) , X)
-g( H_{12} , Y)\,g(H_3 , X) + g( \nabla_{X} H_{12} , Y ) \nonumber\\
&& - \!\sum\nolimits_{a,\mu} g( T_3 (e_\mu , Y) , E_a )\,g( T_1 (E_a , X) , e_\mu  )
 - \!\sum\nolimits_{i,\mu} g( T_3 (e_\mu , Y) , {\cal E}_i )\,g( {\tilde T}_3 ({\cal E}_i , X) , e_\mu  ) = 0
\end{eqnarray}
For all $X \in \mD_1$ and $Y \in \mD_3$ we obtain
\begin{eqnarray*}
 && 2\sum\nolimits_{a,i} g(Y , {\tilde A}_{1 , E_a} {\cal E}_i )\,g(X,  \T_{1 ,{\cal E}_i }  E_a ) =
2\sum\nolimits_{a,i} g(Y , \mA_{23 , E_a} {\cal E}_i )\,g(X,  \T_{1 ,{\cal E}_i }  E_a ), \\
 && \sum\nolimits_{a,i} g( {\tilde h}_2 ( X , E_a ) , {\cal E}_i )\,g( {\cal E}_i , {\tilde T}_{2} ( Y , E_a ) ) =
\sum\nolimits_{a,i} g( h_1 ( X , E_a ) , {\cal E}_i )\,g( {\cal E}_i , \mT_{13} ( Y , E_a ) ) \\
&&\quad = \sum\nolimits_{a,i} g( A_{1, {\cal E}_i } E_a , X )\,g( \mT_{13 , {\cal E}_i } Y , E_a )
 = - \sum\nolimits_{a,i} g( A_{1, {\cal E}_i } E_a , X )\,g( \mT_{13 , {\cal E}_i } E_a , Y ), \\
 && \sum\nolimits_{i,\mu} g( {\tilde h}_2 ( X , e_\mu ) , {\cal E}_i )\,g( {\cal E}_i , {\tilde T}_{2} ( Y , e_\mu ) ) =
\sum\nolimits_{i,\mu} g( \mh_{13} ( X , e_\mu ) , {\cal E}_i )\,g( {\cal E}_i , T_3 ( Y , e_\mu ) ) \\
&&\quad = \sum\nolimits_{i,\mu} g( \mA_{13 , {\cal E}_i } X , e_\mu )\,g( \T_{3 , {\cal E}_i } Y , e_\mu )
 = - \sum\nolimits_{i,\mu} g( \mA_{13 , {\cal E}_i } X , e_\mu )\,g( \T_{3 , {\cal E}_i } e_\mu , Y ), \\
 && \sum\nolimits_{a,i} g( {\tilde h}_2 ( Y , E_a ) , {\cal E}_i )\,g( {\cal E}_i , {\tilde T}_{2} ( X , E_a ) ) =
\sum\nolimits_{a,i} g( \mh_{13} ( Y , E_a ) , {\cal E}_i )\,g( {\cal E}_i , T_1 ( X , E_a ) ) \\
&&\quad = \sum\nolimits_{a,i} g( \mA_{13 , {\cal E}_i} Y , E_a )\,g( \T_{1, {\cal E}_i } X , E_a )
= - \sum\nolimits_{a,i} g( \mA_{13 , {\cal E}_i} E_a , Y )\,g( \T_{1, {\cal E}_i } E_a , X ),\\
&& \sum\nolimits_{i,\mu} g( {\tilde h}_2 ( Y , e_\mu ) , {\cal E}_i )\,g( {\cal E}_i , {\tilde T}_{2} ( X , e_\mu ) ) =
\sum\nolimits_{i,\mu} g( h_3 ( Y , e_\mu ) , {\cal E}_i )\,g( {\cal E}_i , \mT_{13} ( X , e_\mu ) ) \\
&&\quad = \sum\nolimits_{i,\mu} g( A_{3, {\cal E}_i }  Y , e_\mu  )\,g( \mTs_{13 , {\cal E}_i} X , e_\mu )
= - \sum\nolimits_{i,\mu} g( A_{3, {\cal E}_i } e_\mu , Y )\,g( \mTs_{13 , {\cal E}_i} e_\mu , X ),\\
 && -2\sum\nolimits_{i,\mu} g( {\tilde T}_2 (X,  e_\mu ) , {\cal E}_i )\,g({\cal E}_i , T_3 ( Y , e_\mu) ) =
-2\sum\nolimits_{i,\mu} g( \mT_{13} (X,  e_\mu ) , {\cal E}_i )\,g({\cal E}_i , T_3 ( Y , e_\mu) ) \\
&&\quad = -2\sum\nolimits_{i,\mu} g( \mTs_{13 , {\cal E}_i } X,  e_\mu )\,g( \T_{3, {\cal E}_i} Y , e_\mu )
= -2\sum\nolimits_{i,\mu} g( \mTs_{13 , {\cal E}_i } e_\mu , X )\,g( \T_{3, {\cal E}_i} e_\mu , Y ),\\
&& - 2 \sum\nolimits_{a,i} g( T_1 ( X,  E_a ) , {\cal E}_i )\,g( {\cal E}_i , {\tilde T}_{2} ( Y , E_a ) ) =
- 2 \sum\nolimits_{a,i} g( T_1 ( X,  E_a ) , {\cal E}_i )\,g( {\cal E}_i , \mT_{13} ( Y , E_a ) ) \\
&&\quad = - 2 \sum\nolimits_{a,i} g( \T_{1 , {\cal E}_i } X,  E_a )\,g( \mTs_{13 , {\cal E}_i} Y , E_a )
= - 2 \sum\nolimits_{a,i} g( \T_{1 , {\cal E}_i } E_a , X )\,g( \mTs_{13 , {\cal E}_i} E_a , Y ), \\
&& -2\sum\nolimits_{a,\mu} g( {\tilde h}_1 (e_\mu , Y) , E_a )\,g( {\tilde T}_3 (E_a , X) , e_\mu  ) =
-2\sum\nolimits_{a,\mu} g( h_3 (e_\mu , Y) , E_a )\,g( T_1 (E_a , X) , e_\mu  ) \\
&&\quad  = -2\sum\nolimits_{a,\mu} g(  A_{3 , E_a } e_\mu , Y )\,g( \T_{1 , e_\mu } E_a , X ),
\end{eqnarray*}
\begin{eqnarray*}
&& -2 \sum\nolimits_{i,\mu} g( {\tilde h}_2 (e_\mu , Y) , {\cal E}_i )\,g( {\tilde T}_3 ({\cal E}_i , X) , e_\mu  ) =
-2 \sum\nolimits_{i,\mu} g( h_3 (e_\mu , Y) , {\cal E}_i )\,g( \mT_{12} ({\cal E}_i , X) , e_\mu  ) \\
&&\quad = -2 \sum\nolimits_{i,\mu} g( A_{3 , {\cal E}_i } e_\mu , Y )\,g( \mTs_{12 , e_\mu } {\cal E}_i , X ), \\
&& - \sum\nolimits_{a,\mu} g( h_3 (e_\mu , Y) , E_a )\,g( h_1 (E_a , X) , e_\mu  ) =
- \sum\nolimits_{a,\mu} g( A_{3 , E_a } e_\mu , Y )\,g( A_{1 , e_\mu } E_a , X ),\\
&& - \sum\nolimits_{i,\mu} g( h_3 (e_\mu , Y) , {\cal E}_i )\,g( {\tilde h}_3 ({\cal E}_i , X) , e_\mu  ) =
- \sum\nolimits_{i,\mu} g( A_{3 , {\cal E}_i } e_\mu , Y )\,g( \mh_{12} ({\cal E}_i , X) , e_\mu  ) \\
&&\quad  = - \sum\nolimits_{i,\mu} g( A_{3 , {\cal E}_i } e_\mu , Y )\,g( \mA_{12 , e_\mu } {\cal E}_i , X ), \\
&& - \sum\nolimits_{a,\mu} g( T_3 (e_\mu , Y) , E_a )\,g( T_1 (E_a , X) , e_\mu  ) =
- \sum\nolimits_{a,\mu} g( \T_{3 , E_a } e_\mu , Y )\,g( \T_{1 , e_\mu } E_a , X ), \\
&& - \sum\nolimits_{i,\mu} g( T_3 (e_\mu , Y) , {\cal E}_i )\,g( {\tilde T}_3 ({\cal E}_i , X) , e_\mu  ) =
- \sum\nolimits_{i,\mu} g( \T_{3, {\cal E}_i } e_\mu , Y )\,g( \mT_{12} ({\cal E}_i , X) , e_\mu  ) \\
&&\quad  = -\sum\nolimits_{i,\mu} g( \T_{3, {\cal E}_i } e_\mu , Y )\,g( \mTs_{12, e_\mu } {\cal E}_i , X ).
\end{eqnarray*}
Using the above, we obtain \eqref{EL13withmixed} as an equivalent form of the Euler-Lagrange equation \eqref{EL13}. 
\end{proof}

The Euler-Lagrange equation for variations of metric satisfying
$B(X,Y) = B(P_2 X, P_3 Y) + B(P_3 X , P_2 Y)$ for all $X,Y \in \mathfrak{X}_M$,
is dual to \eqref{EL13withmixed}, i.e., can be obtained from \eqref{EL13withmixed} by interchanging $\mD_1$ and $\mD_2$ (and the corresponding tensors and vectors from their frames).

\section{Particular cases and examples of critical metrics}
\label{sec:examples}

In this section we find examples of critical points of \eqref{Eq-Smix}, in particular with respect to variations of $g$ preserving the volume of a domain $\Omega$. 
As explained in \cite[
Section 2.3]{rz-2}, a metric $g$ is critical with respect to these variations if and only if there exists $\lambda \in \mathbb{R}$ for which
the Euler-Lagrange equations \eqref{EL11withmixed}, its dual (with respect to interchanging $\mD_1$ and $\mD_2$), and \eqref{EL33withmixed} hold with $\lambda \cdot g(X,Y)$ on their right-hand sides, and \eqref{EL13withmixed} and its dual are satisfied.
We also note that a metric $g$ is a critical point of \eqref{Eq-Smix} with respect to \emph{adapted} volume-preserving variations if and only if there exists $\lambda \in \mathbb{R}$ for which \eqref{EL11withmixed}, its dual,  and \eqref{EL33withmixed} hold with $\lambda \cdot g(X,Y)$ on their right-hand sides, see \cite[Remark 4]{rz-2}.
We note that \eqref{EL33withmixed} and \eqref{EL13withmixed} become trivial when $\mD_3=0$, i.e., $\mD=TM$,
in this case, \eqref{EL11withmixed} reduces to the Euler-Lagrange equation given in \cite{rz-1}.

Examples that we consider in the following subsections include
one-dimensional distributions (in Section~\ref{sec:3.1}),
lifts of complementary distributions to twisted products (in Section~\ref{sec:3.2}),
distributions on domains of Riemannian submersions (in Section~\ref{sec:3.3}),
distributions defined by $K$-contact and $f$-$K$-contact structures (in~Section~\ref{sec:3.4}).
As an illustration of further possible applications of this variational framework, we also consider a modified variational problem (in Section~\ref{sec:3.5}), for mutual curvature of a distribution and its varying (to some extent) orthogonal complement.

\subsection{One-dimensional distributions}
\label{sec:3.1}

Let $\mD_1$ and $\mD_2$ be one-dimensional distributions on $(M,g)$, locally spanned by orthogonal unit vector fields $X_1$ and $X_2$,
respectively, and let $K(X_1,X_2)$ be the sectional curvature of the plane field spanned by $X_1$ and $X_2$. 
Then, the Euler-Lagrange equation \eqref{EL11withmixed} for variations satisfying
 $B(X,Y) = B(P_1 X, P_1 Y)$
reduces~to
\begin{eqnarray}\label{EL11dim11}
\nonumber
&& K(X_1, X_2) +\Div( H_{13} - H_2 -2\,P_2 H_1) + 2\,g( {\tilde h}_2 ( P_3 \nabla_{ X_2 } X_1, X_1) , X_2 )
- 2\,g(X_1, H_2)\,g(X_1 , H_3) \\
&& +\,2\sum\nolimits_{\mu} \Big( g( {\tilde h}_2 ( X_1 , e_\mu ) , X_2 )\, g( X_2 , {\tilde T}_{2} ( X_1 , e_\mu ) )
  - g( {\tilde T}_3 (X_1, X_2 ) , e_\mu )\,g( {\tilde h}_3( X_2 , X_1) , e_\mu )
 \nonumber\\
&& +\,g( {\tilde h}_1 ( X_2, e_\mu ) , X_1 )\,g( {\tilde h}_1 ( X_2 , e_\mu ) , X_1 )
 - g( h_{12}( P_2 \nabla_{e_\mu}  X_1 , X_1) , e_\mu ) \nonumber\\
&& +\,g( {\tilde T}_1 ( X_2 , e_\mu ) , X_1 )\,g( {\tilde T}_1 ( X_2 , e_\mu ) , X_1 )
 +g( {\tilde T}_3 ( X_1, X_2 ) , e_\mu )\, g(e_\mu, {\tilde T}_3 ( X_1,  X_2 ) ) \nonumber\\
&& -\,g( {\tilde T}_2 ( X_1,  e_\mu ) , X_2 )\, g( X_2 , {\tilde T}_{2} ( X_1 , e_\mu ) )\Big) = 0.
\end{eqnarray}

Suppose now that all three distributions $\mD_1$, $\mD_2$ and $\mD_3 = (\mD_1 \oplus \mD_2)^\perp$ are pairwise mixed totally geodesic and pairwise mixed integrable.
Note that $g(P_1 H_2 , P_1 H_3)=g(H_2 , H_3)$ and  $g(P_2 H_1 , P_2 H_3)=g(H_1 , H_3)$. 
Then the Euler-Lagrange equations for volume-preserving variations,  \eqref{EL11dim11} and its dual, become respectively:
\begin{eqnarray}\label{EL11dim11d1}
 K(X_1, X_2) -2\,g(H_2 , H_3) +\Div( P_2 H_3 -P_2 H_1 - P_1 H_2 - P_3 H_2 ) = \lambda\,, \\
\label{EL11dim11d2}
 K(X_1, X_2) -2\,g(H_1 , H_3) +\Div( P_1 H_3 -P_1 H_2 - P_2 H_1 - P_3 H_1 ) = \lambda\,,
\end{eqnarray}
and \eqref{EL33} becomes
\begin{eqnarray}\label{EL33withmixeddim11d3}
&& 2\,{\rm Sym}( (P_3 H_2)^\flat\otimes(P_3 H_1)^\flat) = (\lambda - K(X_1, X_2))\,g_3.
\end{eqnarray}

\begin{proposition}
Let all three distributions $\mD_1$, $\mD_2$ and $\mD_3 = (\mD_1 \oplus \mD_2)^\perp$ be pairwise mixed totally geodesic, pairwise mixed integrable, and
$\mD_1$ and $\mD_2$ be one-dimensional and locally spanned by unit vector fields $X_1$ and~$X_2$.

1. If $\dim \mD_3 >1$, then $g$ is critical for
\eqref{Eq-Smix}
with respect to adapted volume-preserving variations
if and only if
$P_3 H_1=0$ or $P_3 H_2=0$ at each point of $M$,
and the following equations hold:
\begin{eqnarray}\label{eqdim1a}
K(X_1,X_2) + \Div (H_3 - P_3(H_1+H_2)) = \lambda\,,\\
\label{eqdim1b}
K(X_1,X_2) -2g(H_1 + H_2 , H_3)
-2\Div( P_2 H_1 + P_1 H_2 ) = \lambda\,.
\end{eqnarray}

2. If $X_1$ and $X_2$ are geodesic vector fields, then $K(X_1,X_2)=0$; moreover, $g$ is critical for the action \eqref{Eq-Smix}
with respect to adapted  volume-preserving va\-riations 
if and only if
$\Div (P_1 H_3) = 0$ and $\Div (P_2 H_3) =0$.
\end{proposition}

\begin{proof}
1. From \eqref{EL33withmixeddim11d3} with $\dim \mD_3>1$ we get that $P_3 H_1$ and $P_3 H_2$ are linear dependent and
\begin{equation}\label{E-HH-B}
 g(H_2 , Z)\,g(H_1, Y) + g(H_2 , Y)\,g(H_1, Z) =0
\end{equation}
and \eqref{eqdim1a} hold.
For each point $x\in M$ there exist $b_1,b_2\in\RR$ and a unit vector $Y\in \mD_3(x)$ such that $P_3 H_1(x)=b_1 Y$ and $P_3 H_2(x)=b_2 Y$.
Thus, by \eqref{E-HH-B} with $Z=Y$, we get $b_1 b_2=0$, i.e., either $P_3 H_1(x)=0$ or $P_3 H_2(x)=0$.
 The sum
of \eqref{EL11dim11d1} and \eqref{EL11dim11d2}
is
\begin{eqnarray*} 
2\,K(X_1,X_2) - 2g(H_1 + H_2 , H_3)
+\Div(H_3 -2P_2 H_1 - 2P_1 H_2 - P_3(H_1+H_2)) = 2\,\lambda\,.
\end{eqnarray*}
Subtracting \eqref{eqdim1a} from the above, we get \eqref{eqdim1b}.

2. Since $X_1$ and $X_2$ are geodesic, mixed totally geodesic and mixed integrable, we get $K(X_1,X_2)=0$ (this can be computed directly, but also follows from \eqref{Smix12Sd1d2} and \eqref{SPmix} below).
By the assumption $H_1 = 0 = H_2$, from \eqref{EL33withmixeddim11d3} we get $K(X_1,X_2)=\lambda$.
Then from \eqref{EL11dim11d1} and \eqref{EL11dim11d2}, we get $\Div (P_2 H_3) = 0$ and $\Div (P_1 H_3) =0$, respectively.

\end{proof}





\subsection{Lifts of critical metrics to some product manifolds}
\label{sec:3.2}


In this subsection, we use critical points of the integrated mixed scalar curvature, found in \cite{rz-2}, to find metrics critical for the action~\eqref{Eq-Smix} on certain product manifolds. To compare these actions, we introduce and examine an intrinsic counterpart $S^P_{\rm mix}(\mD_1 , \mD_2)$ of $S_{\,\mD_1,\mD_2}$.

Let us assume that $\mD=\mD_1\oplus\,\mD_2$ is integrable, and let $P: TM\to\mD$ be the orthoprojector.
Then the linear connection $\nabla^{P}_X \,Y = P \nabla_X PY$ on $\mD$,  induced by the Levi-Civita connection $\nabla$,
is the Levi-Civita connection of the induced metric on submanifolds tangent to $\mD$.
The \textit{mixed scalar $P$-curvature} corresponding to the curvature tensor
$R^{P}(X,Y) = [\nabla^{P}_X, \nabla^{P}_Y] - \nabla^{P}_{[X,Y]}$ of $\nabla^{P}$ is
the mixed scalar curvature of $(\mD_1,\mD_2)$ on submanifolds tangent to $\mD=\mD_1\oplus\,\mD_2$.
It is given by
\[
 S^P_{\rm mix}(\mD_1 , \mD_2) = \sum\nolimits_{\,a,i} g( R^{P}(E_a, {\cal E}_i) {\cal E}_i, E_a),
\]
see 
\cite{VRAGAG},
and does not depend on the choice of an adapted local orthonormal frame
$\{E_a, {\cal E}_i\}$ on~$M$.

\begin{lemma}
Let $\mD=\mD_1\oplus\,\mD_2$ be integrable, then
the following equality is true $($with $\mD_3=\mD^\bot)$:
\begin{eqnarray}\label{Smix12Sd1d2}
  S_{\,\mD_1, \mD_2} \eq S^P_{\rm mix}(\mD_1 , \mD_2) + g(H_1, H_2)
 +\| {\tilde h}_3 |_{\,\mD_1 \times \mD_2 } \|^2 - \| {\tilde T}_3 |_{\,\mD_1 \times \mD_2 } \|^2
 - \| {\tilde h}_1 |_{\,\mD_2 \times \mD_3 } \|^2
\nonumber \\
&& -\,\| {\tilde h}_2 |_{\,\mD_1 \times \mD_3 } \|^2 +\| {\tilde T}_1 |_{\,\mD_2 \times \mD_3 }\|^2
+ \|{\tilde T}_2|_{\,\mD_1 \times \mD_3}\|^2 .
\end{eqnarray}
\end{lemma}

\begin{proof}
We have the following formula from \cite{VRAGAG} (with the orthoprojectors $P_i:TM\to\mD_i$):
\begin{eqnarray}
\label{SPmix}
 S^P_{\rm mix}(\mD_1 , \mD_2) &=& \Div( P_2 H_1 + P_1 H_2) + g(H_1 + H_2, H_3) - \| P_2 h_1 \|^2 -  \| P_1 h_2 \|^2 \nonumber \\
&& +\,\| P_1 H_2 \|^2 + \| P_2 H_1 \|^2 + \| P_1 T_2 \|^2 + \| P_2 T_1 \|^2 .
\end{eqnarray}
Comparing the above with \eqref{E-new2} yields \eqref{Smix12Sd1d2}.
\end{proof}

\begin{example}
\rm
Suppose that distributions $\mD_1, \mD_2, \mD_3$ 
are pairwise mixed totally geodesic and pairwise mixed integrable on a Riemannian manifold $(M,g)$.
Then the Euler-Lagrange equation \eqref{EL11withmixed} for volume-preserving variations~becomes
\begin{eqnarray}\label{EL11withmixed-2}
\nonumber
&& -\,2 \Div (P_2 h_1) (X,Y)
 - 2\sum\nolimits_{\,a,i} \big(g( \Ts_{1,i} E_a , Y )\,g( A_{1,i} E_a, X )
 + g( \Ts_{1,i} E_a , X )\,g( A_{1,i} E_a, Y )\big) \nonumber\\
\nonumber
&& -\,2\,g(X, H_2)\,g(Y, H_2) -g(X, H_2)\,g(Y,H_3) -g(X, H_3)\,g(Y, H_2) \nonumber\\
&& +\,2\sum\nolimits_{i,j} \big(g(
h_2 ({\cal E}_i , {\cal E}_j ) , X  )\,g(
h_2 ({\cal E}_i , {\cal E}_j ) , Y )
 +
 g(
 T_2 ({\cal E}_i , {\cal E}_j ) , X  )\,g(
 T_2 ({\cal E}_i , {\cal E}_j ) , Y ) \big)
\nonumber\\
&& -\,4\sum\nolimits_{\,a,i} g( \Ts_{1,i} E_a , X )\,g( \Ts_{1,i} E_a , Y )
\nonumber\\
&& +\,S_{\,\mD_1, \mD_2}\cdot g(X,Y) + \Div( P_2 H_1 + P_2 H_3
 - H_2 )\,g(X,Y)
 = 2\lambda\, g(X,Y),\quad X,Y \in \mD_1. 
\end{eqnarray}
for some $\lambda \in \mathbb{R}$.
The Euler-Lagrange equations \eqref{EL11withmixed-2} and its dual have the 
following form: 
\begin{equation} \label{EL11withmixed-2J}
\delta J_{i} - \lambda g_i 
=0 , \quad i=1,2 ,
\end{equation}
for tensors $\delta J_{i}$, i.e., 
$\mD_i$-components of the gradient of the functional \eqref{Eq-Smix}, defined by the left-hand side of \eqref{EL11withmixed-2} and its dual.
\end{example}

Let $\mN_1,\mN_2$ be orthogonal, complementary distributions on a Riemannian manifold $(F, g_F)$.
Let us consider the following action \cite{rz-2}:
\begin{equation}\label{E-JF}
 J_{F} : g_F \mapsto \int_F S_{\rm mix}(\mN_1 , \mN_2)\,{\rm d}\vol_{g_F} ,
\end{equation}
where $S_{\rm mix}(\mN_1 , \mN_2) = S_{\,\mN_1, \mN_2}$ for the metric $g_F$ is the mixed scalar curvature of the almost-product structure $(\mN_1 , \mN_2)$ on $F$.
Recall from \cite[
(33)]{rz-2}, that $g_F$ is critical for $J_F$ with respect to volume-preserving adapted (i.e., preserving orthogonality of $\mN_1$ and $\mN_2$) variations 
if and only if
\begin{eqnarray}\label{gFcritical}
&& -\,2 \Div_F (P_2 h_1) (X,Y) - 2\sum\nolimits_{\,a,i} (g_F( \Ts_{1,i} E_a , Y )\,g_F( A_{1,i} E_a, X )
 +g_F( \Ts_{1,i} E_a, X )\,g_F( A_{1,i} E_a, Y )) \nonumber\\
\nonumber
&& -\,2\,g_F(X, H_2)\,g_F(Y, H_2) +2\sum\nolimits_{i,j} \big(g_F(h_2 ({\cal E}_i , {\cal E}_j ) , X  )\,g_F(
 h_2 ({\cal E}_i , {\cal E}_j ) , Y )\nonumber \\
&&
 + g_F( T_2 ({\cal E}_i , {\cal E}_j ) , X  )\,g( T_2 ({\cal E}_i , {\cal E}_j ) , Y ) \big)
 -4\sum\nolimits_{\,a,i} g_F( \Ts_{1,i} E_a , X )\,g_F( \Ts_{1,i} E_a , Y )
\nonumber\\
&& +\,S_{\rm mix}(\mN_1, \mN_2)\cdot g_F(X,Y) + \Div_F ( P_2 H_1 - P_1 H_2 )\,g_F(X,Y)
 = 2\mu \, g_F(X,Y),\quad X,Y \in \mN_1,
\end{eqnarray}
and the equation dual to above (with respect to interchanging $\mN_1$ and $\mN_2$ on $F$) hold
for some constant $\mu \in \mathbb{R}$. In \eqref{gFcritical}, all tensors such as $h, A, T, T^\sharp$ are defined analogously as in Section \ref{sec:prel}, for the Levi-Civita connection $\nabla^F$ of $g_F$ on $F$; in particular, $\Div_F$ is the divergence with respect to this connection.
We can write \eqref{gFcritical} and its dual 
in the following form:
\begin{equation}\label{E-JF-both}
 ( \delta J_{F} - \mu\,g_F ) |_{\,\mN_1 \times \mN_1 }=0,\quad
 ( \delta J_{F} - \mu\,g_F ) |_{\,\mN_2 \times \mN_2 }=0 .
\end{equation}
In this section,
we will use critical points of \eqref{E-JF}, to obtain critical metrics for \eqref{Eq-Smix} on certain product manifolds, by comparing \eqref{EL11withmixed-2J} and \eqref{E-JF-both} in some special cases. First we need to further examine the notion of distributions being mixed integrable and mixed totally geodesic, which was needed to obtain \eqref{EL11withmixed-2}.

\begin{lemma} \label{lemma3product}
Suppose that distributions $\mD_1, \mD_2, \mD_3$ 
are integrable, pairwise mixed totally geodesic and pairwise mixed integrable on $M$ and $P_3 h_1= P_3 h_2=0$. Then all terms in \eqref{EL13withmixed} vanish.
\end{lemma}
\begin{proof}
From the assumptions we have $T_i =0$, $\mT_{ij}=0$ and $\mh_{ij}=0$ for all $i,j \in \{1,2\}$.
For most terms in \eqref{EL13withmixed} it is clear that they vanish, we consider here only non-trivial terms.
From $P_3 h_1=0$ it follows that $A_{1,Y}$ vanishes everywhere for all $Y \in \mD_3$, so also $\Div A_{1,Y}=0$. We also have, e.g.,
\[
g(\nabla_{X} H_1 ,Y) = g( P_3 (\mh_{12} + \mT_{12}) (X, H_1) , Y) =0, \quad (X \in \mD_1, Y \in \mD_3) .
\]
We can assume that for $X \in \mD_1$ and $Y \in \mD_3$ at a point $x \in M$ we have $\nabla_Z X \perp \mD_1$, $\nabla_Z Y \perp \mD_3$ for all $Z \in T_xM$, then at $x$ we have
\begin{eqnarray*}
\Div {\tilde h}_2 (X,Y) \eq \Div (  h_{13} (X,Y)  ) - \sum_i g( {\tilde h}_2 ( \nabla_{ {\cal E}_i } X, Y) , {\cal E}_i ) - \sum_i g( {\tilde h}_2 ( \nabla_{ {\cal E}_i } Y, X) , {\cal E}_i ) \\
\eq  - \sum_{i,\mu} g(\nabla_{ {\cal E}_i } X , e_\mu) g( {\tilde h}_2 ( e_\mu , Y) , {\cal E}_i ) - \sum_{a,i} g(\nabla_{ {\cal E}_i } Y ,E_a) g( {\tilde h}_2 ( E_a , X) , {\cal E}_i ) \\
 \eq  - \sum_{i, \mu} g( ( \mh_{12} + \mT_{12}) ( {\cal E}_i  , X) , e_\mu)   g( {\tilde h}_2 ( e_\mu , Y) , {\cal E}_i ) \\
&&   - \sum_{a,i} g( ( \mh_{23} +  \mT_{23}) (  {\cal E}_i , Y ) ,E_a) g( {\tilde h}_2 ( E_a , X) , {\cal E}_i ) =0.
\end{eqnarray*}
Similarly, $h_{12}=0$, so $\Div A_{12,Y}=0$ for all $Y\in \mD_3$, and $P_3 h_1=0=P_3 h_2$, so $P_3(H_1+H_2)=0$. 
\end{proof}

In what follows, for Riemannian manifolds $(F,g_F)$ and $(B,g_B)$, we consider lifts $\mD_1,\,\mD_2$ of two complementary orthogonal distributions $\mN_1$ and $\mN_2$ on $(F,g_F)$, to the product manifold $M = F \times B$.
Then $\mD=\mD_1\oplus\mD_2$ defines a natural foliation of $M$ by $F\times\{y\}\ (y\in B)$.

\begin{lemma} \label{lemmamixedintegrablemixedgeodesic}
Consider a twisted product $M = F \times B$ 
with the metric $g=g_F + e^{2f}g_B$ for some smooth function $f: M \to \RR$.
Let  $\mD_1,\,\mD_2$ be lifts to $M$ of two complementary orthogonal 
distributions $\mN_1,\mN_2$ on $(F,g_F)$ and let $\mD_3$ be the lift of $TB$.
Then all distributions $\mD_1,\mD_2,\mD_3$ 
are
pairwise mixed integrable and pairwise mixed totally geodesic. Moreover, if $\mN_1,\mN_2$ are integrable, then all distributions  $\mD_1,\mD_2,\mD_3$ are integrable.
\end{lemma}
\begin{proof}
We will show that $g(\nabla_X Y ,Z) =0$ for $X,Y,Z$ each from a different distribution among $\mD_1,\mD_2, \mD_3$.
Since $\pi : M \rightarrow F$ is a Riemannian submersion with integrable horizontal distribution (which is the lift of $TF$), we have
\begin{equation} \label{RSintgeod}
g( \nabla_X Y, Z)=0  \quad  ( X \in \mathfrak{X}_{\mD_1} , Y \in  \mathfrak{X}_{\mD_2} , Z \in  \mathfrak{X}_{\mD_3} ).
\end{equation}

If $X \in  \mathfrak{X}_{\mD_1}$, $Y \in  \mathfrak{X}_{\mD_3}$ and $Z \in  \mathfrak{X}_{\mD_2}$ we can use $g(\nabla_X Y ,Z) = -g(\nabla_X Z, Y)$ and \eqref{RSintgeod}, to obtain $g(\nabla_X Y,Z)=0$.

We note that $g(P_2 \nabla_{P_3 X} (P_1 Y) , Z) = g( (\mh_{13} + \mT_{13})(X,Y), Z)$ depends only on values of $X,Y,Z$ at a point. Let $X_x \in   {\mD_3}_x$, $Y_x \in   {\mD_1}_x$ and $Z_x \in   {\mD_2}_x$ for $x \in M$, then we can extend $\pi_* Y_x$ to a vector field  $Y_{\mN} \in  \mathfrak{X}_{\mN_1}$ and consider its horizontal lift $Y \in \mathfrak{X}_{\mD_1}$, also, let $Z \in \mathfrak{X}_{\mD_2}$ be any extension of $Z_x$. We can find a coordinate system $(x_1 , \ldots x_n, y_1 , \ldots , y_m)$ on a neighbourhood of $x$; where $y_1 , \ldots , y_m$ are coordinates on $B$, and extend $X_x$ to $X \in \mathfrak{X}_{\mD_3}$ on this neighbourhood, such that $X  = \sum\nolimits_{k=1}^m c_k \frac{\partial}{\partial {y_k}}$ with constant coefficients $\{ c_k \}$. In that coordinate system, we have $Y = \sum\nolimits_{a=1}^n Y_j  \frac{\partial}{\partial {x_j}}$, where coefficients $\{ Y_j \}$ of $Y$ do not depend on $y_j, j \in \{1, \ldots , m \}$, and hence we have $g([X,Y],Z)=0$, which we use to obtain
$g(\nabla_X Y ,Z) = g(\nabla_Y X, Z) = -g(\nabla_Y Z, X) =0$, where the last equality follows from \eqref{RSintgeod}.
Hence, $g(P_2 \nabla_{P_3 X } (P_1 Y ) , Z ) =0$ for all $X,Y,Z \in \mathfrak{X}_M$, from which it also follows that $g(P_1 \nabla_{P_3 X } (P_2 Y ) , Z ) = - g( P_2 \nabla_{P_3 X } (P_1 Z ) , Y )  = 0$.
The same reasoning can be applied for the above cases with interchanged $\mD_1$ and $\mD_2$, so the proof covers all possible cases of $X,Y,Z$ each from a different distribution among $\mD_1,\mD_2, \mD_3$. 

For the last claim, let $X_1, Y_1 \in \mathfrak{X}_{\mD_1}$, $Z_2 \in \mathfrak{X}_{\mD_2}$ and $Z_3 \in \mathfrak{X}_{\mD_3}$, then
$g([X_1, Y_1] , Z_3)=0$, because $TF$ lifts to integrable horizontal distribution on $M$. We also have $g([X_1, Y_1] , Z_2) = g_F ( [\pi_* X_1, \pi_* Y_1] , \pi_* Z_2)$, which shows that $\mD_1$ is integrable if and only if $\mN_1$ is integrable. The proof for $\mD_2$ and $\mN_2$ is analogous.
\end{proof}

\begin{proposition} \label{propdirectproduct}
Let $g=g_F + g_B$ be a Riemannian metric on $M = F \times B$, and $\mD_1,\,\mD_2$ be lifts to $M$ of two complementary orthogonal distributions $\mN_1$ and $\mN_2$ on $(F,g_F)$.
Let $g_F$ be critical for the action \eqref{E-JF} with respect to volume-preserving adapted variations, such that \eqref{E-JF-both} hold.
Then:
\begin{enumerate}
\item $g$ is critical for the action \eqref{Eq-Smix} with respect to adapted volume-preserving variations if and only if $S_{\rm mix}(\mN_1 , \mN_2) = \mu$ on $F$.
\item If both $\mD_1$ and $\mD_2$ are integrable, then $g$ is critical with respect to
volume-preserving variations keeping $\mD_1$ and $\mD_2$ orthogonal if and only if $S_{\rm mix}(\mN_1 , \mN_2) = \mu$ on $F$.
\end{enumerate}
\end{proposition}

\begin{proof}
For projection $\pi : F \times B \rightarrow F$  we have $g_F (\nabla^F_{\pi_* X} \pi_*Y , \pi_*Z ) = g(\nabla_X Y ,Z)$, and
$P_3 H_1 = P_3 H_2=0$. We can write \eqref{SPmix} as
\begin{eqnarray*}
 S^P_{\rm mix}(\mD_1 , \mD_2) &=& \Div_F ( P_2 H_1 + P_1 H_2)  - \| P_2 h_1 \|^2 -  \| P_1 h_2 \|^2 \nonumber \\
&& +\,\| P_1 H_2 \|^2 + \| P_2 H_1 \|^2 + \| P_1 T_2 \|^2 + \| P_2 T_1 \|^2 .
\end{eqnarray*}
and hence, by \cite[(9)]{rz-2}, we have
$S_{\rm mix}(\mN_1 , \mN_2)\circ \pi =  S^P_{\rm mix}(\mD_1 , \mD_2)$.
Since $g(H_1,H_2)=0$, and, by Lemma \ref{lemmamixedintegrablemixedgeodesic}, distributions $\mD_1,\mD_2,\mD_3$ are pairwise mixed integrable and pairwise mixed totally geodesic, from \eqref{Smix12Sd1d2} we also have
$  S_{\,\mD_1, \mD_2} =  S^P_{\rm mix}(\mD_1 , \mD_2)$.
Since $g(H_1,H_2)=0$, $P_3 H_1=P_3 H_2=0$ and $P_1 H_3 = P_2 H_3 =0$, comparing \eqref{EL11withmixed-2} and \eqref{gFcritical} we obtain
\begin{equation} \label{deltaJJF}
\delta J_{i} (X,Y) =  \delta J_{F} (\pi_* X, \pi_* Y) 
,\quad (X,Y \in \mD_i , \quad  i=1,2 ).
\end{equation}
Hence, \eqref{EL11withmixed-2}  and its dual hold by \eqref{deltaJJF} and \eqref{E-JF-both}; with $\lambda = \mu$.
Also, \eqref{EL33withmixed} becomes
\begin{equation*}
 2\,{\rm Sym}((P_3 H_2)^\flat\otimes(P_3 H_1)^\flat) = (\lambda - S_{\,\mD_1,\mD_2})\,g_3.
\end{equation*}
From $P_3 H_1 = P_3 H_2 =0$ we reduce the above to $S_{\,\mD_1,\mD_2}=\lambda$. From this and \eqref{deltaJJF} it follows that if $g_F$ is critical point of action \eqref{E-JF} with respect to volume-preserving adapted variations, then $g$ is critical for \eqref{Eq-Smix} with respect to volume-preserving adapted variations if and only if $S_{\,\mD_1,\mD_2}=\mu$.
%
The second claim follows from the first, and Lemmas \ref{lemma3product} and \ref{lemmamixedintegrablemixedgeodesic}.
\end{proof}


The relation between critical metrics on $(F, g_F)$ and critical twisted product metrics $g=g_F + e^{2f}g_B$, for some smooth function $f: F \times B\to \RR$, on $M = F \times B$ is more complicated - it is examined in the following result, and then used to obtain examples of critical metrics in the rest of this section.

\begin{proposition}
Let $M$ be  a twisted product $F \times_{f} B$ with the metric $g=g_F + e^{2f}g_B$ for some smooth function $f: F \times B\to \RR$.
Let  $\mD_1,\,\mD_2$ be lifts to $M$ of two complementary orthogonal distributions $\mN_1, \mN_2$ on $F$.
%
%
Let $g_F$ be critical for the action \eqref{E-JF} with respect to volume-preserving adapted variations, satisfying \eqref{E-JF-both}.
Then $g$ is cri\-tical for the action \eqref{Eq-Smix} on $M=F \times {B}$ with respect to volume-preserving variations
keeping orthogonality of $\mD_1$ and $\mD_2$
if and only if the following Euler-Lagrange equations hold:
\begin{eqnarray}\label{EL11fromSmix12Ftwgradf}
&& 
-  2n_3 \< P_2 h_1 , \nabla f \>   + 2 n_3 \,{\rm Sym}( (P_1 H_2)^\flat \otimes (P_1 \nabla f)^\flat )  - n_3 \Div ( P_2 \nabla f 
)\,g_1
\nonumber\\
&& +n_3 g(H_1, P_2 \nabla f) g_1 - n_3 g(H_2, P_1 \nabla f) g_1
 = 2(\lambda-\mu)\,g_1,\\
\label{EL22fromSmix12Ftwgradf}
&& 
- 2n_3 \< P_1 h_2 , \nabla f \>
+ 2n_3\,{\rm Sym}( (P_2 H_1)^\flat \otimes (P_2 \nabla f)^\flat )  - n_3 \Div ( P_1 \nabla f 
)\,g_2
\nonumber\\
&&  + n_3 g(H_2, P_1 \nabla f) g_2 - n_3 g(H_1, P_2 \nabla f) g_2
= 2(\lambda-\mu)\,g_2, \\ 
\label{EL33fromSmix12Ftw}
&& S_{\rm mix}(\mN_1 , \mN_2) 
=\lambda  .
\end{eqnarray}
\end{proposition}
\begin{proof}
We have $P_3 H_1 = P_3 H_2=0$ and, similarly as in the proof of Proposition \ref{propdirectproduct}, we obtain $S_{\rm mix}(\mN_1 , \mN_2)\circ \pi =  S^P_{\rm mix}(\mD_1 , \mD_2)$ for projection $\pi : F \times B \rightarrow F$.
Since $g(H_1,H_2)=0$, and by Lemma \ref{lemmamixedintegrablemixedgeodesic} distributions $\mD_1,\mD_2,\mD_3$ are pairwise mixed integrable and pairwise mixed totally geodesic, by \eqref{Smix12Sd1d2} we also have
$S_{\,\mD_1, \mD_2} =  S^P_{\rm mix}(\mD_1 , \mD_2)$.
However, $P_1 H_3$ and $P_2 H_3$ might be non-zero, and we have
\[
-\,2 \Div (P_2 h_1) (X,Y) = -\,2 \Div_F (P_2 h_1) (X,Y) + 2 g( P_2 h_1  (X,Y) , H_3)
\]
and
\[
\Div(P_2 H_1 - P_1 H_2 ) = \Div_F (P_2 H_1 - P_1 H_2) - g(H_1, H_3) + g(H_2,H_3).
\]
It follows from comparison of \eqref{EL11withmixed-2} and \eqref{gFcritical} that 
\[
2\delta J_1 = 2 (\pi^* \delta J_F)  |_{\,\mD_1 \times \mD_1 } + 2 \< P_2 h_1 , H_3 \>
- g(H_1, H_3) g_1 + g(H_2,H_3) g_1
\]
where $(\pi^* \delta J_F )(X,Y) = \delta J_F (\pi_* X, \pi_*Y) $ and $\< P_2 h_1 , H_3 \> (X,Y) =  \< P_2 h_1 (X,Y) , H_3 \> $ for all $X,Y \in TM$. Analogously,
\[
2\delta J_2 = 2(\pi^* \delta J_F) |_{\,\mD_2 \times \mD_2 } + 2 \< P_1 h_2 , H_3 \>
- g(H_2, H_3) g_2 + g(H_1,H_3) g_2 .
\]
Using the above and $P_3 H_1 = P_3 H_2 =0$, we obtain that the Euler-Lagrange equation
\eqref{EL11withmixed-2} and its dual become
\begin{eqnarray}\label{EL11fromSmix12Ftw}
&& 2 (\pi^* \delta J_F) |_{\,\mD_2 \times \mD_2 } +  2 \< P_2 h_1 , H_3 \>   - 2\,{\rm Sym}( (P_1 H_2)^\flat \otimes (P_1 H_3)^\flat )  + \Div ( P_2 H_3 
)\,g_1
\nonumber\\
&& - g(H_1, H_3) g_1 + g(H_2,H_3) g_1
 = 2\lambda\,g_1,\\
\label{EL22fromSmix12Ftw}
&& 2  (\pi^* \delta J_F) |_{\,\mD_2 \times \mD_2 } + 2 \< P_1 h_2 , H_3 \>
- 2\,{\rm Sym}( (P_2 H_1)^\flat \otimes (P_2 H_3)^\flat )  + \Div ( P_1 H_3 
)\,g_2
\nonumber\\
&&  - g(H_2, H_3) g_2 + g(H_1,H_3) g_2
= 2\lambda\,g_2 .
\end{eqnarray}
We obtain \eqref{EL11fromSmix12Ftwgradf}-\eqref{EL22fromSmix12Ftwgradf}
from \eqref{EL11fromSmix12Ftw}-\eqref{EL22fromSmix12Ftw} using  \eqref{E-JF-both} and the formula for the mean curvature of a factor of a twisted product: $H_3 = - n_3 P \nabla f$, see e.g., \cite{Reckziegel}.
Equation \eqref{EL33withmixed} becomes
$S_{\,\mD_1,\mD_2}=\lambda$, which at the beginning of this proof was shown to be equivalent to \eqref{EL33fromSmix12Ftw}.
We note that if distributions $\mD_1$ and $\mD_2$ 
are integrable, 
then by Lemmas  \ref{lemma3product} and \ref{lemmamixedintegrablemixedgeodesic} also \eqref{EL13withmixed} holds, as all its terms vanish.
\end{proof}

If $\mN_1$ is tangent to a codimension-one foliation on $F$, and $\mN_2$ is spanned by its unit normal vector field $N$ on $F$,
then the Euler-Lagrange equations \eqref{E-JF-both} read as, see 
\cite[Eqs.~(41),~(43)]{rz-2},
\begin{eqnarray}
\label{f1}
 \Div_F(\tau_1 N) = N(\tau_1) - \tau_1^2
 \eq \frac{2\,n_1}{1-n_1}\,\mu, \\
\label{f2}
 \tau_1^2-\tau_2 \eq 2\,\mu
 ,
\end{eqnarray}
where $\tau_k = \tr (A_N^k)$ are the ``power sums" invariants of the shape operator $A_N$ of $\mN_1$ on~$F$.

\begin{lemma}\label{propcodimone}
Let $\mN_1$ 
be tangent to a foliation of codimension one on $F$,
and $\mN_2$ be tangent to its unit normal vector field $N$ on $(F,g_F)$.
Let $\mD_1,\,\mD_2$ be lifts of $\mN_1, \mN_2$ to a product manifold $M = F \times B$.
If $g_F$ is critical for the action \eqref{E-JF}
with respect to adapted volume-preserving variations, i.e., \eqref{f1}--\eqref{f2} hold, and
\begin{equation} \label{EL0twpf}
 P_1 \nabla f = 0  \quad {\rm or} \quad P_1 \nabla_N N  =0,
\end{equation}
then
 the twisted product metric $g=g_F + e^{2f}g_B$ is critical for the action \eqref{Eq-Smix} on $M=F \times {B}$ with respect to volume-preserving variations
keeping orthogonality of $\mD_1$ and $\mD_2$
if and only if
\begin{eqnarray}\label{EL1twpf}
&& -2n_3  \< P_2 h_1  , \nabla f \>  - n_3  ( N (N(f)) - 2 \tau_1 N(f) + n_3 (N(f))^2 ) \,g_1
 = 2(\lambda-\mu)\,g_1,\\
\label{EL2twpf}
&& \tau_1 N(f)
- \Delta_1\,f - n_3\,g( P_1 \nabla f , P_1 \nabla f )
 =    \frac{2(\lambda-\mu)}{n_3} , \\
\label{EL3twpf}
&&
  \Div_F (P_1 \nabla_N N)  -\frac2{n_1-1}\,\mu = \lambda ,
\end{eqnarray}
where $\Div_F$ is the divergence along $F$ and $\Delta_1 = \Div_1 \circ\,P_1 \nabla$ is the Laplacian along $\mD_1$.
\end{lemma}

\begin{proof}
We have $n_2=1$ and
\[
 H_3 = -n_3 P \nabla f,\quad P_2 H_1 = \tau_1 N,\quad P_2 \nabla f = N(f) N.
\]
We also have, by \eqref{EL0twpf}, $(P_1 \nabla_N N) (f)=g(P_1 \nabla_N N, \nabla f)=0$.

Under assumption \eqref{EL0twpf}, equations \eqref{EL1twpf} and \eqref{EL2twpf} are equivalent to \eqref{EL11fromSmix12Ftwgradf} and \eqref{EL22fromSmix12Ftwgradf}, by the following computations:
\begin{eqnarray*}
 \Div (P_2 \nabla f ) \eq g(\nabla_N (N(f)N) , N) - g(P_2 \nabla f , P_2 H_1) - g(P_2 \nabla f , P_2 H_3) \\
 \eq N (N(f)) - \tau_1 N(f) + n_3 (N(f))^2,\\
\Div (P_1 \nabla f) \eq \Delta_1\,f - g(P_1 \nabla f , P_1 \nabla_N N ) - g( P_1 \nabla f , P_1 H_3) \\
 \eq \Delta_1\,f + n_3\,g( P_1 \nabla f , P_1 \nabla f ) .
\end{eqnarray*}
To obtain \eqref{EL3twpf}, from \eqref{f1}--\eqref{f2} we get $N(\tau_1) - \tau_2=\frac2{1-n_1}\,\mu$
and use it 
in \eqref{EL33fromSmix12Ftw} 
together with the following:
\begin{equation*}
 S_{\rm mix}(\mN_1 , \mN_2) = \tau_1^2 - \tau_2 + \Div_F ( P_1 H_2 + P_2 H_1)
 = N(\tau_1) - \tau_2 + \Div_F (P_1 \nabla_N N) .
\end{equation*}
For the first equality above, see computations in the proof of \cite[Corollary 2]{rz-2}, as formulas there, describing one-dimensional distributions, are dual to those describing a codimen\-sion-one distribution.
For the second equality, we use $\Div_F (P_2 H_1) = N(\tau_1) + \tau_1 \Div_F N =N(\tau_1) - \tau_1^2$ and $ \Div_F ( P_1 H_2) = \Div_F (P_1 \nabla_N N)$.
From Lemmas \ref{lemma3product} and \ref{lemmamixedintegrablemixedgeodesic} it follows that the Euler-Lagrange equation \eqref{EL13withmixed} and its dual are satisfied.
\end{proof}

\begin{remark}\rm
Under assumptions of Lemma~\ref{propcodimone},
if $\tau_1\equiv0$,  
then $\mu=0$ by \eqref{f1} and $\tau_2\equiv0$ by \eqref{f2}.
In this case, $A_N\equiv0$, i.e.,
$\mD_1$ determines a totally geodesic foliation of $F$.
Then, the Euler-Lagrange equation \eqref{EL1twpf} becomes
\[
 N (N(f)) 
 + n_3 (N(f))^2 = -\frac{2\lambda}{n_3}
\]
Equation \eqref{EL3twpf} becomes $\Div_F (P_1\nabla_N N) = \lambda$,
and using $\Delta_1 f + n_3\,g( P_1\nabla f, P_1 \nabla f) = \frac1{n_3}e^{-n_3 f}\Delta_1 e^{\,n_3 f}$,
we can write \eqref{EL2twpf} in the following form:
\begin{equation*}
 \Delta_1\,e^{\,n_3 f}=-2\lambda\,e^{n_3 f} .
\end{equation*}
\end{remark}

\begin{remark} \rm
Suppose that $\mN_1$ 
 is tangent to a foliation of codimension one on $F$,
and $g_F$ is critical for the action \eqref{E-JF}
with respect to adapted volume-preserving variations, i.e., \eqref{f1}--\eqref{f2} hold.
Comparing ranks of tensors in \eqref{EL11fromSmix12Ftwgradf},  
we see that if $\< P_2 h_1 , \nabla f \>$ is non-zero on a $k$-dimensional distribution with $k>2$ and \eqref{EL11fromSmix12Ftwgradf} holds, then \eqref{EL0twpf} holds.
\end{remark}

\begin{remark} \label{remarktau1const} \rm
From \eqref{EL1twpf} follows that $\mN_1$ is totally umbilical on $F$. Then,
from \cite[Proposition 7]{rz-2} it follows that if 
\eqref{f1}--\eqref{f2} hold, then $\tau_1 =const$ on $F$.
Therefore, in the next two propositions we consider the case $0 \neq \tau_1 =const$ to find some particular solutions of \eqref{EL0twpf}--\eqref{EL3twpf}.
\end{remark}

\begin{proposition}
Let $\mN_1$ 
be tangent to a foliation of codimension one on $F$,
and $\mN_2$ be tangent to its unit normal vector field $N$ on $(F,g_F)$.
Let $\mD_1,\,\mD_2$ be lifts of $\mN_1, \mN_2$ to a product manifold $M = F \times B$.
Let $g_F$ be critical for the action \eqref{E-JF}
with respect to adapted volume-preserving variations, i.e., \eqref{f1}--\eqref{f2} hold,
and let
$P_1 \nabla_N N =0$ and $\tau_1\ne0$.

Suppose that $N(f) = const$, then
the twisted product metric $g=g_F + e^{2f}g_B$ is critical for the action \eqref{Eq-Smix} on $M=F \times {B}$
with respect to volume-preserving variations
keeping orthogo\-nality of $\mD_1$ and $\mD_2$
if and only if:
\begin{itemize}
\item $\mD_1$ is totally umbilical,
\item $ \tau_1 N(f) =c$, 
\item $e^{\,n_3 f}$ is an eigenfunction~of~$\Delta_1$ corresponding to 
an eigenvalue $n_3 c + \frac{2(n_1+1)}{n_1-1}\mu$, where $c$ is given by \eqref{cpm}.
\end{itemize}
\end{proposition}

\proof
From Lemmas \ref{lemma3product} and \ref{lemmamixedintegrablemixedgeodesic} it follows that the Euler-Lagrange equation \eqref{EL13withmixed} and its dual are satisfied.
From  \eqref{EL1twpf} it follows that $\mD_1$ is totally umbilical with mean curvature $H_1 = \tau_1 N$ satisfying:
\begin{equation}\label{D1umbilicaltau1}
-2n_3 \tau_1 N(f) - n_1 n_3 ( N (N(f)) - 2\tau_1 N(f) + n_3 (N(f))^2 ) = 2n_1 (\lambda - \mu).
\end{equation}
As $P_1 \nabla_N N =0$, equation \eqref{EL3twpf} becomes 
\begin{equation}\label{EL3twpf1lambdamu}
\lambda = \frac{2\mu}{1-n_1} .
\end{equation}
For $N(f) =const$, we have the following solution of \eqref{EL2twpf}: 
\begin{equation}\label{EL2twpfc}
 \tau_1 N(f) =c , \quad  \Delta_1\,f + n_3\,g( P_1 \nabla f , P_1 \nabla f ) =  - \frac{2(\lambda-\mu)}{n_3}
 + c
\end{equation}
where $c\in \mathbb{R}$.
Then, by condition $\tau_1\ne0$, from \eqref{EL2twpfc}$_1$ and Remark \ref{remarktau1const}, we get $N(N(f)) = -c\,\frac{N(\tau_1)}{\tau_1^2} =0$
and equation  
\eqref{D1umbilicaltau1},
by \eqref{f1}  and \eqref{EL3twpf1lambdamu}, becomes quadratic equation for $c$:
\[
\frac{n_3^2 (n_1 - 1)}{2\mu}c^2 - 2 n_3 (n_1 -1 ) c - \frac{2 n_1 (n_1 +1) \mu}{n_1 - 1} =0.
\]
which has two solutions: 
\begin{equation} \label{cpm}
c = \frac{2 }{n_3  (n_1 - 1 )} \big(   n_1 -1  \pm  \sqrt{   (n_1-1)^2 + n_1 (n_1 +1) ) }  \big) \mu.
\qquad\Box
\end{equation}

\begin{proposition}
Let $\mN_1$ 
be tangent to a foliation of codimension one on $F$,
and $\mN_2$ be tangent to its unit normal vector field $N$ on $(F,g_F)$, such that $g_F$ is critical for the action \eqref{E-JF}
with respect to adapted volume-preserving variations, i.e., \eqref{f1}--\eqref{f2} hold.
Let $\mD_1,\,\mD_2$ be lifts of $\mN_1,\, \mN_2$ to a product manifold $M = F \times B$.
If
$P_1 \nabla f =0$ and $0 \neq \tau_1=const$,
then
the twisted product metric $g=g_F + e^{2f}g_B$ is critical for the action \eqref{Eq-Smix} on $M=F \times {B}$
with respect to volume-preserving variations
keeping orthogo\-nality of $\mD_1$ and $\mD_2$
if and only if
$\mN_1$ is totally umbilical, with mean curvature $H_1 = \tau_1 N$, such that
\begin{equation*}
 \tau_1 N(f) = \frac{2(-3n_1 -2)  }{n_3 (n_1-1)} \mu, 
\quad {\rm and} \quad \Div_F (P_1 \nabla_N N) = \frac{-2n_1 - 5}{n_1 -1} \mu . 
\end{equation*}
\end{proposition}

\begin{proof}
From Lemmas \ref{lemma3product} and \ref{lemmamixedintegrablemixedgeodesic} it follows that the Euler-Lagrange equation \eqref{EL13withmixed} and its dual are satisfied.
From \eqref{EL1twpf} it follows that $\mD_1$ is totally umbilical with mean curvature $H_1 = \tau_1 N$ satisfying \eqref{D1umbilicaltau1}, then, by $P_3 h_1=0$, also $\mN_1$ is totally umbilical on $F$.
From $P_1 \nabla f =0$ it follows that $\Delta_1 f=\Div_1(P_1 \nabla f)=0$ and \eqref{EL2twpf} yields 
\begin{equation} \label{EL2twpf-b}
\tau_1 N(f) =  \frac{2(\lambda-\mu)}{n_3} .
\end{equation}
From \eqref{EL2twpf-b}, 
using $\tau_1\ne0$, we find $N(f) = \frac{2\lambda}{ n_3 \tau_1}$ and, since $\tau_1$ is constant, $N(N(f))=0$.
Using the above and \eqref{f1} in \eqref{D1umbilicaltau1} yields
\[
(\lambda - \mu) (-4 - 6n_1 -2 (\lambda -\mu) \frac{n_1 - 1}{\mu} ) =0
\]
which has solution $\lambda = \frac{-2n_1 - 3}{n_1-1}\mu$.
Using this solution in \eqref{EL2twpf-b} and \eqref{EL3twpf} completes the proof.
We~note that $\mu > 0$ by \eqref{f1}, Remark \ref{remarktau1const} and assumption $\tau_1 \neq 0$.
\end{proof}

\subsection{Riemannian submersions}
\label{sec:3.3}

In this subsection, we
consider 
domains of Riemannian submersions with totally geodesic fibers, and 
relations between integrability and mixed integrability of distributions on such manifolds with metrics critical for the action \eqref{Eq-Smix}.  


\begin{proposition}\label{propRS}
Let $\pi : (M,g) \rightarrow (M', g')$ be a Riemannian submersion with totally geodesic fibers. 
Let $\mD_3\subset TM$ be the distribution tangent to the fibers of $\pi$,
and let $\mathcal{H}\subset TM$ be the distribution orthogonal to the fibers. 
Let $\mD_i = \pi_*^{-1} ({\cal N}_i) \cap \mathcal{H}$ for $i=1,2$, where
 ${\cal N}_i\ (i=1,2)$ are two $g'$-orthogonal, complementary, totally geodesic distributions on $M'$.
If $g$ is critical for the action \eqref{Eq-Smix}, with respect to adapted  volume-preserving variations, then for some $\lambda \in \mathbb{R}$ we have
\begin{eqnarray}\label{E-4eqs}
\nonumber
  S_{\,\mD_1, \mD_2} = \frac{n_1+n_2+n_3}{n_1+n_2+n_3-2}\,\lambda,\quad
 \|\mT_{12} \|^2 = \frac{2 n_3}{3(n_1+n_2+n_3-2)}\,\lambda\,,\\
 \|P_1 T_2 \|^2 = \frac{n_1+2n_2+3n_3}{3(n_1+n_2+n_3-2)}\,\lambda,\quad
 \|P_2 T_1 \|^2 = \frac{2n_1+n_2+3n_3}{3(n_1+n_2+n_3-2)}\,\lambda\, .
\end{eqnarray}
\end{proposition}

\begin{proof}
We have for all $\xi \in \mD_3$:
\begin{equation} \label{hmixTmix}
\mh_{i3}(X,\xi) = - \mTs_{12,\xi} X, \quad ( X \in \mD_1\oplus \mD_2, \quad i=1,2).
\end{equation}
Indeed,
let $X \in \mD_1$ and $Y \in \mD_2$ be horizontal lifts of unit vector fields $X_1 \in {\cal N}_1$ and $Y_2 \in {\cal N}_2$ on~$M'$.
Then $g(Y, [\xi,X])=0$, hence,
\begin{equation*}
g( \mh_{13}(X ,\xi) , Y) = \frac{1}{2}\,g( \nabla_X \xi + \nabla_\xi X , Y) =  g( \nabla_X \xi  , Y) = -g(\mTs_{12,\xi} X,Y).
\end{equation*}
Analogously, we obtain $g( \mh_{23}(Y ,\xi) , X) = -g(\mTs_{12,\xi} Y,X)$. Similarly, $\mT_{i3}=0\ (i=1,2)$, as, e.g., $g(\mT_{13} (X, \xi), Y) = g(Y , [X, \xi])=0$ for $X,Y$ as above.
We also get $h_1=h_2=0$ from the assumptions about distributions on $M'$ and the fact that, see \cite{g1967},
\[
 g(\nabla_X Y ,Z ) = g'(\nabla'_{\pi_* X} \pi_* Y , \pi_* Z ),\quad
 X,Y,Z\in{\cal H},
\]
where $\nabla'$ is the Levi-Civita connection of $g'$. As $\pi$ is a Riemannian submersion, we have $h_{12}=0$, and as fibers are integrable and totally geodesic, we have $h_3=0=T_3$.

The Euler-Lagrange equation \eqref{EL11withmixed} reduces to the following (for $X,Y \in \mD_1$):
\begin{eqnarray} \label{EL11RSG1F}
&& 2 \sum\nolimits_{\,i,j} g( P_1 T_2 ( {\cal E}_i , {\cal E}_j ), X)\, g( P_1 T_2 ( {\cal E}_i , {\cal E}_j ), Y)
 - 4 \sum\nolimits_{\,a,i} g( T_1 (E_a, X) , {\cal E}_i )\,g( T_1 (E_a, Y) , {\cal E}_i ) \nonumber \\
&& +\, 6 \sum\nolimits_{\,i,\mu} g( \mTs_{12,\mu} {\cal E}_i , X)\,g( \mTs_{12,\mu} {\cal E}_i , Y)
 = (\lambda - S_{\,\mD_1, \mD_2})\,g(X,Y).
\end{eqnarray}
Similarly, its dual is the following (for $X,Y \in \mD_2$):
\begin{eqnarray} \label{EL22RSG1F}
&& 2 \sum\nolimits_{\,a,b} g( P_2 T_1 ( E_a , E_b ), X)\,g( P_2 T_1 ( E_a , E_b ), Y) - 4 \sum\nolimits_{\,a,i} g( T_2 ({\cal E}_i, X) , E_a )\,g( T_1 ({\cal E}_i, Y) , E_a ) \nonumber \\
&& +\, 6 \sum\nolimits_{\,a,\mu} g( \mTs_{12,\mu} E_a , X)\,g( \mTs_{12,\mu} E_a , Y) = (\lambda - S_{\,\mD_1, \mD_2})\,g(X,Y),
\end{eqnarray}
and the Euler-Lagrange equation \eqref{EL33withmixed} becomes the following (for $X,Y \in \mD_3$):
\begin{equation} \label{EL33RSG1F}
 -6 \sum\nolimits_{\,a,i} g( \mTs_{12,X} E_a , {\cal E}_i )\,g( \mTs_{12,Y} E_a , {\cal E}_i ) = (\lambda - S_{\,\mD_1, \mD_2})\,g(X,Y).
\end{equation}
Taking traces of equations \eqref{EL11RSG1F}--\eqref{EL33RSG1F}, we obtain the following:
\begin{eqnarray}\label{trEL11RSG1F}
 2\,\| P_1 T_2 \|^2 - 4\,\| P_2 T_1 \|^2 + 3\,\| \mT_{12} \|^2 + n_1 ( S_{\,\mD_1, \mD_2} - \lambda ) =0,\\
 \label{trEL22RSG1F}
 2\,\| P_2 T_1 \|^2 - 4\,\| P_1 T_2 \|^2 + 3\,\| \mT_{12} \|^2 + n_2 ( S_{\,\mD_1, \mD_2} - \lambda ) =0, \\
\label{trEL33RSG1F}
 3\,\| \mT_{12} \|^2 - n_3 ( S_{\,\mD_1, \mD_2} - \lambda ) =0.
\end{eqnarray}
Using \eqref{E-new2}
and \eqref{hmixTmix}, we obtain
\begin{equation} \label{tr-4}
S_{\,\mD_1, \mD_2} = \| P_1 T_2 \|^2 + \| P_2 T_1 \|^2 - \frac{3}{2}\,\| \mT_{12} \|^2 .
\end{equation}
The solution of the linear system
\eqref{trEL11RSG1F}--\eqref{tr-4} is \eqref{E-4eqs}.
\end{proof}

\begin{corollary} \label{corRSGF}
With assumptions of Proposition~\ref{propRS},
suppose that at some point of $M$ one of the tensors: $\mT_{12}, P_1 T_2$, $P_2 T_1$ vanishes.
Then metric $g$ is critical for the action \eqref{Eq-Smix} with respect to adapted volume-preserving variations 
if and only if all tensors $\mT_{12}, P_1 T_2$ and $P_2 T_1$ vanish everywhere on $M$,
in particular, distributions ${\cal N}_i\ (i=1,2)$ are integrable and $S_{\,\mD_1, \mD_2} =0$.
In~this $($and only this$)$ case, $g$ is critical with respect to all adapted variations $($not only volume-preserving$)$. 
\end{corollary}

\begin{proof}
This follows from 
\eqref{E-4eqs}, 
from which we get $\lambda=0$ and $S_{\,\mD_1, \mD_2} =0$.
From $\pi_* [X,Y] = [\pi_*X,\pi_*Y]$ and $g(X,Y) = g'(\pi_* X, \pi_*Y)$ it follows that $P_1 T_2$ (respectively, $P_2 T_1$) vanishes if and only if the distribution ${\cal N}_2$ (respectively, ${\cal N}_1$) is integrable.
\end{proof}

\begin{remark} \rm
A pair of Riemannian submersions $\pi_1 : (M,g) \rightarrow (M',g')$ and $\pi_2 : (M',g') \rightarrow (B,g_B)$ allows us to determine three pairwise orthogonal distributions on $(M,g)$ in the following way. Let $\mathcal{V}_i$ and $\mathcal{H}_i$ be, respectively, vertical (i.e., tangent to the fibers) and horizontal (i.e., orthogonal to the fibers) distributions of $\pi_i\ (i=1,2)$.
Then we define $\mD_1 = \mathcal{H}_1 \cap \pi_{1*}^{-1}( \mathcal{V}_2)$, $\mD_2 = \mathcal{H}_1 \cap \pi_{1*}^{-1}( \pi_{2*}^{-1}(B) \cap \mathcal{H}_2)$ and $\mD_3 = \mathcal{V}_1$.
If Riemannian submersions $\pi_1,\pi_2$ have totally geodesic fibers, then assumptions of Proposition \ref{propRS} about $\mD_1,\mD_2,\mD_3$ are satisfied.
It follows from Corollary~\ref{corRSGF} that for a pair of Riemannian submersions $\pi_1 : S^{4n+3}\to {\mathbb{C}}P^{2n+1}$ and $\pi_2 : {\mathbb{C}}P^{2n+1}\to {\mathbb{H}}P^n$ and $\mD_i$ defined as above, the standard metric on $S^{4n+3}$ is not critical for the action \eqref{Eq-Smix}. Indeed, since ${\mathbb{C}}P^{2n+1}$ is not a metric product, $\pi_{2*}^{-1}(B) \cap \mathcal{H}_2$ is not integrable.
Hence, $P_2 T_1 =0$, but $P_1 T_2 \neq 0$.
\end{remark}

%

\subsection{$f$-$K$-contact manifolds}
\label{sec:3.4}

In this subsection, we present critical metrics of the action \eqref{Eq-Smix} in a subclass of metric $f$-contact manifolds.
Assume that $\xi_i\ (i\le s)$ are given linearly independent vector fields on a smooth manifold $M^{2n+s}$, and that one-forms
$\eta_i\ (i\le s)$ satisfy $\eta^i(\xi_j)=\delta^i_j$.
If $f$ is a (1,1)-tensor on $M$ satisfying
\begin{equation*}
 f^2 = -{\rm id} +\sum\nolimits_{\,i}\eta^i\otimes\xi_i,
\end{equation*}
then we get a framed $f$-structure, where $f$ satisfies $f^3 + f = 0$ and has constant rank $2n$.
In~this case, $TM$ splits into two complementary subbundles: a $2n$-dimensional image $f(TM)$ and an $s$-dimensional kernel $\ker f$,
moreover, the restriction of $f$ to $f(TM)$ determines a complex structure.
 For a framed $f$-structure the equalities ${f}\,\xi_i=0$ and $\eta^i\circ{f}=0$ hold for $1\le i\le s$.
 If there exists a \textit{compatible Riemannian metric} $g$, i.e., $g(\xi_i,\xi_j)=\delta_{ij}$ and
\begin{align}\label{2.2}
 g({f} X,{f} Y)= g(X, Y) -\sum\nolimits_{\,i} \eta^i(X)\,\eta^i(Y),\quad X,Y\in\mathfrak{X}_M,
\end{align}
then
$M({f},\xi_i,\eta^i,g)$ is called a \textit{metric $f$-manifold}.
 Putting $Y=\xi_i$ in \eqref{2.2} gives
 $g(X,\xi_i) = \eta^i(X)$,
thus, $f(TM)\,\bot\,\ker f$ and $\{\xi_i\}$ is an orthonormal basis of $\ker f$.
For a metric $f$-structure, the tensor ${f}$ is skew-symmetric.
The \textit{fundamental $2$-form} $\Phi$ on $M({f},\xi_i,\eta^i,g)$ is given by
 $\Phi(X,Y)=g(X,{f} Y)$ for $X,Y\in\mathfrak{X}_M$.
A metric $f$-manifold is called \emph{metric $f$-contact manifold} if $d\eta^i = \Phi\ (i\le s)$; a metric $f$-contact manifold is called \emph{metric $f$-$K$-contact manifold} if all $\xi_i$ are Killing vector fields, see~\cite{gy,Goertsches-2}.
Then $\ker f$ determines a totally geodesic Riemannian foliation and we have
\begin{equation}\label{E-nabla-f}
 {\cal L}_{\xi_i}\,f=0,\quad
 \nabla_{\xi_i}\,f=0,\quad
 \nabla_{\xi_i}\,\xi_j=0,
\end{equation}
where ${\cal L}$ is the Lie derivative.
If $s=1$, then a metric $f$-$K$-contact manifold is called \emph{$K$-contact manifold}.
We show that for an appropriate choice of distributions on metric $f$-$K$-contact manifold
the metric is a critical point of the action \eqref{Eq-Smix}.

\begin{proposition}\label{propfKcontact1}
Let $M({f},\xi_i,\eta^i,g)$ be a metric $f$-$K$-contact manifold,
$\mD_1$ be spanned by vector fields $\xi_1, \ldots, \xi_m$ for some $m<s$ and $\mD_2$ be spanned by $\xi_{m+1}, \ldots, \xi_s$.
Then $g$ is critical for the action \eqref{Eq-Smix} with respect to all variations preserving orthogonality of \,$\mD_1$ and $\mD_2$.
\end{proposition}

\begin{proof}
All distributions $\mD_1,\mD_2$ and $\mD_3=f(TM)$ are pairwise mixed totally geodesic and pairwise mixed integrable, all are totally geodesic and only $\mD_3$ is non-integrable, with $\Ts_{\xi_{\alpha}} = f\ (\alpha = 1, \ldots ,s)$, see~\cite{blairfKcontact}.
The Euler-Lagrange equation \eqref{EL11withmixed}, for volume-preserving variations, becomes:
\begin{equation} \label{EL11fKcontact}
 S_{\,\mD_1, \mD_2}\cdot g(\xi_\alpha ,\xi_\beta ) = \lambda\,g(\xi_\alpha ,\xi_\beta ),\quad
 \alpha, \beta \in \{ 1, \ldots, m \}.
\end{equation}
Its dual equation (with respect to interchanging $\mD_1$ and $\mD_2$) has the same form as above, only with $\alpha, \beta \in \{ m+1, \ldots, s \}$.
Also, the Euler-Lagrange equation \eqref{EL33withmixed} becomes
\begin{equation} \label{EL33fKcontact}
 (S_{\,\mD_1,\mD_2} -\lambda)\, g(X,Y) = 0,\quad X,Y \in \mD_3.
\end{equation}
The Euler-Lagrange equation \eqref{EL13withmixed} becomes
\begin{eqnarray*}
&&- (\Div \tT_{1,X}) Y
 +\sum\nolimits_{a,\mu} g(Y ,\T_{ 3 , E_a} e_\mu )\,g( X , \T_{1 , e_\mu} E_a  )
 +\sum\nolimits_{i,j} g(X, T_2 ( {\cal E}_i , {\cal E}_j ) )\,g(Y, T_2 ( {\cal E}_i , {\cal E}_j ) ) \nonumber\\
&& +\,(\Div \tT_{12,X} ) Y - \sum\nolimits_{a,\mu} g( \T_{3 , E_a } e_\mu , Y )\,g( \T_{1 , e_\mu } E_a , X ) = 0,
\quad X \in \mD_1,\ Y \in \mD_3,
\end{eqnarray*}
which, as $\mD_1$ and $\mD_2$ are integrable, can be simplified to the following form:
\begin{equation} \label{EL13fKcontact}
 (\Div \tT_{1,X}) Y + (\Div \tT_{12,X} ) Y = 0,\quad X \in \mD_1,\ Y \in \mD_3.
\end{equation}
We have
\begin{eqnarray*}
 && (\Div \tT_{1, \xi_\alpha}) Y = \sum\nolimits_\mu g( \nabla_{e_\mu} \tT_{1,\xi_\alpha} Y , e_\mu )
- \sum\nolimits_\mu g( \tT_{1,\xi_\alpha} ( \nabla_{e_\mu} Y ) , e_\mu ) \\
\eq \sum\nolimits_\mu g( \nabla_{e_\mu} \tT_{1,\xi_\alpha} Y, e_\mu) +\sum\nolimits_\mu g(\tT_{1,\xi_\alpha} e_\mu, \nabla_{e_\mu} Y) \\
\eq
 \sum\nolimits_\mu g( \nabla_{e_\mu} \tT_{1,\xi_\alpha} Y , e_\mu ) + \sum\nolimits_\mu g(\tT_{1,\xi_\alpha} e_\mu, T_3( {e_\mu}, Y ) ) \\
\eq  \sum\nolimits_\mu g( \nabla_{e_\mu} \tT_{1,\xi_\alpha} Y , e_\mu )
+ \sum\nolimits_{\mu,\beta}  g( \tT_{1,\xi_\alpha}  e_\mu , \xi_\beta )\,g(  \Ts_{3 ,\xi_\beta} {e_\mu} , Y )
 =\sum\nolimits_\mu g( \nabla_{e_\mu} \tT_{1,\xi_\alpha} Y , e_\mu ) ,
\end{eqnarray*}
as $g( \tT_{1,\xi_\alpha}  e_\mu , \xi_\beta ) = g( \mTs_{23,\xi_\alpha}  e_\mu , \xi_\beta ) =0$ by mixed integrability of $\mD_2$ and $\mD_3$.
On the other hand, we can assume that $\nabla_{e_\mu} Y\perp\mD_3$ at a point, and hence $\nabla_{e_\mu} Y$ belongs to $\mD_1 \oplus \mD_2$, on which $\tT_{12,\xi_\alpha}$ vanishes;~then
\begin{eqnarray*}
 (\Div \tT_{12, \xi_\alpha}) Y \eq \sum\nolimits_\mu g( \nabla_{e_\mu} \tT_{12,\xi_\alpha} Y , e_\mu )
 - \sum\nolimits_\mu g( \tT_{12,\xi_\alpha} ( \nabla_{e_\mu} Y ) , e_\mu )
 =\sum\nolimits_\mu g( \nabla_{e_\mu} \tT_{12,\xi_\alpha} Y , e_\mu ) .
\end{eqnarray*}
Also, $ g( \nabla_{e_\mu} \tT_{12,\xi_\alpha} Y , e_\mu ) =  g( \nabla_{e_\mu} \tT_{1,\xi_\alpha} Y , e_\mu )$
as $Y,e_\mu \in \mD_3$, and it follows that
\begin{eqnarray*}
 -(\Div \tT_{1,\xi_\alpha}) Y + (\Div \tT_{12,\xi_\alpha} ) Y
 = \sum\nolimits_\mu g( \nabla_{e_\mu} \tT_{1,\xi_\alpha} Y , e_\mu )
 -\sum\nolimits_\mu g( \nabla_{e_\mu} \tT_{1,\xi_\alpha} Y , e_\mu )
  =0.
\end{eqnarray*}
Similar reasoning can be applied to the equation dual to \eqref{EL13fKcontact} with respect to the interchanging $\mD_1$ and $\mD_2$.
Finally, using \eqref{E-nabla-f}$_3$ in 
$R(\xi_\alpha,\xi_\beta)\xi_\beta$,
we get $S_{\,\mD_1,\mD_2} =0$; hence, all Euler-Lagrange equations: \eqref{EL11fKcontact}, its dual, \eqref{EL33fKcontact}, \eqref{EL13fKcontact} and its dual, hold with $\lambda = 0$.
\end{proof}


The following Lemma 
is related to 
assumptions about mixed integrability, considered below. 
\begin{lemma}
Let $\mD_3$ be tangent to a Riemannian foliation on $(M,g)$, let $\mD$ be orthogonal complement of $\mD_3$, and let $\mD_1$ and $\mD_2$ be any orthogonal distributions such that $\mD_1 \oplus \mD_2 = \mD$. Then $\mT_{13}=0$ if and only if $\mT_{23}=0$.
\end{lemma}

\begin{proof}
Suppose that $\mT_{13}=0$, let $\xi \in \mD_3$, $X \in \mD_1$ and $Y \in \mD_2$. Then
\[
0 = (\mathcal{L}_\xi g) (X,Y) = \xi (g(X,Y)) - g( [\xi, X ] , Y) - g([\xi,Y],X) = -g(\mT_{13}(\xi,X), Y) - g(\mT_{23}(\xi,Y),X),
\]
that completes the proof.
\end{proof}

\begin{proposition} \label{propH1}
Let $M^{2n+1}({f},\xi,\eta,g)$ be a $K$-contact manifold with a global $f$-basis $e_1 , \ldots , e_{2n}$ of $f(TM)$,
i.e., $f e_i = e_{n+i}\ (1 \leq i \le n)$.
Suppose that $\mD_1$ is spanned by $e_1 , \ldots , e_{n}$, $\mD_2$ is spanned by $e_{n+1} , \ldots , e_{2n}$, $\mD_3$ is spanned by $\xi$,
$\mD_1,\mD_2$ are integrable and totally geodesic, and $\mT_{13}=\mT_{23}=0$.
Then $g$ is a critical point of the action \eqref{Eq-Smix}
with respect to volume-preserving variations of metric,
keeping orthogonality of $\mD_1$ and $\mD_2$,
and such that $\dt g_t |_{\,t=0}(\xi,\xi)=0$.
\end{proposition}

\begin{proof}
All $\mD_1,\mD_2,\mD_3$ are integrable and totally geodesic. 
We have
$\mh_{i3}(X,\xi) = - \mTs_{12,\xi} X$ for $X \in \mD_i\ (i=1,2)$, as, e.g., using 
$\mT_{13}=0$, we have
\begin{equation} \label{eqmhmT}
g( \mh_{13}(X ,\xi) , Y) = \frac{1}{2}\,g( \nabla_X \xi + \nabla_\xi X , Y) =  g( \nabla_X \xi  , Y) = -g(\mTs_{12,\xi} X,Y),
\quad X \in \mD_1,\ Y \in \mD_2.
\end{equation}
Also,
$\mT_{12}( e_i , e_j ) = \delta_{j , n+i}\,\xi\ (i<j)$.
For variations preserving orthogonality of distributions, we get the following Euler-Lagrange equations.
For $B(X,Y) = B(P_1 X, P_1 Y)$, the Euler-Lagrange equation \eqref{EL11withmixed} becomes
\begin{eqnarray*}
 6\sum\nolimits_{j} g( \mTs_{12,\xi} {\cal E}_j , X)\,g( \mTs_{12,\xi} {\cal E}_j , Y) = (\lambda - S_{\,\mD_1, \mD_2})\,g(X,Y),
\end{eqnarray*}
for all $X,Y \in \mD_1$, i.e.,
$ 6\,g( \tT_\xi X , \tT_\xi Y  ) = (\lambda - S_{\,\mD_1, \mD_2})\,g(X,Y)$,
which can be written as
\begin{equation}\label{E-S12}
(S_{\,\mD_1, \mD_2} - \lambda + 6)\,g(X,Y) =0 , \quad (X,Y \in \mD_1).
\end{equation}
For $B(X,Y) = B(P_2 X, P_2 Y)$ we have the dual equation, i.e., \eqref{E-S12} for $X,Y \in\mD_2$.
In Euler-Lagrange equation \eqref{EL13withmixed} (evaluated on $X \in \mD_1$ and $\xi$) all terms vanish except
\[
 - \Div {\tilde h}_2 (X,\xi) = (\Div \mTs_{12,\xi})(X) = (\Div \tT_{3,\xi})(X),
\]
which vanishes for contact metric structures, see the proof of \cite[Proposition 8]{rz-2}. Thus, the Euler-Lagrange equation for $B(X,Y) = B(P_1 X ,P_3 Y) + B(P_3 X , P_1 Y)$ is satisfied, as well as its dual for $B(X,Y) = B(P_2 X ,P_3 Y) + B(P_3 X , P_2 Y)$. Hence, all Euler-Lagrange equations except \eqref{EL33withmixed} hold, which proves the claim.

We note that for $B(X,Y) = B(P_3 X, P_3 Y)$ we get the following (scalar, as $\mD_3$ is 1-dimensional) Euler-Lagrange equation~\eqref{EL33withmixed}:
\begin{eqnarray*} 
 - 2 \sum\nolimits_{\,a,i} g(\xi, \mT_{12} ( E_a , {\cal E}_i ) )\, g(\xi, \mT_{12} (E_a , {\cal E}_i)) = \lambda - S_{\,\mD_1,\mD_2}\, ,
\end{eqnarray*}
from where it follows that
$S_{\,\mD_1,\mD_2} = 2n + \lambda$,
which cannot hold together with \eqref{E-S12} - hence, $g$ is not critical with respect to all volume-preserving variations keeping orthogonality of $\mD_1$ and $\mD_2$.
\end{proof}

\begin{proposition} \label{propH2}
Let $M^{2n+1}({f},\xi,\eta,g)$ be a $K$-contact manifold. 
Let $\mD_1$ and $\mD_2$ be distributions orthogonal and complementary in $f(TM)$, each of them invariant under $f$  
 and let $\mD_3=\ker f$
be spanned by the Reeb field $\xi$.
Suppose that $\mT_{13}=\mT_{23}=0$
and $P_2 T_1 = P_1 T_2 = P_1 h_2 = P_2 h_1 =0$.
Then $g$ is a critical point of the action \eqref{Eq-Smix}
with respect to volume-preserving variations of metric keeping the orthogonality of $\mD_1$ and $\mD_2$.
\end{proposition}

\begin{proof}
All $\mD_1,\mD_2, \mD_3$ 
are totally geodesic, pairwise mixed totally geodesic and pairwise mixed integrable, but $\mD_1$ and $\mD_2$ are not integrable.
%
The Euler-Lagrange equation \eqref{EL11withmixed} reduces to
$S_{\,\mD_1, \mD_2}\cdot g(X,Y) = \lambda\, g(X,Y)$
for $X,Y \in \mD_1$.
The~dual equation has the same form for $X,Y \in \mD_2$.
The Euler-Lagrange equation \eqref{EL33withmixed} on $\mD_3$ becomes
$S_{\,\mD_1,\mD_2} = \lambda$, which is satisfied, and equivalent to the Euler-Lagrange equations considered above.
In \eqref{EL13withmixed} all terms vanish, so it also holds, as well as its dual. 
\end{proof}

\begin{remark} \rm
The 
distributions $\mD_1$ and $\mD_2$ satisfying all assumptions of Propositions~\ref{propH1} and \ref{propH2} 
exist,
 for example, on
Heisenberg groups with left-invariant $K$-contact structures, see~\cite{blair}.
\end{remark}

\subsection{A modified variational problem}
\label{sec:3.5}

Formulas similar to \eqref{SD1D2} can be used also in other variational problems. For example, consider two
distributions $\mD_1$ and $\mD$ on $M$ such that $\mD_1 \subset \mD$.
Let $g_t$ be a family of metrics on $M$ that keep $\mD^\perp$ fixed, let $\mD_2 = \mD^\perp$. 
Then, both distributions $\mD_1$ and $\mD_2$ are fixed and orthogonal for all $t$.
Let $\mD_3(t)$ be the $g_t$-orthogonal complement of $\mD_1$ within $\mD$, i.e., $\mD_3(t) = \mD_1^\perp(t) \cap\,\mD$.
Considering orthonormal frames: $\{ E_a \}$ of $\mD_1$, $\{ e_\mu(t) \}$ of $\mD_3(t)$ 
and $\{ {\cal E}_i \}$ of $\mD_2$, we obtain for the mutual curvature of distributions $\mD_1$ and $\mD_1^\perp \cap\,\mD$, analogously as \eqref{SD1D2}:
\begin{eqnarray}\label{SD1inD}
 2\,S_{\,\mD_1 , \mD_1^\perp  \cap\,\mD} \eq S_{\,\mD_1 , \mD_1^\perp}
 - S_{\,\mD_2 , \mD_2^\perp} + S_{\,\mD_1 \oplus \mD_2 ,\,(\mD_1 \oplus \mD_2)^\perp } \,.
\end{eqnarray}
Assuming that $\mD^\perp$ is preserved as the metric changes, formula \eqref{SD1inD} is true for all $g_t$, and every term on its right-hand side is the mixed scalar curvature of a fixed distribution and its varying orthogonal complement. We can use \eqref{SD1inD} to derive variational formulas at $t=0$,
for the modified action \eqref{Eq-Smix}:
\begin{equation}\label{Eq-Smix-2}
 J_{\mD_1 \subset \mD}: g \mapsto \int_M \,S_{\,\mD_1,\,\mD_1^\perp \cap\,\mD}\, {\rm d} \vol_g,
\end{equation}
as long as $B(X_1,X_2) = 0 = B(X_2,X_3)$ for all $X_i \in \mathfrak{X}_{\mD_i} \ (i=1,2,3)$.
By \eqref{SD1inD} and  \eqref{E-varJh-init2-1}--
\eqref{E-varJh-init2-3}, we get, similarly as in \eqref{dtS1S2general} and with the same notation as in Proposition \ref{propdtS1S2general}, the following:
\begin{equation*}
 2\,{\rm\frac{d}{dt}} J_{\mD_1 \subset \mD}(g_t)|_{\,t=0} =  I_1 + I_2 -I_3 -I_4 -I_5 -I_6.
\end{equation*}

\begin{proposition}
Let $M^{2n+s}({f},\xi_i,\eta^i,g)$ be a metric $f$-$K$-contact manifold,
let $\mD_1$ be spanned by $\xi_1 , \ldots , \xi_m$ for some $m<s$, and let $\mD=\ker f$.
Then $g$ is critical for the action \eqref{Eq-Smix-2}
with respect to all variations that preserve orthogonality of $f(TM)$ and $\mD$.
\end{proposition}

\begin{proof}
Put $\mD_2= f(TM)$. Let $g_t$ be a family of metrics that keep $\mD^\perp$ fixed.
By \eqref{SD1inD}, from \eqref{E-varJh-init2-1}--\eqref{E-varJh-init2-3} we obtain
\begin{eqnarray}\label{SD1inDint}
\nonumber
&& 2\,\frac{d}{dt}\,J_{\mD_1 \subset \mD}(g_t)|_{\,t=0} = \int_M \big\<
 2\,\widetilde{\cal T}_1^\flat - 2 ( \Div {\tilde \theta}_1 )_{\,|\,(\mD_1 \times \mD_1^\perp) \cup (\mD_1^\perp \times \mD_1) } \\
&& \nonumber
 +\,\frac{1}{2}\,S_{\,\mD_1 , \mD_1^\perp}\cdot g_1
 +\frac{1}{2}\,\Upsilon_{ {\tilde T}_1 , {\tilde T}_1 }
 - 2\,( \Div {\tilde \theta}_{12} )_{\,|\,(\mD_{12} \times \mD_{12}^\perp )\cup (\mD_{12}^\perp \times \mD_{12}) }
 + \frac{1}{2} \Upsilon_{ T_{12} , T_{12}}  \\
&& +\frac{1}{2}\,S_{\,\mD_{12} , \mD_{12}^\perp}\cdot g_{12}
 +2\,{\cal T}_{12}^\flat
 - \frac{1}{2} \Upsilon_{T_2 , T_2}
 - 2\,{\cal T}_{2}^\flat
 - \frac{1}{2}\,S_{\,\mD_{2} , \mD_{2}^\perp}\cdot g_{2} , B \big\> \, {\rm d} \vol_g.
\end{eqnarray}
Let $\mD_3 = \mD_1^\perp(0) \cap\,\mD$.
It follows from \eqref{SD1inDint} 
that the Euler-Lagrange equations
for $B(X,Y) = B(P_1 X, P_1 Y)$,
$B(X,Y) = B(P_2 X, P_2 Y)$
and $B(X,Y) = B(P_3 X, P_3 Y)$, are all of the same form:
\begin{eqnarray*}
 S_{\,\mD_1 , \mD_1^\perp \cap\,\mD} \cdot g_{\,|\, \mD_i \times \mD_i} \eq \lambda\, g_{\, |\, \mD_i \times \mD_i}  ,\quad i=1,2,3.
\end{eqnarray*}
For $B(X,Y) = B(P_1 X, P_3 Y) + B(P_3 X, P_1 Y)$ we obtain the Euler-Lagrange equation
\begin{equation*}
( \Div {\tilde \theta}_1 + \Div {\tilde \theta}_{12} )_{\,|\,(\mD_1 \times \mD_3 ) \cup ( \mD_3 \times \mD_1) }=0.
\end{equation*}
We have ${\tilde \theta}_{12}=0$ and from \cite[the proof of Proposition 8]{rz-2}, we get $\Div {\tilde \theta}_{1} =0$.
Hence, all terms in this Euler-Lagrange equation 
vanish, and so it is satisfied. Note that due to assumptions $B(P_2 X, P_3 Y)=0=B(P_1 X, P_2 Y)$ for all $X,Y \in \mathfrak{X}_M$, there are no more Euler-Lagrange equations to consider.
As in the proof of Proposition~\ref{propfKcontact1}, 
we get $S_{\,\mD_1,\mD_1^\perp \cap\,\mD} =0$; thus, all Euler-Lagrange equations hold with~$\lambda=0$.
\end{proof}

The following result shows that variational problem 
\eqref{Eq-Smix-2}
can lead to different Euler-Lagrange equations than action \eqref{Eq-Smix}.


\begin{proposition}\label{propH1SD1inD}
Let $(M,g)$ be a Riemannian manifold with a unit Killing vector field $\xi$.
Let $\mD$ be the distribution orthogonal to $\xi$, and $\mD_1$ and $\mD_3$ be orthogonal complementary distributions in~$\mD$. Let $\mD_2$ be spanned by $\xi$.
Suppose that $\mT_{12} = \mT_{23} =0$, and $\mD_1,\mD_3$ are integrable and totally geodesic.
Then $g$ is critical for the action \eqref{Eq-Smix-2} with respect to volume-preserving variations that preserve the orthogonal complement
of $\mD$ if and only if the following Euler-Lagrange equations hold:
\begin{eqnarray} \label{EL11forD1inD}
 6\,g((\mTs_{13,\xi})^2 X, Y) \eq (S_{\,\mD_1, \mD_1^\perp \cap\,\mD} -\lambda)\,g(X,Y),
 \quad (X, Y \in \mD_i , \quad i=1,3 ) , \\ 
 \label{EL22forD1inD}
 3\,\| \mT_{13} \|^2 \eq S_{\,\mD_1 , \mD_1^\perp \cap\,\mD} - \lambda, \\
 \label{EL13forD1inD}
g ( (\nabla_\xi \mTs_{13,\xi}) X,Y) \eq 0 , \quad ( X \in \mD_1 ,\ Y \in \mD_3).
\end{eqnarray}
\end{proposition}

\begin{proof}
Since $\xi$ is a unit Killing vector field, we have $h_2 = {\tilde h_2} = T_2 =0$, and from assumptions about $\mD_1,\mD_3$ we have $T_1 = T_3 = h_1 = h_3 =0$. 
From $T_3 =0 $ and $\mT_{23}= 0$ it follows that ${\tilde T}_1=0$, analogously we obtain ${\tilde T}_3 =0$.
Similarly as in \eqref{eqmhmT}, we obtain
\[
\mh_{12}(X,\xi) = - \mTs_{13,\xi} X , \quad \mh_{23}(Y,\xi) = - \mTs_{13,\xi} Y  \quad (X \in \mD_1 , Y \in \mD_3).
\]
By \eqref{SD1inD}, from \eqref{E-varJh-init2-1}--\eqref{E-varJh-init2-3} we obtain the following variational formulas.

For variations satisfying $B(X,Y) = B(P_1 X, P_1 Y)$ we have
\begin{equation*}
 2\,\frac{d}{dt}\,J_{\mD_1 \subset \mD}(g_t)|_{\,t=0}
= \int_M \big\< \frac{1}{2} \Upsilon_{ \th_1 , \th_1 } + S_{\,\mD_1 , \mD_1^\perp \cap\,\mD}\cdot g^\top_1 - ( \Div h_{12} )_{\,| \mD_{12} \times \mD_{12}} - 2\,\widetilde{\cal T}^\flat_2  , B \big\>\, {\rm d} \vol_g .
\end{equation*}
For $X,Y \in \mD_1$ we have
\begin{align*}
 \Div h_{12}(X,Y) =-\sum\nolimits_\mu g(h_{12}(\nabla_{e_\mu} X, Y), e_\mu) -\sum\nolimits_\mu g( h_{12}(\nabla_{e_\mu} Y, X), e_\mu) \\
\hskip-0mm= -\sum\nolimits_\mu g( h_{12} (\xi, Y), e_\mu )\,g( T_{13}(e_\mu, X), \xi)
-\sum\nolimits_\mu g(h_{12}(\xi, X), e_\mu )\,g( T_{13}(e_\mu, Y), \xi) \\
\hskip-0mm= -\,2\,g(\mTs_{13,\xi} X ,\mTs_{13,\xi} Y), \\
 \frac{1}{2} \Upsilon_{ \th_1 , \th_1 } (X,Y) = \sum\nolimits_\mu g( X , \th_1 (\xi , e_\mu) )\,g( Y, \th_1 (\xi , e_\mu))
+ \sum\nolimits_\mu g( X , \th_1 (e_\mu , \xi) )\,g( Y, \th_1 ( e_\mu , \xi)) \\
= 2\sum\nolimits_\mu g( X, \mTs_{13,\xi} e_\mu )\,g( Y, \mTs_{13,\xi} e_\mu) = 2\,g(\mTs_{13,\xi} X ,\mTs_{13,\xi} Y),\\
  2\widetilde{\cal T}^\flat_2 (X,Y) = 2\,g(\tT_{2,\xi} \tT_{2,\xi} X,Y)
= -2\sum\nolimits_\mu g(\mTs_{2,\xi} X , e_\mu)\,g(\mTs_{2,\xi} Y , e_\mu) = -2\,g(\mTs_{13,\xi} X ,\mTs_{13,\xi} Y).
\end{align*}
Hence, the first Euler-Lagrange equation is \eqref{EL11forD1inD} for $X,Y \in \mD_1$. 

For variations satisfying $B(X,Y) = B(P_3 X, P_3 Y)$, we have
\begin{eqnarray*}
 2\,\frac{d}{dt}\,J_{\mD_1 \subset \mD}(g_t)|_{\,t=0}
= \int_M \big\< - ( \Div {\tilde h}_{1} )_{\,| \mD_{1}^\perp \times \mD_{1}^\perp } - 2\,\widetilde{\cal T}^\flat_2 + \frac{1}{2} \Upsilon_{ h_{12} , h_{12} } + S_{\,\mD_1 , \mD_1^\perp \cap\,\mD}\cdot g^\top_3 , B \big\>\, {\rm d} \vol_g .
\end{eqnarray*}
For $X,Y \in \mD_3$, we have
\begin{eqnarray*}
&&\Div {\tilde h}_{1} (X,Y) = -\sum\nolimits_a g( {\tilde h}_1 (\nabla_{E_a} X,Y) , E_a  ) -\sum\nolimits_a g( {\tilde h}_1 (\nabla_{E_a} Y,X) , E_a  ) \\
&=& -\sum\nolimits_a g( h_{23} ( \xi , Y ) , E_a  ) g(T_{13}(E_a,X) , \xi )
-\sum\nolimits_a g( h_{23} ( \xi , X ) , E_a  ) g(T_{13}(E_a,Y) ,\xi ) \\
&=& -2\,\sum\nolimits_a g( \mTs_{13,\xi} X , E_a  )g( \mTs_{13,\xi} Y , E_a  )
=-2\,g(\mTs_{13,\xi} X ,\mTs_{13,\xi} Y) , \\
&&\frac{1}{2} \Upsilon_{ h_{12} , h_{12} } (X,Y) = 2 \sum\nolimits_a g( X , h_{12} (\xi , E_a) )\,g( Y,  h_{12} (\xi , E_a) ) \\
&=& 2\sum\nolimits_a g( X, \mTs_{13,\xi} E_a )\,g( Y, \mTs_{13,\xi} E_a)
= 2\,g(\mTs_{13,\xi} X ,\mTs_{13,\xi} Y) , \\
&& 2\widetilde{\cal T}^\flat_2 (X,Y) = 2\,g(\tT_{2,\xi} \tT_{2,\xi} X,Y)
= -2\sum\nolimits_a g(\mTs_{2,\xi} X , E_a)\,g(\mTs_{2,\xi} Y , E_a) \\
&=& -2\,g(\mTs_{13,\xi} X ,\mTs_{13,\xi} Y).
\end{eqnarray*}
Hence, we obtain as the Euler-Lagrange equation \eqref{EL11forD1inD}, but for $X,Y \in \mD_3$.

For variations with $B(X,Y) = B(P_2 X, P_2 Y)$ we have
\begin{equation*}
 2\,\frac{d}{dt}\,J_{\mD_1 \subset \mD}(g_t)|_{\,t=0}
= \int_M \< S_{\,\mD_1 , \mD_1^\perp \cap\,\mD}\cdot g^\top_2 - ( \Div \th_1 )_{\,| \mD_{1}^\perp \times \mD_{1}^\perp }
   {-} ( \Div h_{12} )_{\,| \mD_{12} \times \mD_{12}}
   {-} \frac{1}{2} \Upsilon_{ {\tilde T}_2 , {\tilde T}_2 } , B \>\, {\rm d} \vol_g .
\end{equation*}
We have
\begin{eqnarray*}
 (\Div \th_1 )(\xi,\xi) \eq -2 \sum\nolimits_a g( \th(\nabla_{E_a} \xi,\xi) ,E_a) = 2\sum\nolimits_a g(h_{13} (\tT_{2,\xi} E_a, \xi), E_a) \\
 \eq -2\sum\nolimits_a g((\mTs_{13,\xi})^2 E_a, E_a) = \| \mT_{13} \|^2 ,\\
 (\Div h_{12} )(\xi,\xi) \eq -2 \sum\nolimits_\mu g( h_{12} ( \nabla_{e_\mu} \xi, \xi) , e_\mu) \\
 \eq -2 \sum\nolimits_\mu g(\mTs_{13,\xi} E_a , e_\mu)\,g( T_{13} (e_\mu, E_a) , \xi) = \| \mT_{13} \|^2, \\
 \Upsilon_{ {\tilde T}_2 , {\tilde T}_2 } (\xi,\xi) \eq 2 \sum\nolimits_{a,\mu} g(\xi , T_{13} (E_a , e_\mu) )\,g(\xi , T_{13} (E_a , e_\mu) ) = \| \mTs_{13} \|^2.
\end{eqnarray*}
Hence, the Euler-Lagrange equation for variations with $B(X,Y) = B(P_2 X, P_2 Y)$ is \eqref{EL22forD1inD}.

For variations with $B(X,Y) = B(P_1 X, P_3 Y) + B(P_3 X, P_1 Y)$ we have
\begin{eqnarray*}
&& 2\,\frac{d}{dt}\,J_{\mD_1 \subset \mD}(g_t)|_{\,t=0}
 = \int_M \< 2\,( \Div \alpha_1 )_{\,| (\mD_1 \times \mD_1^\perp ) \cup ( \mD_1^\perp \times \mD_1 ) } \\
&&  +\,2\,( \Div ( \alpha_{12} - {\tilde \theta}_{12} ) )_{\,| ( \mD_1 \times \mD_1^\perp ) \cup ( \mD_1^\perp \times \mD_1 ) } - 2\Upsilon_{ {\tilde \theta}_{12} , \alpha_{12} } - 2\,\widetilde{\cal T}^\flat_2  , B \>\,{\rm d} \vol_g .
\end{eqnarray*}
But $\alpha_1 =0 = {\tilde \theta}_{12}$, so their divergences also vanish.
For $X \in \mD_1$ and $Y \in \mD_3$ we have by $h_1=0$
\[
\frac{1}{2} \Upsilon_{ {\tilde \theta}_{12} , \alpha_{12} }(X,Y) = \sum\nolimits_{a,\mu} g(Y , \tT_{12,a} e_\mu)\,g(X, A_{12,\mu} E_a) =0,
\]
and by $T_3=0$ we get
\[
\widetilde{\cal T}^\flat_2 (X,Y) = - g( \tT_{2,\xi} X, \tT_{2,\xi} Y ) = -\sum\nolimits_\mu g(\tT_{2,\xi} X, e_\mu)\,g(e_\mu, \tT_{2,\xi} Y).
\]
For $X \in \mD_1$ and $Y \in \mD_3$ such that $\nabla_Z X \in \mD_1^\perp$ and $\nabla_Z Y \in \mD_3^\perp$ for all $Z \in T_xM$ at the point $x$ where the following formula is computed, we have, using $h_1=0=T_1$:
\begin{eqnarray*}
&& ( \Div \alpha_{12} )(X,Y) = \xi ( g( A_{12,Y} X ,\xi) ) - \sum\nolimits_\mu g(A_{12,\mu} X , \xi)\,g(\nabla_{\xi} Y , e_\mu) \\
&& - \sum\nolimits_a g(A_{12, Y} \xi , E_a)\,g( \nabla_{E_a} X, \xi) - \sum\nolimits_a g( A_{12,Y} E_a , \xi)\,g(\nabla_\xi X, E_a) \\
&& = \xi ( g( A_{12,Y} X ,\xi) ) = \xi (g(h_{12} (X,\xi), Y)) = -\xi ( g( \mTs_{13, \xi} X, Y)) = -g ( (\nabla_\xi \mTs_{13,\xi}) X,Y).
\end{eqnarray*}
Hence, the Euler-Lagrange equation for variations with $B(X,Y) = B(P_1 X, P_3 Y) + B(P_3 X, P_1 Y)$ is \eqref{EL13forD1inD}.
\end{proof}

\begin{remark}\rm
Taking trace of \eqref{EL11forD1inD}, we get $6\,\| \mT_{13} \|^2 = -(S_{\,\mD_1 , \mD_1^\perp \cap\,\mD} - \lambda)(n_1+n_3)$.
Comparing this with \eqref{EL22forD1inD} and using $n_1+n_3>0$, gives $\mT_{13}=0$ and $S_{\,\mD_1 , \mD_1^\perp \cap\,\mD} = \lambda$.
\end{remark}

From Proposition~\ref{propH1SD1inD} we obtain the following analogue of Proposition \ref{propH1}, for the modified variational problem.

\begin{corollary} \label{corKcontactSD1inD}
Let $M^{2n+1}(f, \xi, \eta, g)$ be a $K$-contact manifold with 
a global $f$-basis
$e_1 , \ldots , e_{2n}$, 
such that 
$f e_i = e_{n+i}\ (1 \leq i \le n)$.
Let $\mD_1$ be spanned by $e_1 , \ldots , e_{n}$, 
$\mD_2$ be spanned by $\xi$ and let $\mD$ be the distribution orthogonal to $\xi$.
Suppose that distributions: $\mD_1$ and $\mD_3 = \mD_1^\perp \cap \mD$ are both integrable and totally geodesic and $\mT_{12} = \mT_{23}=0$. 
Then $g$ is a critical point of the action \eqref{Eq-Smix-2} 
with respect to volume-preserving variations that preserve the orthogonal complement of $\mD$ and the length of $\xi$ if and only if $(\nabla_\xi \mTs_{13,\xi})=0$.
\end{corollary}

\begin{proof}
Equations \eqref{EL11forD1inD}
hold, but do not hold together with \eqref{EL22forD1inD} for the same constant $\lambda$. Hence, $g$ is critical with respect to variations such that $\dt g_t(\xi,\xi) |_{t=0} =0$ (e.g., keeping $\xi$ a unit vector field) if and only if \eqref{EL13forD1inD} holds.
\end{proof}

\begin{remark} \rm
The distributions considered in Proposition~\ref{propH1SD1inD} and Corollary~\ref{corKcontactSD1inD} exist
on the real Heisenberg group
$H_{2n+1}$ with a left-invariant $K$-contact metric, and also \eqref{EL13forD1inD} is satisfied there.
\end{remark}

\textbf{Funding}.
The second author was supported from Narodowe Centrum Nauki grant Miniatura 2021/05/X/ST1/00359.

\end{document}